\documentclass[11pt]{article}
\usepackage[left=1in,top=0.8in,right=1in,bottom=0.8in,head=0in,nofoot]{geometry}

\usepackage{setspace,url,bm,amsmath} 

\usepackage{titlesec} 
\titlelabel{\thetitle.\quad} 
\titleformat*{\section}{\bf\Large\center}

\usepackage{graphicx} 
\usepackage{bbm}
\usepackage{latexsym}
\usepackage{caption}
\usepackage[margin=20pt]{subcaption}
\usepackage{enumerate}

\usepackage{arydshln}

\usepackage{tikz}

\usepackage{tikzsymbols}

\usepackage{wasysym}
\def\smiley{$\checkmark$}
\def\frownie{$\times$}
\def\neutranie{$\ocircle$}
\def\ssize{\Large}

\newcommand{\GG}[1]{}

\usepackage{amsthm}
\usepackage{amssymb}
\usepackage{amsmath}
\usepackage{color}

\usepackage{comment}
\theoremstyle{definition}

\newtheorem*{theorem*}{Theorem}
\newtheorem{theorem}{Theorem}
\newtheorem*{rmk*}{Remark}
\newtheorem{proposition}{Proposition}
\newtheorem{lemma}{Lemma}
\newtheorem{example}{Example}

\newtheorem{condition}{Condition}

\newtheorem{definition}{Definition}
\newtheorem{rmk}{Remark}
\newtheorem{corollary}{Corollary}
\newtheorem*{corollary*}{Corollary}

\usepackage{natbib} 
\bibpunct{(}{)}{;}{a}{}{,} 

\usepackage{etoolbox} 
\apptocmd{\sloppy}{\hbadness 10000\relax}{}{} 

\usepackage{multibib}
\newcites{sec}{References}

\usepackage{color}
\usepackage{listings}
\usepackage{hyperref}
\usepackage{booktabs}

\usepackage{lscape}

\DeclareMathOperator*{\argmin}{arg\,min}

\def\Var{\text{Var}}
\def\Cov{\text{Cov}}

\def\proj{\text{proj}}
\def\res{\text{res}}

\def\apprsim{\overset{.}{\sim}}

\def\rcoef{\gamma}

\RequirePackage[normalem]{ulem}

\usepackage{tikz}
\usetikzlibrary{fit, matrix}

\tikzset{ 
table/.style={
  matrix of nodes,
  row sep=-\pgflinewidth,
  column sep=-\pgflinewidth,
  nodes={rectangle,text width=14em,align=center},
  text depth=1.25ex,
  text height=2.5ex,
  nodes in empty cells
},
row 1/.style={nodes={fill=gray!80}},
row 2/.style={nodes={fill=gray!80}},
row 3/.style={nodes={text height=4ex,align=center}},
row 4/.style={nodes={text height=4ex,align=center}},
column 1/.style={nodes={fill=gray!80, text width=3em,align=center}},
column 2/.style={nodes={fill=gray!80, text width=3em,align=center}},
column 4/.style={nodes={text width=4em,align=center}},
}

\allowdisplaybreaks
\begin{document}
\onehalfspacing
\title{\bf 
Rerandomization and Regression Adjustment
}
\author{Xinran Li and Peng Ding
\footnote{Xinran Li, Department of Statistics, University of Illinois at Urbana-Champaign, IL 61820 (E-mail: xinranli@illinois.edu). Peng Ding, Department of Statistics, University of California, Berkeley, CA 94720 (E-mail: pengdingpku@berkeley.edu).
}
}
\date{}
\maketitle
\begin{abstract}
Randomization is a basis for the statistical inference of treatment effects without strong assumptions on the outcome-generating process. Appropriately using covariates further yields more precise estimators in randomized experiments. R. A. Fisher suggested blocking on discrete covariates in the design stage or conducting analysis of covariance (ANCOVA) in the analysis stage. We can embed blocking into a wider class of experimental design called rerandomization, and extend the classical ANCOVA to more general regression adjustment. Rerandomization trumps complete randomization in the design stage, and regression adjustment trumps the simple difference-in-means estimator in the analysis stage. It is then intuitive to use both rerandomization and regression adjustment. Under the randomization-inference framework, we establish a unified theory allowing the designer and analyzer to have access to different sets of covariates. We find that asymptotically (a) for any given estimator with or without regression adjustment, rerandomization never hurts either the sampling precision or the estimated precision, and (b) for any given design with or without rerandomization, our regression-adjusted estimator never hurts the estimated precision. Therefore, combining rerandomization and regression adjustment yields better coverage properties and thus improves statistical inference. To theoretically quantify these statements, we discuss optimal regression-adjusted estimators in terms of the sampling precision and the estimated precision, and then measure the additional gains of the designer and the analyzer. We finally suggest using rerandomization in the design and regression adjustment in the analysis followed by the Huber--White robust standard error.
\end{abstract}

\section{Introduction}

In his seminal book {\it Design of Experiments}, \citet{Fisher:1935} first formally discussed the value of randomization in experiments: randomization balances observed and unobserved covariates on average, and serves as a basis for statistical inference. Since then, randomized experiments have been widely used in agricultural sciences \citep[e.g.,][]{Fisher:1935, kempthorne1952design}, industry \citep[e.g.,][]{box2005statistics, wu2011experiments}, and clinical trials \citep[e.g.,][]{rosenberger2015randomization}. Recent years have witnessed the popularity of using randomized experiments in social sciences \citep[e.g.,][]{duflo2007using, gerber2012field, ATHEY201773} and technology companies \citep[e.g.,][]{kohavi2017online}. Those modern applications often have richer covariates.

In completely randomized experiments, covariate imbalance often occurs by chance. \citet{Fisher:1935} proposed to use analysis of covariance (ANCOVA) to adjust for covariate imbalance and thus improve estimation efficiency. \citet{Fisher:1935}'s ANCOVA uses the coefficient of the treatment in the ordinary least squares (OLS) fit of the outcome on the treatment and covariates.  However, \citet{freedman2008regression_a, freedman2008regression_b} criticized ANCOVA by showing that it can be even less efficient than the simple difference-in-means estimator under \citet{Neyman:1923}'s potential outcomes framework.  \citet{freedman2008regression_a, freedman2008regression_b}'s analyses allowed for treatment effect heterogeneity, in contrast to the existing literature on ANCOVA which often assumed additive treatment effects \citep{Fisher:1935, kempthorne1952design, cox2000theory}. \citet{lin2013} proposed a solution to Freedman's critique by running the OLS of the outcome on the treatment, covariates, and their {\it interactions}. \citet{lidingclt2016} showed the ``optimality'' of \citet{lin2013}'s estimator  within a class of regression-adjusted estimators.

Aware of the covariate imbalance issue, \citet{fisher1926} also proposed a strategy to actively avoid it in experiments. With a few discrete covariates, he proposed to use blocking, that is, to conduct completely randomized experiments within blocks of covariates. This remains a powerful tool in modern experiments \citep{miratrix2013adjusting, higgins2016improving, ATHEY201773}. Blocking is a special case of rerandomization \citep{morgan2012rerandomization}, which rejects ``bad'' random allocations that violate certain covariate balance criterion. Rerandomization can also deal with more general covariates. \citet{morgan2012rerandomization} demonstrated that rerandomization improves covariate balance. \citet{asymrerand2106} further derived the asymptotic distribution of the difference-in-means estimator, and demonstrated that rerandomization improves its precision compared to complete randomization.

Rerandomization and regression adjustment are two ways to use covariates to improve efficiency. The former uses covariates in the design stage, and the latter uses covariates in the analysis stage. It is then natural to combine them in practice, i.e., to conduct rerandomization in the design and use regression adjustment in the analysis. Several theoretical challenges remain. First, how do we conduct statistical inference? We will derive the asymptotic distribution of the regression-adjusted estimator under rerandomization without assuming any outcome-generating model. Our theory is purely randomization-based, in which potential outcomes are fixed numbers and the only randomness comes from the treatment allocation.

Second, what is the optimal regression adjustment under rerandomization? The optimality depends on the criterion. We will introduce two notions of optimality, one based on the {\it sampling precision} and the other based on the {\it estimated precision}. Because our general theory allows the designer and analyzer to have different sets of covariates, it is possible that the estimated precision differs from the sampling precision asymptotically, even under the case with additive treatment effects. We will show that asymptotically, rerandomization never hurts either the sampling precision or the estimated precision, and \citet{lin2013}'s regression adjustment never hurts the estimated precision. Therefore, combining rerandomization and regression adjustment improves the coverage properties of the associated confidence intervals. Based on these findings, we suggest using \citet{lin2013}'s estimator in general settings, and show that the Huber--White variance estimator is a convenient approximation to its variance under rerandomization. Importantly, our theory does not rely on the linear model assumption.

Third, how do we quantify the gains from the designer and analyzer? In particular, if the analyzer uses an optimal regression adjustment, what is the additional gain of rerandomization compared to complete randomization? If the designer uses rerandomization, what is the additional gain of using an optimal regression adjustment compared to the simple difference-in-means? Our theory can quantitatively answer these questions.

This paper proceeds as follows. Section \ref{sec::framework} introduces the framework and notation. Section \ref{sec:sampling_dist} derives the sampling distribution of the regression-adjusted estimator under rerandomization. Section \ref{sec:opt_rem} discusses optimal regression adjustment 
 in terms of
the sampling precision. Section \ref{sec::estimatedprecision} addresses estimation and inference issues. Section \ref{sec::Coptimalreg} discusses optimal regression adjustment 
in terms of
the estimated precision. 
Section \ref{sec:gains} quantifies the gains from the analyzer and the designer in both the sampling precision and the estimated precision. 
Section \ref{sec::unification} unifies the discussion and gives practical suggestions. 
Section \ref{sec::illustrations} uses  examples to illustrate the theory. Section \ref{sec::discussion} concludes, and the Supplementary Material contains all the technical details.

\section{Framework and notation}\label{sec::framework}

Consider an experiment on $n$ units, with $n_1$ of them assigned to the treatment and $n_0$ of them assigned to the control. Let $r_1 = n_1/n$ and $r_0 = n_0/n$ be the proportions of units receiving the treatment and control. We use potential outcomes to define treatment effects \citep{Neyman:1923}. For unit $i$, let $Y_i(1)$ and $Y_i(0)$ be the potential outcomes under the  treatment and control, and $\tau_i = Y_i(1) - Y_i(0)$ be the individual treatment effect. For this finite population of $n$ units, the average potential outcome under treatment arm $z$ $(z=0,1)$ is $\bar{Y}(z) = n^{-1}\sum_{i=1}^{n}Y_i(z)$, and the average treatment effect is $\tau = n^{-1}\sum_{i=1}^n \tau_i = \bar{Y}(1)-\bar{Y}(0)$.
Let $Z_i$ be the treatment assignment for unit $i$ ($Z_i=1$ for the treatment; $Z_i=0$ for the control), and $\bm{Z}=(Z_1,Z_2, \ldots,Z_n)'$ be the treatment assignment vector.  The observed outcome for unit $i$ is $Y_i = Z_iY_i(1)+(1-Z_i)Y_i(0)$.

\subsection{Regression adjustment in the analysis}\label{sec:reg_in_ana}
In a completely randomized experiment (CRE), the probability that $\bm{Z}$ takes a particular value $\bm{z}=(z_1,\ldots,z_n)$ is 
$
\binom{n}{n_1}^{-1},
$ 
where $\sum_{i=1}^n z_i = n_1$ and $\sum_{i=1}^n (1-z_i) = n_0$ are fixed and do not depend on the values of covariates or potential outcomes. Equivalently, $\bm{Z}$ is a random permutation of a vector of $n_1$ 1's and $n_0$ 0's.  Let $\bm{w}_{i}=(w_{i1}, \ldots, w_{iJ})'$ be the $J$ observed pretreatment covariates available to the analyzer. For descriptive convenience, we center these covariates at mean zero, i.e., $n^{-1}\sum_{i=1}^{n}\bm{w}_i=\bm{0}.$ Let 
\begin{align*}
\hat{\tau} = n_1^{-1}\sum_{i=1}^n Z_i Y_i - n_0^{-1}\sum_{i=1}^n (1-Z_i) Y_i, \quad 
\hat{\bm{\tau}}_{\bm{w}} = n_1^{-1}\sum_{i=1}^n Z_i \bm{w}_{i}- n_0^{-1}\sum_{i=1}^n (1-Z_i) \bm{w}_{i} 
\end{align*}
be the difference-in-means of the outcome $Y$ and covariates $\bm{w}$, respectively. Without covariate adjustment, $\hat{\tau}$ is unbiased for $\tau$. After the experiment, the analyzer can improve the estimation precision for the average treatment effect by adjusting for the observed covariate imbalance $\hat{\bm{\tau}}_{\bm{w}}$. A general linear regression-adjusted estimator has the following equivalent forms: 
\begin{align}\label{eq:reg}
\hat{\tau}(\bm{\beta}_1, \bm{\beta}_0) & =
n_1^{-1}\sum_{i=1}^n Z_i\left(
Y_i - {\bm{\beta}}_{1}' \bm{w}_{i}
\right) - 
n_0^{-1}\sum_{i=1}^n (1-Z_i)\left(
Y_i - {\bm{\beta}}_0' \bm{w}_{i}
\right)
\nonumber\\& 
= \hat{\tau} - \left( r_0 \bm{\beta}_1 + r_1 \bm{\beta}_0\right)'\hat{\bm{\tau}}_{\bm{w}} 
=  \hat{\tau} - \bm{\rcoef}'\hat{\bm{\tau}}_{\bm{w}},
\end{align}
where $\bm{\beta}_1, \bm{\beta}_0$ and $\bm{\rcoef}  =  r_0 \bm{\beta}_1 + r_1 \bm{\beta}_0$ are $J$ dimensional coefficients. From \eqref{eq:reg}, $\hat{\tau}(\bm{\beta}_1, \bm{\beta}_0)$ depends on $(\bm{\beta}_1, \bm{\beta}_0)$ only through $\bm{\rcoef}  =  r_0 \bm{\beta}_1 + r_1 \bm{\beta}_0$. Therefore, the choice of $(\bm{\beta}_1, \bm{\beta}_0)$ is not unique to achieve the same efficiency gain. For simplicity, we will also call \eqref{eq:reg} an {\it adjusted estimator} from now on.

\citet{Fisher:1935}'s ANCOVA chose $\bm{\beta}_1 =  \bm{\beta}_0$ to be the coefficient of $\bm{w}$ in the OLS fit of the observed outcome $Y$ on the treatment $Z$ and covariates $\bm{w}$. \citet{freedman2008regression_b} criticized ANCOVA because (a) the resulting estimator can be even less efficient than $\hat{\tau}$ and (b) the standard error based on the OLS can be inconsistent under the potential outcomes framework. \citet{lin2013} fixed (a) by choosing $\bm{\beta}_1$ and $ \bm{\beta}_0$ to be the coefficients of $\bm{w}$ in the OLS fit of $Y$ on $\bm{w}$ for treated and control units, respectively. The resulting adjusted estimator is numerically identical to the coefficient of $Z$ in the OLS fit of $Y$ on $Z$, $\bm{w}$ and $Z\times \bm{w}$. \citet{lin2013} fixed (b) by using the Huber--White robust standard error for linear models. Asymptotically, \citet{lin2013}'s estimator has smaller standard error and estimated standard error than $\hat{\tau}$.

As a side note, \citet{lin2013}'s estimator also appeared in the semiparametric efficiency theory for the average treatment effect under independent sampling from a superpopulation \citep{koch1998issues, yang2001efficiency, leon2003semiparametric, tsiatis2008covariate, rubin2011targeted}.

\subsection{Rerandomization in the design}

The above regression adjustment uses covariates in the analysis stage. We can also use covariates in the design stage to improve the quality of randomization and the efficiency of estimates. Before conducting the experiment, the designer collects $K$ covariates $\bm{x}_{i}=(x_{i1}, \ldots, x_{iK})'$ for unit $i.$ Similarly, we center the covariates at mean zero, i.e., $n^{-1}\sum_{i=1}^{n}\bm{x}_i=\bm{0}.$ Note that we allow $\bm{x}$ to be different from $\bm{w}$. The CRE balances covariates on average, but an unlucky draw of the treatment vector can result in large covariate imbalance \citep{student:1938, cox:1982, cox:2009, Bruhn:2009, morgan2012rerandomization}. Therefore, it is sensible for the designer to check the covariate balance before conducting the experiment. 
Let 
$
\hat{\bm{\tau}}_{\bm{x}}  =n_1^{-1}\sum_{i=1}^n Z_i \bm{x}_{i}- n_0^{-1}\sum_{i=1}^n (1-Z_i) \bm{x}_{i}
$
be the difference-in-means of the covariates $\bm{x}$ between the treatment and control groups. It has mean zero under the CRE. However, imbalance in covariate distributions often occurs for a realized treatment allocation. We can discard those unlucky treatment allocations with large covariate imbalance, and rerandomize until the allocation satisfies a certain covariate balance criterion. This is rerandomization, which has the following steps: 
\begin{itemize}
\item[(S1)] collect covariate data and specify a covariate balance criterion;
\item[(S2)] randomize the units into treatment and control groups;
\item[(S3)]
if the allocation satisfies the balance criterion, proceed to (S4); otherwise, return to (S2);
\item[(S4)] conduct the experiment using the accepted allocation from (S3). 
\end{itemize}  

The balance criterion in (S1) can be a general function of the treatment assignment $\bm{Z}$ and the covariates $(\bm{x}_1, \ldots, \bm{x}_n)$.  
\citet{morgan2012rerandomization} focused on rerandomization using the Mahalanobis distance (ReM). ReM accepts a randomization if and only if $M\equiv \bm{\hat{\tau}}_{\bm{x}}'\{\Cov(\bm{\hat{\tau}}_{\bm{x}})\}^{-1}\bm{\hat{\tau}}_{\bm{x}} \leq a$, where $M$ is the Mahalanobis distance between the covariate means in two groups and $a>0$ is the predetermined threshold. \citet{asymrerand2106} derived the asymptotic distribution of $\hat{\tau}$ under ReM, and showed that it is more precise than $\hat{\tau}$ under the CRE. They further showed that when $a$ is small and $\bm{x}=\bm{w}$, the asymptotic variance of $\hat{\tau}$ under ReM is nearly identical to \citet{lin2013}'s adjusted estimator under the CRE. Therefore, rerandomization and regression adjustment both use covariates to improve efficiency of treatment effect estimates, but in the design and analysis stages, respectively.

\section{Sampling distributions of regression adjustment under ReM}\label{sec:sampling_dist}

Section \ref{sec::framework} shows that ReM trumps the CRE in the design stage and regression adjustment trumps the difference-in-means in the analysis stage. Therefore, it is natural to combine ReM and regression adjustment. Then a key question is how do we conduct statistical inference. This requires us to study the sampling distribution of the adjusted estimator $\hat{\tau}(\bm{\beta}_1, \bm{\beta}_0)$ in \eqref{eq:reg} under ReM.

\subsection{Basics of randomization-based inference}\label{sec::basics-for-theory}

To facilitate the discussion, we introduce some basic results from finite population causal inference. The first part describes fixed finite population quantities without randomness. The second part describes the repeated sampling properties and asymptotics under the CRE. 

\subsubsection{Finite population quantities, projections, and regularity conditions}

For the treatment arm $z$ $(z=0,1)$, let $S_{Y(z)}^2 = (n-1)^{-1}\sum_{i=1}^{n}\{Y_i(z)-\bar{Y}(z)\}^2$ be the finite population variance of the potential outcomes, and $\bm{S}_{Y(z),\bm{x}}=\bm{S}_{\bm{x},Y(z)}' = (n-1)^{-1}\sum_{i=1}^{n}  \{Y_i(z) - \bar{Y}(z)\} \bm{x}_i'$ be the finite population covariance between the potential outcomes and covariates. Let $\bm{S}_{\bm{x}}^2=(n-1)^{-1}\sum_{i=1}^{n}\bm{x}_i \bm{x}_i'$ be the finite population covariance of the  covariates. We can similarly define $\bm{S}_{Y(z), \bm{w}}$, $\bm{S}^2_{\bm{w}}$, and other covariances. 

We introduce linear projections among these fixed quantities. For example, the linear projection of the potential outcome $Y(z)$ on covariates $\bm{w}$ is $\bar{Y}(z) + \tilde{\bm{\beta}}_z'\bm{w}_i$ for unit $i$, with the coefficient
\begin{align}\label{eq:beta_tilde_z}
\tilde{\bm{\beta}}_z = \argmin_{\bm{b} \in \mathbb{R}^J}n^{-1}\sum_{i=1}^n \left\{ Y_i(z) - \bar{Y}(z) - \bm{b}'\bm{w}_i \right\}^2 =  \left(\bm{S}_{\bm{w}}^2\right)^{-1}\bm{S}_{\bm{w}, Y(z)}, \quad (z=0,1) . 
\end{align}
The residual from this projection is $Y_i(z) - \bar{Y}(z) - \tilde{\bm{\beta}}_z'\bm{w}_i$ for unit $i$. Let $S^2_{Y(z)\mid \bm{w}}\equiv \bm{S}_{Y(z),\bm{w}} (\bm{S}_{\bm{w}}^2)^{-1}\bm{S}_{\bm{w}, Y(z)}$ denote the finite population variance of the linear projections, and $S^2_{Y(z)\setminus \bm{w}} \equiv S^2_{Y(z)} - S^2_{Y(z)\mid \bm{w}}$ the finite population variance of the residuals. We can similarly define $S^2_{\tau \mid \bm{w}}$, $S^2_{\tau \setminus \bm{w}}$,  ${S}^2_{Y(z)\mid \bm{x}}$,  ${S}^2_{Y(z)\setminus \bm{x}},{S}^2_{\tau \mid \bm{x}}, \bm{S}^2_{\bm{w}\mid \bm{x}}$ and $\bm{S}^2_{\bm{w}\setminus \bm{x}}$.

The exact distributions of the estimators depend on unknown potential outcomes in general. We will use asymptotic approximations. Finite population asymptotics embeds the $n$ units into a sequence of finite populations with increasing sizes. Technically, all quantities above depend on $n$, but we keep their dependence on $n$ implicit for notational simplicity. Moreover, the sequence of finite populations must satisfy some regularity conditions to ensure the existence of the limiting distributions of the estimators. We use the regularity conditions motivated by the finite population central limit theorems \citep{lidingclt2016}.

\begin{condition}\label{con:fp}
As $n\rightarrow\infty$, 
the sequence of finite populations satisfies that, for $z=0,1$, 
\begin{itemize}
\item[(i)] $r_z = n_z/n$, the proportion of units receiving treatment $z$, has a positive limit; 
\item[(ii)] 
the finite population variances and covariances, 
$S_{Y(z)}^2, S^2_\tau, \bm{S}_{\bm{x}}^2, \bm{S}_{\bm{w}}^2, \bm{S}_{Y(z),\bm{x}}, \bm{S}_{Y(z),\bm{w}}$ and $\bm{S}_{\bm{x},\bm{w}}$, have limiting values, 
and the limits of $\bm{S}_{\bm{x}}^2$ and $\bm{S}_{\bm{w}}^2$ are nonsingular;
\item[(iii)] $\max_{1\leq i\leq n}|Y_i(z)-\bar{Y}(z)|^2/n\rightarrow 0$, $\max_{1\leq i\leq n}\|\bm{x}_i\|_2^2/n \rightarrow 0$, and 
$\max_{1\leq i\leq n}\|\bm{w}_i\|_2^2/n \rightarrow 0$.
\end{itemize}
\end{condition}

In Condition \ref{con:fp}, (i) and (ii) are natural, and (iii) holds almost surely if all the variables are independent and identically distributed (i.i.d.) draws from a superpopulation with more than two moments \citep{lidingclt2016}. 
Throughout the paper, we assume the numbers of covariates $K$ in the design and $J$ in the analysis are both fixed and do not increase with the sample size $n$.

\subsubsection{Repeated sampling inference under the CRE}
\label{sec::CRE-neyman}

Under the CRE, over all $\binom{n}{n_1}$ randomizations, $ n^{1/2}(\hat{\tau}-{\tau}, \hat{\bm{\tau}}_{\bm{x}}',  \hat{\bm{\tau}}_{\bm{w}}') ' $ has mean $\bm{0}$ and covariance 
\begin{align}
\bm{V} & \equiv
\begin{pmatrix}
{r_1^{-1}}S_{Y(1)}^2 + {r_0^{-1}}S_{Y(0)}^2 - S_{\tau}^2 & {r_1^{-1}}\bm{S}_{Y(1),\bm{x}}+{r_0^{-1}}\bm{S}_{Y(0),\bm{x}} & 
{r_1^{-1}}\bm{S}_{Y(1),\bm{w}}+{r_0^{-1}}\bm{S}_{Y(0),\bm{w}}
\\
{r_1^{-1}}\bm{S}_{\bm{x}, Y(1)}+{r_0^{-1}}\bm{S}_{\bm{x},Y(0)} & (r_1r_0)^{-1}\bm{S}_{\bm{x}}^2 & 
(r_1r_0)^{-1}\bm{S}_{\bm{x}, \bm{w}}\\
{r_1^{-1}}\bm{S}_{\bm{w}, Y(1)}+{r_0^{-1}}\bm{S}_{\bm{w},Y(0)} & 
(r_1r_0)^{-1}\bm{S}_{\bm{w}, \bm{x}} & (r_1r_0)^{-1}\bm{S}_{\bm{w}}^2
\end{pmatrix}  \nonumber \\
& \equiv \begin{pmatrix}
V_{\tau\tau} & \bm{V}_{\tau \bm{x}} & \bm{V}_{\tau \bm{w}}\\
\bm{V}_{\bm{x}\tau} & \bm{V}_{\bm{xx}}  & \bm{V}_{ \bm{xw}}\\
\bm{V}_{\bm{w}\tau} & \bm{V}_{ \bm{wx}} & \bm{V}_{\bm{ww}}
\end{pmatrix}.  \label{eq::vvvv}
\end{align}

The finite population central limit theorem of \citet{lidingclt2016} ensures that $n^{1/2}(\hat{\tau}-{\tau}, \hat{\bm{\tau}}_{\bm{x}}',  \hat{\bm{\tau}}_{\bm{w}}')'$ is asymptotically Gaussian with mean $\bm{0}$ and covariance matrix $\bm{V}$ under the CRE and Condition \ref{con:fp}. We use $\apprsim$ for two sequences of random vectors (or distributions) converging weakly to the same distribution. Therefore, 
$
n^{1/2}(\hat{\tau} -{\tau} , \hat{\bm{\tau}}_{\bm{x}}', \hat{\bm{\tau}}_{\bm{w}}')'
\apprsim
\mathcal{N}(\bm{0}, \bm{V}).
$

We define linear projections for random variables. We use $\mathbb{E}(\cdot), \Var(\cdot)$ and $\Cov(\cdot)$ for mean, variance and covariance, and $\proj(\cdot \mid \cdot)$ and $\res(\cdot \mid \cdot)$ for linear projections and corresponding residuals, exclusively under the CRE.  For example, the linear projection of $\hat{\tau}$ on $\hat{\bm{\tau}}_{\bm{w}}$ is 
$
\proj(\hat{\tau} \mid \hat{\bm{\tau}}_{\bm{w}}) = \tau + \tilde{\bm{\rcoef}}'\hat{\bm{\tau}}_{\bm{w}}, 
$
with the coefficient 
\begin{align}\label{eq:beta_tilde}
\tilde{\bm{\rcoef}} 
= \argmin_{\bm{b}\in \mathbb{R}^J}\mathbb{E}\left(
\hat{\tau} - \tau - \bm{b}'\hat{\bm{\tau}}_{\bm{w}}
\right)^2 
= \left\{\Cov\left(\hat{\bm{\tau}}_{\bm{w}}\right)\right\}^{-1}  \Cov\left(\hat{\bm{\tau}}_{\bm{w}}, \hat{\tau}\right)
 = \bm{V}_{\bm{ww}}^{-1} \bm{V}_{\bm{w}\tau} . 
\end{align}
The residual from this projection is 
$\res(\hat{\tau} \mid \hat{\bm{\tau}}_{\bm{w}}) = \hat{\tau} - \proj(\hat{\tau} \mid \hat{\bm{\tau}}_{\bm{w}}) = \hat{\tau} - \tau - \tilde{\bm{\rcoef}}'\hat{\bm{\tau}}_{\bm{w}}.$
We can similarly define 
$\proj(\hat{\bm{\tau}}_{\bm{x}} \mid \hat{\bm{\tau}}_{\bm{w}})$ and $\res(\hat{\bm{\tau}}_{\bm{x}} \mid \hat{\bm{\tau}}_{\bm{w}})$.

Finally, the three linear projection coefficients $\tilde{\bm{\beta}}_1, \tilde{\bm{\beta}}_0$ and $\tilde{\bm{\rcoef}}$ defined in \eqref{eq:beta_tilde_z} and \eqref{eq:beta_tilde} have the following relationship.

\begin{proposition}\label{prop:three_beta_tilde}
$r_0\tilde{\bm{\beta}}_1+r_1 \tilde{\bm{\beta}}_0 = \tilde{\bm{\rcoef}}.$
\end{proposition}

Proposition \ref{prop:three_beta_tilde} is related to the non-uniqueness of the regression adjustment in \eqref{eq:reg}. It is important for the discussion below. 

\subsection{Asymptotic distribution of regression adjustment under ReM}\label{sec:reg_rem}

Equipped with the tools in Section \ref{sec::basics-for-theory}, we now can derive the asymptotic distribution of $\hat{\tau}(\bm{\beta}_1, \bm{\beta}_0)$ under ReM. We first fix the coefficients $\bm{\beta}_1$ and $\bm{\beta}_0$, and will devote several sections to discuss the optimal choices of them.

For unit $i$, let $Y_i(z;\bm{\beta}_z) \equiv   Y_i(z) - {\bm{\beta}}_{z}'\bm{w}_{i}$ be the ``adjusted" potential outcome under the treatment level $z\ (z=0,1)$, 
and $\tau_i(\bm{\beta}_1, \bm{\beta}_0)\equiv \tau_i- (\bm{\beta}_1-\bm{\beta}_0)'\bm{w}_{i}$ be the ``adjusted" individual treatment effect. The average ``adjusted'' treatment effect $\tau(\bm{\beta}_1, \bm{\beta}_0) \equiv n^{-1}\sum_{i=1}^n\tau_i(\bm{\beta}_1, \bm{\beta}_0) =  \tau$ is identical to the average unadjusted  treatment effect because of the centering of $\bar{\bm{w}}=\bm{0}$. The ``adjusted" observed outcome is   $Y_i(\bm{\beta}_1,\bm{\beta}_0)=Z_iY_i(1;\bm{\beta}_1)+(1-Z_i)Y_i(0;\bm{\beta}_0)$. The adjusted estimator \eqref{eq:reg} is essentially the difference-in-means estimator with the ``adjusted'' potential outcomes. 
For $z=0,1,$ let
$S_{Y(z;\bm{\beta}_z)}^2$ and 
$
S_{Y(z;\bm{\beta}_z)\mid \bm{x}}^2
$ 
be the finite population variances of 
$Y_i(z;\bm{\beta}_z)$ and its linear projection on $\bm{x}_i$. Let
$S_{\tau(\bm{\beta}_1, \bm{\beta}_0)}^2$ 
and 
$S_{\tau(\bm{\beta}_1, \bm{\beta}_0)\mid \bm{x}}^2$
be the finite population variances of $\tau_i(\bm{\beta}_1, \bm{\beta}_0)$ and its linear projection on $\bm{x}_i$. 
From Section \ref{sec::CRE-neyman}, under the CRE, 
the variance of $n^{1/2} \{  \hat{\tau}(\bm{\beta}_1, \bm{\beta}_0) - \tau \} $ is 
\begin{align}\label{eq:V_tau_adj}
V_{\tau\tau}(\bm{\beta}_1, \bm{\beta}_0)
= {r_1^{-1}}S_{Y(1;\bm{\beta}_1)}^2 + {r_0^{-1}}S_{Y(0;\bm{\beta}_0)}^2 - S_{\tau(\bm{\beta}_1, \bm{\beta}_0)}^2,
\end{align}
and the squared multiple correlation between 
$\hat{\tau}(\bm{\beta}_1, \bm{\beta}_0)$ and $\hat{\bm{\tau}}_{\bm{x}}$ is  
\citep[][Proposition 1]{asymrerand2106}
\begin{align}\label{eq:R2_tau_x_beta}
R^2_{\tau, \bm{x}}(\bm{\beta}_{1}, \bm{\beta}_0)
= 
\frac{\Var\left\{\proj\left(\hat{\tau}(\bm{\beta}_{1}, \bm{\beta}_0) \mid \hat{\bm{\tau}}_{\bm{x}}\right)\right\}}{
\Var\left\{ \hat{\tau}(\bm{\beta}_1, \bm{\beta}_0) \right\}
}
=
\frac{r_1^{-1}S_{Y(1;\bm{\beta}_1)\mid \bm{x}}^2+r_0^{-1}S_{Y(0;\bm{\beta}_0)\mid \bm{x}}^2- S_{\tau(\bm{\beta}_1, \bm{\beta}_0)\mid \bm{x}}^2}{
{r_1^{-1}}S_{Y(1;\bm{\beta}_1)}^2 + {r_0^{-1}}S_{Y(0;\bm{\beta}_0)}^2 - S_{\tau(\bm{\beta}_1, \bm{\beta}_0)}^2
}. 
\end{align}

The asymptotic distribution of $\hat{\tau}(\bm{\beta}_1, \bm{\beta}_0)$ under ReM is a linear combination of two independent random variables $\varepsilon$ and $L_{K,a} $, where $\varepsilon \sim \mathcal{N}(0,1)$ is  a standard Gaussian random variable and $L_{K,a} \sim D_1\mid \bm{D}'\bm{D}\leq a$ is a truncated Gaussian random variable with $\bm{D} = (D_1,\ldots, D_K)\sim \mathcal{N}(\bm{0}, \bm{I}_K)$. Let $\mathcal{M}$ denote the event $M\leq a$.

\begin{theorem}\label{thm:adj_rem}
Under ReM and Condition \ref{con:fp}, 
\begin{eqnarray}\label{eq:adj_rem}
n^{1/2}\left\{\hat{\tau}(\bm{\beta}_1, \bm{\beta}_0)-\tau \right\} \mid 
\mathcal{M} 
\ \apprsim \ 
 V_{\tau\tau}^{1/2}(\bm{\beta}_1, \bm{\beta}_0)
\left[
\left\{1-R^2_{\tau, \bm{x}}(\bm{\beta}_{1}, \bm{\beta}_0)\right\}^{1/2}\cdot \varepsilon + \left\{R^2_{\tau, \bm{x}}(\bm{\beta}_{1}, \bm{\beta}_0)\right\}^{1/2} \cdot  L_{K,a}\right]. 
\end{eqnarray}
\end{theorem}

The $\varepsilon$ component in \eqref{eq:adj_rem} represents the part of $\hat{\tau}(\bm{\beta}_1, \bm{\beta}_0)$ that cannot be explained by $\hat{\bm{\tau}}_{\bm{x}}$ and is thus unaffected by rerandomization. The $L_{K,a}$ component in \eqref{eq:adj_rem} represents the part of $\hat{\tau}(\bm{\beta}_1, \bm{\beta}_0)$ that can be explained by $\hat{\bm{\tau}}_{\bm{x}}$ and is thus affected by rerandomization. Moreover, 
the asymptotic distribution \eqref{eq:adj_rem} is symmetric around zero, 
and the adjusted estimator is consistent for the average treatment effect, for any fixed values of the coefficients $\bm{\beta}_1$ and $\bm{\beta}_0$.  Theorem \ref{thm:adj_rem} immediately implies the following two important special cases.

\subsubsection{Special case: regression adjustment under the CRE}

The CRE is a special case of ReM with $a=\infty$. Therefore, Theorem \ref{thm:adj_rem} implies that under the CRE,  $n^{1/2}\{\hat{\tau}(\bm{\beta}_1, \bm{\beta}_0)-\tau\}$ is asymptotically Gaussian with mean $0$ and variance $V_{\tau\tau}(\bm{\beta}_1, \bm{\beta}_0)$.

\begin{corollary}\label{corr:reg_cre} 
Under CRE and Condition \ref{con:fp}, 
$
n^{1/2}\left\{\hat{\tau}(\bm{\beta}_1, \bm{\beta}_0)-\tau \right\} 
\apprsim 
V_{\tau\tau}^{1/2}(\bm{\beta}_1, \bm{\beta}_0) \cdot \varepsilon .
$
\end{corollary}

Corollary \ref{corr:reg_cre} is a known result from \citet{lin2013} and \citet{lidingclt2016}. 

\subsubsection{Special case: no covariate adjustment under ReM}
Using Theorem \ref{thm:adj_rem} with $\bm{\beta}_1=\bm{\beta}_0=\bm{0}$, we can immediately obtain the asymptotic distribution of $\hat{\tau}\equiv \hat{\tau}(\bm{0}, \bm{0})$ under ReM. Let $R^2_{\tau, \bm{x}} \equiv R^2_{\tau, \bm{x}}(\bm{0}, \bm{0})$ be the squared multiple correlation between $\hat{\tau}$ and $\hat{\bm{\tau}}_{\bm{x}}$ under the CRE: 
\begin{align}\label{eq:R2_tau_x}
R^2_{\tau, \bm{x}}
& = 
\frac{\Var\left\{\proj(\hat{\tau} \mid \hat{\bm{\tau}}_{\bm{x}})\right\}}{
\Var(\hat{\tau})	
}
= 
\frac{\bm{V}_{\tau \bm{x}} \bm{V}_{\bm{x}\bm{x}}^{-1} \bm{V}_{ \bm{x}\tau}}{V_{\tau\tau}}
= 
\frac{r_1^{-1}S_{Y(1)\mid \bm{x}}^2+r_0^{-1}S_{Y(0)\mid \bm{x}}^2
- S_{\tau\mid \bm{x}}^2
}{{r_1^{-1}}S_{Y(1)}^2 + {r_0^{-1}}S_{Y(0)}^2 - S_{\tau}^2}. 
\end{align}
Then $\hat{\tau}$ has the following asymptotic distribution. 
\begin{corollary}\label{corr:diff_rem}
Under ReM and Condition \ref{con:fp}, 
\begin{eqnarray}\label{eq:diff_rerand}
n^{1/2}\left( \hat{\tau} - \tau \right)\mid \mathcal{M}
& \apprsim &
V_{\tau\tau}^{1/2}
\left\{
\left( 1-R^2_{\tau, \bm{x}} \right)^{1/2}\cdot \varepsilon + \left( R^2_{\tau, \bm{x}} \right)^{1/2} \cdot L_{K,a}\right\}.
\end{eqnarray}
\end{corollary}

Corollary \ref{corr:diff_rem} is a main result of \citet{asymrerand2106}.

\section{$\mathcal{S}$-optimal regression adjustment}\label{sec:opt_rem}

How do we choose the adjustment coefficients $(\bm{\beta}_1, \bm{\beta}_0)$ or $\bm{\rcoef}$? It is an important practical question. From Theorem \ref{thm:adj_rem}, the adjusted estimator is consistent for any fixed coefficients $\bm{\beta}_1$ and $\bm{\beta}_0$. Therefore, it is intuitive to choose the coefficients to achieve better precision. A measure of precision is based on the quantile ranges of an estimator. 

We introduce the notion of the $\mathcal{S}$-optimal adjusted estimator, using $\mathcal{S}$ to emphasize its dependence on the {\it sampling} distribution.

\begin{definition}\label{def:optimal_regression}
Given the design,
$\hat{\tau}(\bm{\beta}_1, \bm{\beta}_0)$ is $\mathcal{S}$-optimal if $n^{1/2}\{\hat{\tau}(\bm{\beta}_1, \bm{\beta}_0)-\tau\}$ has the shortest asymptotic $1-\alpha$ quantile range among all adjusted estimators in \eqref{eq:reg}, for any $\alpha \in (0,1).$ 
\end{definition}

In general, quantile ranges are not unique. Importantly, the asymptotic distribution in \eqref{eq:adj_rem} is symmetric and unimodal around $\tau$ \citep{asymrerand2106}. We consider only symmetric quantile ranges because they have the shortest lengths \citep[][Theorem 9.3.2]{Casella:2002aa}.
The $\mathcal{S}$-optimal adjusted estimator also has the smallest asymptotic variance among all estimators in \eqref{eq:reg} \citep{li2018rerandomization}. Conversely, if all adjusted estimators in \eqref{eq:reg} are asymptotically Gaussian, then the one with the smallest asymptotic variance is $\mathcal{S}$-optimal.

Theorem \ref{thm:adj_rem} shows a complicated relationship between the coefficients $(\bm{\beta}_1, \bm{\beta}_0)$ and the asymptotic distribution \eqref{eq:adj_rem}. Below we simplify \eqref{eq:adj_rem}. 
Let $\proj(\hat{\bm{\tau}}_{\bm{w}} \mid \hat{\bm{\tau}}_{\bm{x}}) \equiv \bm{V}_{\bm{wx}}\bm{V}_{\bm{xx}}^{-1}\hat{\bm{\tau}}_{\bm{x}}$ be the linear projection of $\hat{\bm{\tau}}_{\bm{w}}$ on $\hat{\bm{\tau}}_{\bm{x}}$, and 
$\res(\hat{\bm{\tau}}_{\bm{w}} \mid \hat{\bm{\tau}}_{\bm{x}})
\equiv \hat{\bm{\tau}}_{\bm{w}} - 
\bm{V}_{\bm{wx}}\bm{V}_{\bm{xx}}^{-1}\hat{\bm{\tau}}_{\bm{x}}$ be the residual from this linear projection. 
We further consider two projections. First, the linear projection of $\proj(\hat{\tau}\mid \hat{\bm{\tau}}_{\bm{x}})$ on 
$\proj(\hat{\bm{\tau}}_{\bm{w}} \mid \hat{\bm{\tau}}_{\bm{x}})$ has coefficient and squared multiple correlation
\begin{align}\label{eq:beta_proj}
\tilde{\bm{\rcoef}}_{\proj} 
=  (
\bm{V}_{\bm{wx}}\bm{V}_{\bm{xx}}^{-1}\bm{V}_{\bm{xw}}
)^{-1}
\bm{V}_{\bm{wx}}\bm{V}_{\bm{xx}}^{-1}\bm{V}_{\bm{x}\tau}
\text{ \ and \ }
R^2_{\proj} .
\end{align}
Second, the linear projection of $\res(\hat{\tau} \mid \hat{\bm{\tau}}_{\bm{x}})$ on $\res(\hat{\bm{\tau}}_{\bm{w}} \mid \hat{\bm{\tau}}_{\bm{x}})$ has coefficient and squared multiple correlation
\begin{align}\label{eq:beta_res}
\tilde{\bm{\rcoef}}_{\res} \equiv (
\bm{V}_{\bm{ww}} - \bm{V}_{\bm{wx}}\bm{V}_{\bm{xx}}^{-1}\bm{V}_{\bm{xw}}
)^{-1}
(\bm{V}_{\bm{w}\tau} - \bm{V}_{\bm{wx}}\bm{V}_{\bm{xx}}^{-1}\bm{V}_{\bm{x}\tau} )
\text{ \ and \ }
R^2_{\res} .
\end{align}

Technically, the expressions for $\tilde{\bm{\rcoef}}_{\proj}$ and $\tilde{\bm{\rcoef}}_{\res}$ above are well-defined only if the covariance matrices of  $\proj(\hat{\bm{\tau}}_{\bm{w}} \mid \hat{\bm{\tau}}_{\bm{x}})$ and $\res(\hat{\bm{\tau}}_{\bm{w}} \mid \hat{\bm{\tau}}_{\bm{x}})$ are nonsingular. 
Otherwise, they are not unique. However, this will not cause any issues in the later discussions because the linear projections themselves
are always unique. 

Recall that we have defined $\bm{V}$ in \eqref{eq::vvvv}, and $\bm{S}^2_{\bm{w}\mid \bm{x}}$ and 
$\bm{S}^2_{\bm{w}\setminus \bm{x}}$ as the finite population covariances of the linear projections of $\bm{w}$ on $\bm{x}$ and the corresponding residuals. 
The following proposition shows the relationship among the three linear projection coefficients $(\tilde{\bm{\rcoef}}, \tilde{\bm{\rcoef}}_{\proj}, \tilde{\bm{\rcoef}}_{\res})$. 

\begin{proposition}\label{lemma:relation_three}
$
\bm{S}^2_{\bm{w}\setminus \bm{x}}(\tilde{\bm{\rcoef}}-\tilde{\bm{\rcoef}}_{\res}) + 
\bm{S}^2_{\bm{w}\mid \bm{x}}
(\tilde{\bm{\rcoef}}-\tilde{\bm{\rcoef}}_{\proj})  =  \bm{0}.  
$
\end{proposition}

The linear projection coefficients $(\tilde{\bm{\rcoef}},   \tilde{\bm{\rcoef}}_{\proj}, \tilde{\bm{\rcoef}}_{\res}  )$ are different in general.
However, if any two of them are equal, all of them must be equal with nonsingular $\bm{S}^2_{\bm{w}\setminus \bm{x}}$ and $\bm{S}^2_{\bm{w}\mid \bm{x}}$. 
The following theorem decomposes the asymptotic distribution \eqref{eq:adj_rem} based on $( \tilde{\bm{\rcoef}}_{\proj}, \tilde{\bm{\rcoef}}_{\res})$.

\begin{theorem}\label{thm:reg_rem_beta}
Under ReM and Condition \ref{con:fp}, recalling that $\bm{\rcoef}  \equiv r_0 \bm{\beta}_1 + r_1 \bm{\beta}_0$, we have
\begin{eqnarray}\label{eq:adj_rem_beta}
& & n^{1/2}\{\hat{\tau}(\bm{\beta}_1, \bm{\beta}_0)-\tau\} \mid  \mathcal{M} \\
& \apprsim & 
\left\{
V_{\tau\tau}
\left(1-R^2_{\tau, \bm{x}}\right)
\left(1-R^2_{\res} \right)
+ \left(r_1r_0\right)^{-1}
\left(\bm{\rcoef}-\tilde{\bm{\rcoef}}_{\res}\right)'
\bm{S}^2_{\bm{w}\setminus \bm{x}}
\left(\bm{\rcoef}-\tilde{\bm{\rcoef}}_{\res}\right)
\right\}^{1/2} \cdot  \varepsilon  \nonumber \\
& & +  
\left\{ V_{\tau\tau}R^2_{\tau, \bm{x}}
\left(1 - R^2_{\proj}\right) + \left(r_1r_0\right)^{-1}
\left(\bm{\rcoef}-\tilde{\bm{\rcoef}}_{\proj}\right)'
\bm{S}^2_{\bm{w}\mid \bm{x}}
\left(\bm{\rcoef}-\tilde{\bm{\rcoef}}_{\proj}\right)
\right\}^{1/2} \cdot   L_{K,a}. \nonumber
\end{eqnarray}
\end{theorem}

The asymptotic distribution \eqref{eq:adj_rem_beta} has two independent components. The $\varepsilon$ component in \eqref{eq:adj_rem_beta} represents the part of $\hat{\tau}(\bm{\beta}_1, \bm{\beta}_0)$ that is orthogonal to $\hat{\bm{\tau}}_{\bm{x}}$. The coefficient of $\varepsilon$ attains its minimal value at  $\bm{\rcoef}=\tilde{\bm{\rcoef}}_{\res}$ with squared minimal value 
$
V_{\tau\tau}   (1-R^2_{\tau, \bm{x}})   (1-R^2_{\res}).
$
The first term $V_{\tau\tau}$ is the variance of $n^{1/2} (\hat{\tau} - \tau ) $. The second term $1-R^2_{\tau, \bm{x}}$ represents the proportion of the variance of $\hat{\tau}$ unexplained by $\hat{\bm{\tau}}_{\bm{x}}$. The third term $1-R^2_{\res}$ represents the proportion of the variance of $\hat{\tau}$ unexplained by $\hat{\bm{\tau}}_{\bm{w}}$, after projecting onto the space orthogonal to $\hat{\bm{\tau}}_{\bm{x}}$.

The $L_{K,a}$ component in \eqref{eq:adj_rem_beta} represents the linear projection of  $\hat{\tau}(\bm{\beta}_1, \bm{\beta}_0)$ on $\hat{\bm{\tau}}_{\bm{x}}$ with the ReM constraint. The coefficient of $L_{K,a}$ attains its minimal value at  $\bm{\rcoef}=\tilde{\bm{\rcoef}}_{\proj}$  
with squared minimal value 
$
V_{\tau\tau}  R^2_{\tau, \bm{x}}  ( 1 - R^2_{\proj} ).
$
The first term $V_{\tau\tau}$ is again the variance of $n^{1/2}(\hat{\tau}-\tau)$. The second term $R^2_{\tau, \bm{x}}$ represents the proportion of the variance of $\hat{\tau}$ explained by $\hat{\bm{\tau}}_{\bm{x}}$. The third term $1-R^2_{\proj}$ represents the proportion of the variance of $\hat{\tau}$ unexplained by $\hat{\bm{\tau}}_{\bm{w}}$, after projecting onto the space of $\hat{\bm{\tau}}_{\bm{x}}$.

Because $\tilde{\bm{\rcoef}}_{\res}$ and  $\tilde{\bm{\rcoef}}_{\proj}$ are different in general, the coefficients of $\varepsilon$ and $L_{K,a}$ cannot attain their minimal values simultaneously. Consequently,
the adjusted estimator $\hat{\tau}(\tilde{\bm{\beta}}_1, \tilde{\bm{\beta}}_0)$ may not be $\mathcal{S}$-optimal under ReM, i.e., it may not have the shortest asymptotic $1-\alpha$  quantile range among \eqref{eq:reg} for all $\alpha\in (0,1)$.

The $\mathcal{S}$-optimal adjustment is complicated under ReM, especially when the designer and the analyzer have different covariate information. We will consider different scenarios based on the relative amount of covariate information used by the designer and the analyzer.

\subsection{The analyzer has no less covariate information than the designer}\label{sec:ana_more_rem}

We first consider the scenario under which the covariates $\bm{w}$ in the analysis can linearly represent the covariates $\bm{x}$ in the design. 

\begin{condition}\label{con:ana_more}
There exists a constant matrix $\bm{B}_1 \in \mathbb{R}^{K\times J}$ such that $\bm{x}_i=\bm{B}_1 \bm{w}_i$ for all unit $i$. 
\end{condition}

Condition \ref{con:ana_more} holds when the analyzer has access to all the covariates used in the design, and possibly collects more covariates after the experiment. For example, Condition \ref{con:ana_more} holds if $\bm{x}$ is a subset of $\bm{w}$. Under Condition \ref{con:ana_more}, we can simplify the asymptotic distribution of the adjusted estimator under ReM. Analogous to \eqref{eq:R2_tau_x_beta} and \eqref{eq:R2_tau_x}, let $R^2_{\tau, \bm{w}}$ be the squared multiple correlation between $\hat{\tau}$ and $\hat{\bm{\tau}}_{\bm{w}}$ under the CRE \citep{asymrerand2106}: 
\begin{align}\label{eq:R2_mid_w}
R^2_{\tau, \bm{w}}
& = 
\frac{\Var\left\{\proj(\hat{\tau} \mid \hat{\bm{\tau}}_{\bm{w}})\right\}}{
\Var(\hat{\tau})	
}
= 
\frac{\bm{V}_{\tau \bm{w}} \bm{V}_{\bm{w}\bm{w}}^{-1} \bm{V}_{ \bm{w}\tau}}{V_{\tau\tau}}
= 
\frac{r_1^{-1}S_{Y(1)\mid \bm{w}}^2+r_0^{-1}S_{Y(0)\mid \bm{w}}^2
- S_{\tau\mid \bm{w}}^2
}{{r_1^{-1}}S_{Y(1)}^2 + {r_0^{-1}}S_{Y(0)}^2 - S_{\tau}^2}. 
\end{align}

\begin{corollary}\label{cor:analysis_more}
Under Conditions \ref{con:fp} and \ref{con:ana_more},
\begin{align*}
\tilde{\bm{\rcoef}}_{\proj}=\tilde{\bm{\rcoef}}_{\res}=\tilde{\bm{\rcoef}}, \quad 
R^2_{\res} =   ( R^2_{\tau,\bm{w}} - R^2_{\tau,\bm{x}} ) / (1-R^2_{\tau,\bm{x}}) ,  \quad
R^2_{\proj}=1,
\end{align*}
and the asymptotic distribution of $\hat{\tau}({\bm{\beta}}_1, {\bm{\beta}}_0)$ under ReM is 
\begin{align}\label{eq::corollary3}
&  \  n^{1/2}\{\hat{\tau}(\bm{\beta}_1, \bm{\beta}_0)-\tau\} \mid \mathcal{M}\\
\apprsim  &
\left\{
V_{\tau\tau}\left(1-R^2_{\tau, \bm{w}}\right)
+ 
\left(r_1r_0\right)^{-1}\left(\bm{\rcoef}-\tilde{\bm{\rcoef}}\right)'
\bm{S}^2_{\bm{w}\setminus \bm{x}}
\left(\bm{\rcoef}-\tilde{\bm{\rcoef}}\right)
\right\}^{1/2}  \cdot  \varepsilon  \nonumber    \\
& +  
\left\{ 
\left(r_1r_0\right)^{-1}
\left(\bm{\rcoef}-\tilde{\bm{\rcoef}}\right)'
\bm{S}^2_{\bm{w}\mid \bm{x}}
\left(\bm{\rcoef}-\tilde{\bm{\rcoef}}\right)
\right\}^{1/2} \cdot   L_{K,a}.  \nonumber
\end{align}
\end{corollary}

From Corollary \ref{cor:analysis_more}, the coefficients of $ \varepsilon$ and $L_{K,a}$ attain minimum values at the same $\bm{\rcoef} = \tilde{\bm{\rcoef}} .$ We can then derive the $\mathcal{S}$-optimal adjusted estimator and its asymptotic distribution. 

\begin{theorem}\label{thm:analysis_more_opt}
Under ReM and Conditions \ref{con:fp} and \ref{con:ana_more}, the $\mathcal{S}$-optimal adjusted estimator is attainable when $\bm{\rcoef}=\tilde{\bm{\rcoef}}$ or 
$r_0\bm{\beta}_1+r_1\bm{\beta}_0 = r_0 \tilde{\bm{\beta}}_1+r_1 \tilde{\bm{\beta}}_0 ,$ 
with the asymptotic distribution 
\begin{eqnarray}\label{eq:opt_ana_more}
n^{1/2}\left\{\hat{\tau}(\tilde{\bm{\beta}}_1, \tilde{\bm{\beta}}_0)-\tau \right\} \mid \mathcal{M}
& \apprsim & 
\left\{V_{\tau\tau}\left(1-R^2_{\tau,\bm{w}}\right)\right\}^{1/2} \cdot  \varepsilon.
\end{eqnarray}
\end{theorem}

From Proposition \ref{prop:three_beta_tilde}, an optimal choice is $(  \bm{\beta}_1, \bm{\beta}_0) = ( \tilde{\bm{\beta}}_1, \tilde{\bm{\beta}}_0 )$ under Condition \ref{con:ana_more}. An important feature of Theorem \ref{thm:analysis_more_opt} is that the limiting distribution in \eqref{eq:opt_ana_more} does not depend on covariates $\bm{x}$ and the threshold $a$ of ReM.
Theorem \ref{thm:analysis_more_opt} has many implications, 
as discussed below. 

\subsubsection{Special case: $\mathcal{S}$-optimal adjustment under the CRE}
\label{sec:sopt_cre}

Theorem \ref{thm:analysis_more_opt} holds for the CRE (ReM with $\bm{x} = \emptyset$ and $a = \infty$). 
It thus recovers the optimality property of \citet{lin2013}'s estimator under the CRE previously proved by \citet{lidingclt2016}. 
Therefore, 
	 when the analyzer has no less covariate information than the designer, the
	  $\mathcal{S}$-optimal adjusted estimators under ReM and the CRE are the same and follow the same asymptotic distribution.
	 This implies that, with more covariates in the analysis, there is no additional gain from the designer through ReM as long as the analyzer performs the optimal  adjustment.
Section \ref{sec:gains_samp_ana_more} later contains related discussions.
	 
\subsubsection{Special case: the designer and analyzer have the same covariates}\label{sec:sam_prec_same}

 Consider the case where the analyzer has the same covariates as the designer ($\bm{x} = \bm{w}$). Compare $\hat{\tau}$ under ReM to the $\mathcal{S}$-optimal adjusted estimator $\hat{\tau}(\tilde{\bm{\beta}}_1, \tilde{\bm{\beta}}_0)$ under the CRE. 
From Corollary \ref{corr:diff_rem} and Section \ref{sec:sopt_cre}, the former has an additional independent component of $(V_{\tau\tau}R^2_{\tau, \bm{x}})^{1/2} L_{K,a}$ in the asymptotic distribution. 
	When the threshold $a$ is small, this additional component is approximately zero, and thus they have almost the same asymptotic distribution.
	Therefore, we can view rerandomization as covariate adjustment in the design stage \citep{asymrerand2106}. 
	Moreover, the former has the following advantages. 
	First, rerandomization in the design stage does not use the outcome data. Second, $\hat{\tau}$ is simpler and thus provides a more transparent analysis \citep{cox2007, Freedman2008chance, Rosenbaum:2010, lin2013}. Using $\hat{\tau}$ in rerandomization can thus avoid bias due to a specification search  of the outcome model (i.e., data snooping). 
Remark \ref{rmk:gain_same_covariate} later contains related discussions.

\subsection{The analyzer has no more covariate information than the designer}\label{sec:des_more_rem}

We then consider the scenario under which the covariates $\bm{x}$ in the design can linearly represent
the covariates $\bm{w}$ in the analysis.

\begin{condition}\label{con:des_more}
There exists a constant matrix $\bm{B}_2\in \mathbb{R}^{J\times K}$ such that $\bm{w}_i=\bm{B}_2\bm{x}_i$ for all unit $i$. 
\end{condition}

Condition \ref{con:des_more} is reasonable when the analyzer has access to only part of the covariates used in the design due to privacy or other reasons. 
For example, 
Condition \ref{con:des_more} holds if $\bm{w} $ is a subset of $\bm{x} $. It also reflects the situation where the analyzer uses only the difference-in-means estimator with $\bm{w}  = \emptyset$ 
even though 
the designer conducts ReM with $\bm{x}$. Condition \ref{con:des_more} implies $\bm{S}^2_{\bm{w}\setminus \bm{x}} = 0 $, which further implies that the coefficient of $\varepsilon$ in the asymptotic distribution \eqref{eq:adj_rem_beta} does not depend on $(\bm{\beta}_1,\bm{\beta}_0)$. We can then simplify the asymptotic distribution of the adjusted estimator.

\begin{corollary}\label{cor:design_more}
Under Conditions \ref{con:fp} and \ref{con:des_more}, 
\begin{align*}
	\tilde{\bm{\rcoef}}_{\proj}=\tilde{\bm{\rcoef}}, \quad R^2_{\res} =  0,  \quad
	R^2_{\proj}= R^2_{\tau, \bm{w}}/R^2_{\tau, \bm{x}},
\quad 
	\bm{S}^2_{\bm{w}\setminus \bm{x}}=\bm{0}, \quad 
	\bm{S}^2_{\bm{w}\mid \bm{x}}=\bm{S}^2_{\bm{w}}, 
\end{align*}
	and the asymptotic distribution of $\hat{\tau}({\bm{\beta}}_1, {\bm{\beta}}_0)$ under ReM is 
\begin{eqnarray*}
& & n^{1/2}\left\{\hat{\tau}(\bm{\beta}_1, \bm{\beta}_0)-\tau \right\} \mid \mathcal{M}  \\ 
& \apprsim & 
\left\{
	V_{\tau\tau}\left(1-R^2_{\tau, \bm{x}}\right)
\right\} ^{1/2} \cdot  \varepsilon   
+  
\left\{ V_{\tau\tau}
\left(R^2_{\tau, \bm{x}} - R^2_{\tau, \bm{w}}\right) + \left(r_1r_0 \right)^{-1}
\left(\bm{\rcoef}-\tilde{\bm{\rcoef}} \right)'
\bm{S}^2_{\bm{w}}
\left(\bm{\rcoef}-\tilde{\bm{\rcoef}}\right)
\right\}^{1/2}  \cdot  L_{K,a}. 
\end{eqnarray*}
\end{corollary}

Under Condition \ref{con:des_more}, $\res(\hat{\bm{\tau}}_{\bm{w}} \mid \hat{\bm{\tau}}_{\bm{x}}) = \bm{0}$, and thus as discussed earlier, the projection coefficient $\tilde{\bm{\rcoef}}_{\res}$ is not unique. Nevertheless, Corollary \ref{cor:design_more} does not depend on $\tilde{\bm{\rcoef}}_{\res}$.
Based on Corollary \ref{cor:design_more}, we can obtain the $\mathcal{S}$-optimal adjusted estimator and its asymptotic distribution under ReM. Let 
\begin{align}\label{eq:rho2_x_minus_w}
	\rho^2_{\tau, \bm{x}\setminus \bm{w}}
	= ( R^2_{\tau, \bm{x}}-R^2_{\tau, \bm{w}})/(1 - R^2_{\tau, \bm{w}})
	\in [0,1]
\end{align}
be the additional proportion of the variance of $\hat{\tau}$ explained by the covariates $\bm{x}$ in the design, after explained by the covariates $\bm{w}$ in the  analysis.

\begin{theorem}\label{thm:design_more_opt}
Under ReM and Conditions \ref{con:fp} and \ref{con:des_more}, the $\mathcal{S}$-optimal adjusted estimator is 
attainable when $\bm{\rcoef}=\tilde{\bm{\rcoef}}$ or 
$r_0\bm{\beta}_1+r_1\bm{\beta}_0 = r_0 \tilde{\bm{\beta}}_1+r_1 \tilde{\bm{\beta}}_0 ,$  with the asymptotic distribution
\begin{eqnarray}\label{eq:design_more_opt}
n^{1/2}\{\hat{\tau}(\tilde{\bm{\beta}}_1, \tilde{\bm{\beta}}_0)-\tau\} \mid 
\mathcal{M} 
& \apprsim & 
V_{\tau\tau}^{1/2}
\left\{
\left( 1-R^2_{\tau, \bm{x}} \right)^{1/2} \cdot  \varepsilon + 
\left( R^2_{\tau, \bm{x}}-R^2_{\tau, \bm{w}} \right)^{1/2} \cdot  L_{K,a}
\right\}\\
& \sim & 
\left\{ V_{\tau\tau}\left(1-R^2_{\tau, \bm{w}}\right) \right\}^{1/2}
\left\{
\left( 1-\rho^2_{\tau, \bm{x}\setminus\bm{w}} \right)^{1/2} \cdot \varepsilon + 
\left( \rho^2_{\tau, \bm{x}\setminus\bm{w}} \right)^{1/2}
\cdot L_{K,a}
\right\}. \nonumber
\end{eqnarray}
\end{theorem}

From Theorem \ref{thm:design_more_opt}, although the analyzer has less covariate information than the designer of ReM, 
s/he can still obtain the $\mathcal{S}$-optimal adjusted estimator using only the covariate information in the analysis.

Theorems \ref{thm:analysis_more_opt} and \ref{thm:design_more_opt} give identical optimal coefficients, but different asymptotic distributions of the optimal estimators. 
When the designer and the analyzer have the same covariates ($\bm{x} = \bm{w}$), 
both Theorems \ref{thm:analysis_more_opt} and \ref{thm:design_more_opt} hold and give identical results. 
Specifically, $\rho^2_{\tau, \bm{x}\setminus \bm{w}}$ in \eqref{eq:rho2_x_minus_w} reduces to zero, 
and the asymptotic distribution in \eqref{eq:design_more_opt} simplifies to \eqref{eq:opt_ana_more}, a Gaussian limiting distribution.

From Corollary \ref{corr:diff_rem} and Theorem \ref{thm:design_more_opt}, under ReM, 
the asymptotic distribution of the $\mathcal{S}$-optimal adjusted estimator  $\hat{\tau}(\tilde{\bm{\beta}}_1, \tilde{\bm{\beta}}_0)$
differs from that of $\hat{\tau}$ only in the coefficient of the truncated Gaussian random variable $L_{K,a}$. With a small threshold $a$, $L_{K,a}$ is close to zero and thus the gain from adjustment is small. 
Similar to the discussion in Section \ref{sec:sam_prec_same}, 
although $\hat{\tau}$ loses a little sampling precision compared to the optimal adjusted estimator, 
it does have the advantage of avoiding data snooping and improving transparency.

\subsection{General scenarios}\label{sec::generalS}

A practical complication is that the designer and analyzer may not communicate. Then it is possible that the designer and the analyzer do not use the same covariate information \citep[e.g.,][]{Bruhn:2009, ke2017errors}. Consequently, the analyzer has part of the covariate information in the design and additional covariate information. Neither Condition \ref{con:ana_more} or \ref{con:des_more} holds. Under general scenarios, unfortunately, the $\mathcal{S}$-optimal adjusted estimator may not exist, in the sense that there does not exist an estimator among \eqref{eq:reg} that has the shortest asymptotic $1-\alpha$ quantile range for all $\alpha\in (0,1)$.

Some sub-optimal strategies exist. First, we can consider the adjusted estimator with the smallest asymptotic variance or the shortest asymptotic $1-\alpha$ quantile range for a particular $\alpha\in (0,1)$. The Supplementary Material gives the formulas for the former. However, explicit formulas  for the latter do not exist.

Second, when $a$ is small, $L_{K,a} \approx 0$, and the asymptotic distribution \eqref{eq:adj_rem_beta} under ReM depends mainly on the $\varepsilon$ component. The coefficient of $\varepsilon$ attains its minimal value at $\tilde{\bm{\rcoef}}_{\res}$. 
Ignoring the $L_{K,a}$ component, $\tilde{\bm{\rcoef}}_{\res}$ gives the $\mathcal{S}$-optimal adjusted estimator. However, this result is not useful because it is infeasible for the analyzer to consistently estimate $\tilde{\bm{\rcoef}}_{\res}$ due to the incomplete information of the covariates  in the design.

Third, we can still use $\hat{\tau}(\tilde{\bm{\beta}}_1, \tilde{\bm{\beta}}_0)$ as a convenient adjusted estimator because we can easily obtain it via the OLS. Not surprisingly, this estimator is not $\mathcal{S}$-optimal in general and it can be even worse than $\hat{\tau}$ under ReM. 
When $a$  is small, the $\varepsilon$ components are the dominating terms in 
their asymptotic distributions under ReM. Therefore, we compare the coefficients of $\varepsilon$.

\begin{theorem}\label{thm:gen_ana_enough}
Under ReM and Condition \ref{con:fp}, 
the squared coefficient of $\varepsilon$ is $V_{\tau\tau}
\left(1-R^2_{\tau,\bm{w}}\right)\{1 - R^2_{\tau,\bm{x}}(\tilde{\bm{\beta}}_1, \tilde{\bm{\beta}}_0)\}$ in the asymptotic distribution \eqref{eq:adj_rem} of $\hat{\tau}(\tilde{\bm{\beta}}_1, \tilde{\bm{\beta}}_0)$, and is $V_{\tau\tau}\left(1 - R^2_{\tau,\bm{x}}\right)$ in the asymptotic distribution \eqref{eq:diff_rerand} of $\hat{\tau}$. The former is smaller than or equal to the latter if and only if 
\begin{align}\label{eq:cond_general}
 R^2_{\tau,\bm{w}}+ (1-R^2_{\tau,\bm{w}})\cdot R^2_{\tau,\bm{x}}(\tilde{\bm{\beta}}_1, \tilde{\bm{\beta}}_0)\geq  R^2_{\tau,\bm{x}} 
\end{align}
\end{theorem}

\begin{rmk}
A sufficient condition for \eqref{eq:cond_general} is $R^2_{\tau,\bm{w}} \geq R^2_{\tau,\bm{x}}$, which holds under Condition \ref{con:ana_more}. 
Recall that $R^2_{\tau,\bm{x}}$ in \eqref{eq:R2_tau_x} measures the covariate information of the designer, and $R^2_{\tau,\bm{w}}$ in \eqref{eq:R2_mid_w} measures the covariate information of the analyzer. From Theorem \ref{thm:gen_ana_enough}, when the analyzer has more covariate information, $\hat{\tau}(\tilde{\bm{\beta}}_1, \tilde{\bm{\beta}}_0)$ is more precise than $\hat{\tau}$ if  the threshold $a$ for ReM is small. 
\end{rmk}

\begin{rmk}\label{rmk:reg_worse}
A counterexample for \eqref{eq:cond_general} is that the finite population partial covariance between $Y(z)$ and $\bm{w}$ given $\bm{x}$ is zero for $z=0,1$. In this case,  the squared coefficient of $\varepsilon$ 
for $\hat{\tau}(\tilde{\bm{\beta}}_1, \tilde{\bm{\beta}}_0)$ is larger than or equal to that for $\hat{\tau}$. 
Intuitively, 
this is because 
the covariates in the analysis are unrelated to the potential outcomes after adjusting for the covariates in the design and using them only introduces additional variability. 
In the extreme case where $\bm{x}$ can linearly represent $Y(1)$ and $Y(0)$, 
the squared coefficient of $\varepsilon$ for $\hat{\tau}$ is zero, 
while that for $\hat{\tau}(\tilde{\bm{\beta}}_1, \tilde{\bm{\beta}}_0)$ is generally positive. See the Supplementary Material for more details.
\end{rmk}

Below we use a numerical example to illustrate the results above. It shows that   $\hat{\tau}(\tilde{\bm{\beta}}_1, \tilde{\bm{\beta}}_0)$ can be superior or inferior to $\hat{\tau}$.

\begin{example}\label{eg:sampling}
We choose $n=1000$, $r_1=r_0=0.5$, and generate the covariates and the potential outcomes using i.i.d. samples from the following model: 
\begin{align}\label{eq:model}
x,\eta, \delta \stackrel{\text{i.i.d.}}{\sim}  \mathcal{N}(0,1), \quad 
w = x + \eta, \quad 
Y(0) = 2x + \rho \eta + (1-\rho^2)^{1/2} \delta, \quad 
Y(1) = Y(0) + 1.
\end{align}
Once generated, the covariates and potential outcomes are all fixed. 
We use ReM based on the covariate $x$, and choose the threshold $a$ to be the 0.001th quantile of the $\chi^2_1$ random variable. 
We then use regression adjustment based on the covariate $w$. 
Figure \ref{fig:hist}(a) shows the histograms of $\hat{\tau}(\tilde{\bm{\beta}}_1, \tilde{\bm{\beta}}_0)$ and $\hat{\tau}$ under 
ReM when $\rho=0.9$.  In this case, regression adjustment increases the sampling precision. Figures \ref{fig:hist}(b) shows the histograms when $\rho=0$. In this case, regression adjustment decreases the sampling precision. The case with $\rho = 0$ reflects the scenario that the designer only gives a covariate with measurement error to the analyzer, possibly due to some privacy consideration. 
\end{example}

\begin{figure}[ht]
	\centering
	\begin{subfigure}{.5\textwidth}
		\centering
		\includegraphics[width=0.7\linewidth]{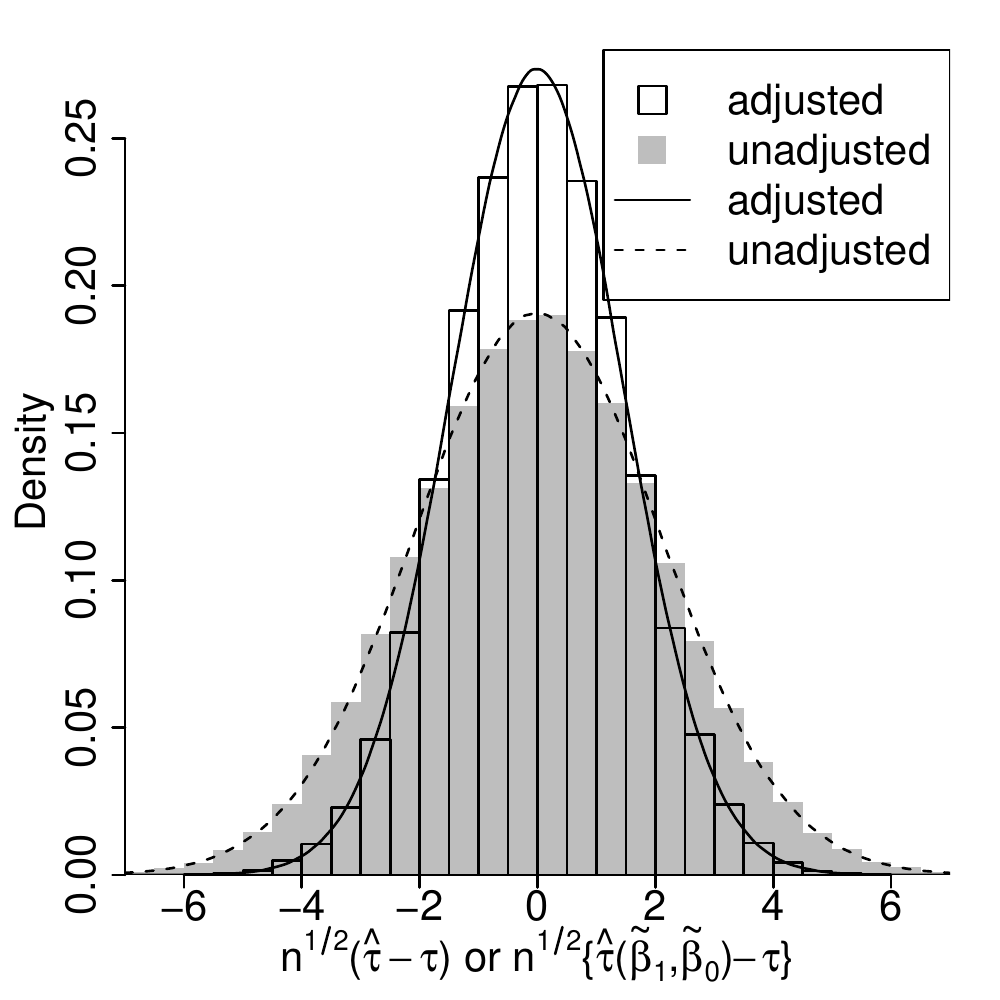}
		\caption{\centering $\rho=0.9$}\label{rho09}
	\end{subfigure}%
	\begin{subfigure}{.5\textwidth}
		\centering
				\includegraphics[width=0.7\linewidth]{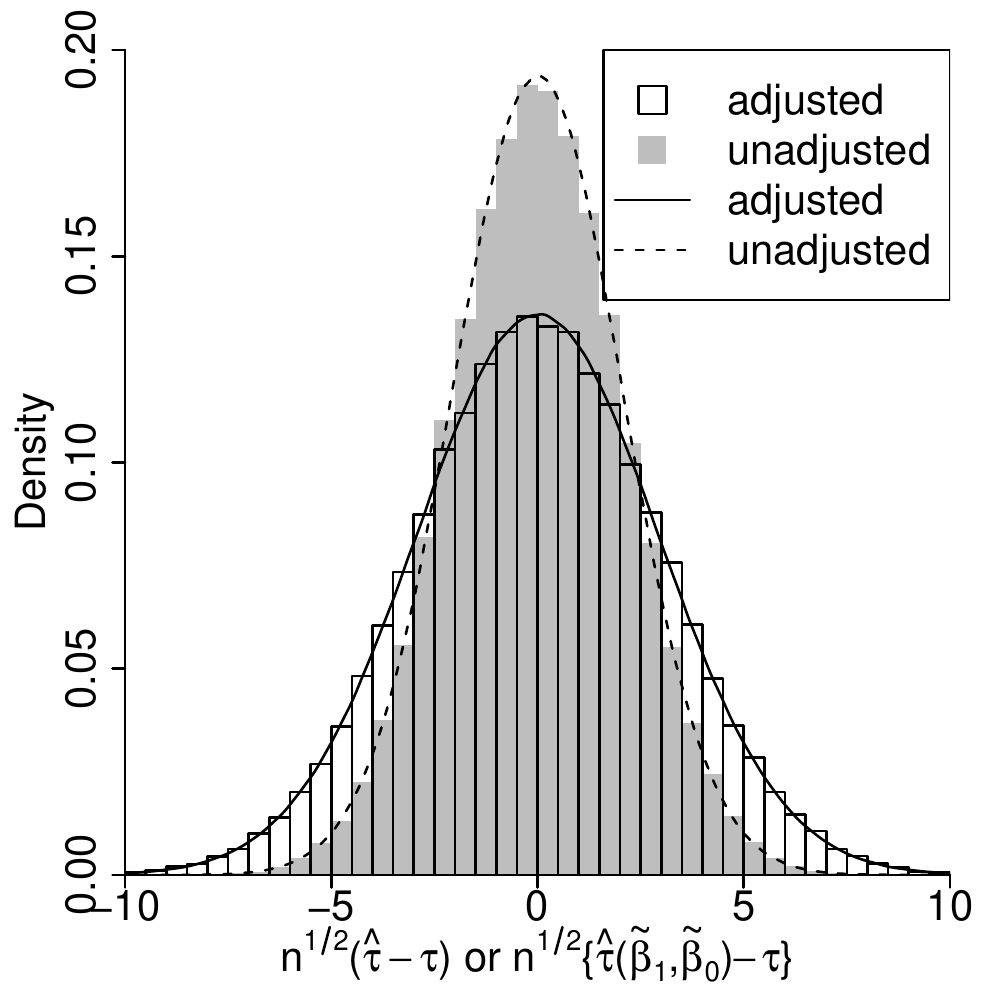}
		\caption{$\rho = 0$} \label{rho01}
	\end{subfigure}
	\caption{Histograms of  
		$n^{1/2}\{\hat{\tau}(\tilde{\bm{\beta}}_1, \tilde{\bm{\beta}}_0)-\tau\}$ and 
	 $n^{1/2}(\hat{\tau}-\tau)$ under ReM based on $10^5$ simulated treatment assignments.
	}
	\label{fig:hist}
\end{figure}

\section{Estimating the sampling distributions of the estimators}\label{sec::estimatedprecision}

Sections \ref{sec:sampling_dist} and \ref{sec:opt_rem}  focused on the asymptotic distribution of the adjusted estimator and discussed optimal choices of the coefficients. In practice, we usually report the uncertainty of estimators in terms of confidence intervals when conducting frequentist inference. Confidence intervals are related to the quantile ranges of the estimated distributions of the corresponding estimators. Therefore, compared to $\mathcal{S}$-optimality, a more practical definition of the optimal adjusted estimator should be based on the quantile ranges of the estimated distributions. This subtle issue does not exist in many other statistical inference problems, because usually consistent estimators exist for the true asymptotic distributions of the estimators. For example, in standard statistical problems, we can consistently estimate the variance of a Gaussian limiting distribution. Because of the possible miscommunication between the designer and analyzer, the analyzer may not be able to estimate all quantities based on the observed data in general. This is a feature of our framework.

In this section, we discuss the estimation of the sampling distributions for fixed $(\bm{\beta}_1, \bm{\beta}_0)$. 
In the next section, we will discuss the optimal choice of these coefficients.

\subsection{The analyzer knows all the information in the design}\label{sec:est_dist_rem_ana_know_all}

We first consider the scenario under which the analyzer knows all the information of the designer.

\begin{condition}\label{con:ana_more_bal_criterion}
The analyzer 
knows all the information in the design, 
including the covariates
$\bm{x}$ and the threshold $a$ for ReM.
\end{condition}
Condition \ref{con:ana_more_bal_criterion} implies Condition \ref{con:ana_more}. 
However, Condition \ref{con:ana_more} does not imply Condition \ref{con:ana_more_bal_criterion}, because the analyzer may not know which covariates are used in the design or which threshold $a$ is chosen for ReM. 
From Theorem \ref{thm:adj_rem}, 
the asymptotic distribution \eqref{eq:adj_rem} of $\hat{\tau}(\bm{\beta}_1, \bm{\beta}_0)$ under ReM depends on 
$V_{\tau\tau}(\bm{\beta}_1, \bm{\beta}_0)$ and $R^2_{\tau, \bm{x}}(\bm{\beta}_{1}, \bm{\beta}_0)$. Under treatment arm $z$ $(z=0,1)$, let
$s^2_{Y(z;\bm{\beta}_z)}$, $\bm{s}_{Y(z;\bm{\beta}_z),\bm{w}}$
and $\bm{s}_{Y(z;\bm{\beta}_z),\bm{x}}$
be the sample variance and covariances for the ``adjusted" observed  outcome $Y_i - \bm{\beta}_z' \bm{w}_{i}$, covariates in the analysis $\bm{w}_i$, and covariates in the design $\bm{x}_i$; let $s^2_{Y(z;\bm{\beta}_z)\mid \bm{x}}$ be the sample variance of the linear projection of $Y_i - \bm{\beta}_z' \bm{w}_{i}$ on $\bm{x}_i$. 
We estimate $V_{\tau\tau}(\bm{\beta}_1, \bm{\beta}_0)$ by 
\begin{eqnarray}
\hat{V}_{\tau\tau}(\bm{\beta}_1, \bm{\beta}_0)  &=&  
r_1^{-1}s^2_{Y(1;\bm{\beta}_1)} + r_0^{-1}s^2_{Y(0;\bm{\beta}_0)} \nonumber  \\
&&- 
(\bm{s}_{Y(1;\bm{\beta}_1),\bm{w}}-
\bm{s}_{Y(0;\bm{\beta}_0),\bm{w}} ) (\bm{S}_{\bm{w}}^{2})^{-1}
(\bm{s}_{\bm{w},Y(1;\bm{\beta}_1)}-\bm{s}_{\bm{w},Y(0;\bm{\beta}_0)} ), \label{eq:V_beta_hat}
\end{eqnarray}
 $R^2_{\tau, \bm{x}}(\bm{\beta}_{1}, \bm{\beta}_0)$ by 
\begin{align}\label{eq:R2_beta_hat}
\hat{R}^2_{\tau, \bm{x}}(\bm{\beta}_{1}, \bm{\beta}_0)
& = 
\hat{V}_{\tau\tau}^{-1}(\bm{\beta}_1, \bm{\beta}_0)
\left\{
r_1^{-1}s_{Y(1;\bm{\beta}_1)\mid \bm{x}}^2+r_0^{-1}s_{Y(0;\bm{\beta}_0)\mid \bm{x}}^2
 \right. 
\nonumber
\\
& 
\quad \quad \quad \quad \quad \quad \quad \quad     
\left. 
 - (\bm{s}_{Y(1;\bm{\beta}_1),\bm{x}}-
\bm{s}_{Y(0;\bm{\beta}_0),\bm{x}} ) (\bm{S}_{\bm{x}}^2 )^{-1}
(\bm{s}_{\bm{x},Y(1;\bm{\beta}_1)}-\bm{s}_{\bm{x},Y(0;\bm{\beta}_0)} )
\vphantom{r_1^{-1}s_{Y(1;\bm{\beta}_1)\mid \bm{x}}^2}
\right\}, 
\end{align}
and the asymptotic distribution of $n^{1/2}\{\hat{\tau}(\bm{\beta}_1, \bm{\beta}_0) - \tau\}$ by 
\begin{equation}
\hat{V}_{\tau\tau}^{1/2}(\bm{\beta}_1, \bm{\beta}_0) 
\left[
\left\{1-\hat{R}^2_{\tau, \bm{x}}(\bm{\beta}_{1}, \bm{\beta}_0)\right\}^{1/2}\cdot \varepsilon + \left\{\hat{R}^2_{\tau, \bm{x}}(\bm{\beta}_{1}, \bm{\beta}_0)\right\}^{1/2}\cdot  L_{K,a} 
\right] . \label{eq::samplingdistribution}
\end{equation}
The estimated distribution \eqref{eq::samplingdistribution} provides a basis for constructing confidence intervals for $\tau$. However,  it is not convenient for theoretical analyses. Below we find the probability limit of \eqref{eq::samplingdistribution}.

\begin{theorem}\label{thm:asymp_behavior_var_ci_ana_more}
Under ReM and Conditions \ref{con:fp} and \ref{con:ana_more_bal_criterion},  
the estimated distribution of $\hat{\tau}(\bm{\beta}_1, \bm{\beta}_0)$ in \eqref{eq::samplingdistribution} has the same limit as 
\begin{eqnarray}
&&\left\{
V_{\tau\tau}(1-R^2_{\tau, \bm{w}})
+
S^2_{\tau\setminus \bm{w}}
+ 
(r_1r_0)^{-1}(\bm{\rcoef}-\tilde{\bm{\rcoef}})'
\bm{S}^2_{\bm{w}\setminus \bm{x}}
	(\bm{\rcoef}-\tilde{\bm{\rcoef}})
\right\}^{1/2} \cdot \varepsilon    \nonumber \\
&& + 
\left\{ 
(r_1r_0)^{-1}
(\bm{\rcoef}-\tilde{\bm{\rcoef}})'
\bm{S}^2_{\bm{w}\mid \bm{x}}
(\bm{\rcoef}-\tilde{\bm{\rcoef}})
\right\}^{1/2}\cdot L_{K,a}.\label{eq::estimateddist4}
\end{eqnarray}
\end{theorem}

From Corollary \ref{cor:analysis_more} and Theorem \ref{thm:asymp_behavior_var_ci_ana_more}, \eqref{eq::estimateddist4} differs from the true asymptotic distribution \eqref{eq::corollary3}  in $S^2_{\tau\setminus \bm{w}}$. We can not estimate $S^2_{\tau\setminus \bm{w}}$ consistently using the observed data. Consequently, the probability limit has wider quantile ranges than the true asymptotic distribution, which results in conservative confidence intervals. This kind of conservativeness is a feature of finite population causal inference known ever since \citet{Neyman:1923}'s seminal work. See the Supplementary Material for a rigorous proof of the conservativeness.
We  discuss two special cases of Theorem \ref{thm:asymp_behavior_var_ci_ana_more} below. 

\subsubsection{Special case: regression adjustment under the CRE}

Again, the CRE is ReM with $a=\infty$ and $\bm{x} = \emptyset$.  
Condition \ref{con:ana_more_bal_criterion} holds automatically under the CRE. 
Theorem \ref{thm:asymp_behavior_var_ci_ana_more} immediately implies 
the following result.

\begin{corollary}\label{corr:infer_eff_cre}
Under the CRE and Condition \ref{con:fp},  the estimated distribution of $\hat{\tau}(\bm{\beta}_1, \bm{\beta}_0)$ in \eqref{eq::samplingdistribution} has the same limit as 
\begin{align}\label{eq:reg_cre_est}
\left\{ 
V_{\tau\tau}(1-R^2_{\tau, \bm{w}}) 
+
S^2_{\tau\setminus \bm{w}}
+ 
(r_1r_0)^{-1}(\bm{\gamma}-\tilde{\bm{\gamma}})'
\bm{S}^2_{\bm{w}}
	(\bm{\gamma}-\tilde{\bm{\gamma}})
\right\}^{1/2} \cdot  \varepsilon.
\end{align}
\end{corollary}

\subsubsection{Special case: no covariate adjustment under ReM}
Using Theorem \ref{thm:asymp_behavior_var_ci_ana_more} with $\bm{\beta}_1 = \bm{\beta}_0 = \bm{0}$, 
we can immediately obtain the probability limit of the estimated distribution of $\hat{\tau}$ under ReM. 
\begin{corollary}\label{corr:est_dist_nonadjust}
Under ReM and Conditions \ref{con:fp} and \ref{con:ana_more_bal_criterion}, the estimated distribution of $\hat{\tau}$ in \eqref{eq::samplingdistribution} has the same limit as 
\begin{align}\label{eq:est_unadj_ana_more}
\left\{V_{\tau\tau}(1-R^2_{\tau,\bm{x}}) + S^2_{\tau\setminus \bm{w}} \right\}^{1/2} \cdot \varepsilon + 
\left( V_{\tau\tau} R^2_{\tau,\bm{x}} \right)^{1/2} \cdot L_{K,a}. 
\end{align}
\end{corollary}

\subsection{General scenarios with partial knowledge of the design}\label{sec:est_dist_rem_general}

We then consider scenarios without Condition \ref{con:ana_more_bal_criterion}. The analyzer either does not have all the covariate information used in the design or does not know the balance criterion for ReM. 
We can still estimate $V_{\tau\tau}(\bm{\beta}_1, \bm{\beta}_0)$ by $\hat{V}_{\tau\tau}(\bm{\beta}_1, \bm{\beta}_0)$ in \eqref{eq:V_beta_hat}. 
However, we cannot consistently estimate $R^2_{\tau, \bm{x}}(\bm{\beta}_{1}, \bm{\beta}_0)$ due to incomplete information of the covariates used in the design stage. 
We can underestimate $R^2_{\tau, \bm{x}}(\bm{\beta}_{1}, \bm{\beta}_0)$ by zero,
and then estimate the sampling distribution of $\hat{\tau}(\bm{\beta}_1, \bm{\beta}_0)$ by 
\begin{equation}
\label{eq::incompleteinformation}
   \hat{V}^{1/2}_{\tau\tau}(\bm{\beta}_1, \bm{\beta}_0) \cdot \varepsilon . 
\end{equation}

An important fact is that the lengths of quantile ranges of the asymptotic distribution \eqref{eq:adj_rem} are nonincreasing in $R^2_{\tau, \bm{x}}(\bm{\beta}_{1}, \bm{\beta}_0)$. This fact guarantees that the estimated distribution
\eqref{eq::incompleteinformation} provides a conservative variance estimator of $\hat{\tau}(\bm{\beta}_1, \bm{\beta}_0)$ and conservative confidence intervals for $\tau$. See the Supplementary Material for a rigorous proof of the conservativeness.
Moreover, \eqref{eq::incompleteinformation} equals \eqref{eq::samplingdistribution} with $a=\infty$, the estimated distribution pretending that ReM does not happen in the design stage. 
Consequently, the probability limit of \eqref{eq::incompleteinformation} equals \eqref{eq:reg_cre_est}, as the following theorem states.

\begin{theorem}\label{thm:conf_general}
Under ReM and Condition \ref{con:fp}, 
the estimated distribution of $\hat{\tau}(\bm{\beta}_1, \bm{\beta}_0)$ in \eqref{eq::incompleteinformation} has the same limit as \eqref{eq:reg_cre_est}.
\end{theorem}

From Theorems \ref{thm:adj_rem} and \ref{thm:conf_general}, under ReM, the limit of the estimated distribution of $\hat{\tau}(\bm{\beta}_1, \bm{\beta}_0)$ in \eqref{eq::incompleteinformation} differs from the corresponding true asymptotic distribution in $S^2_{\tau\setminus \bm{w}}$ and $R^2_{\tau,\bm{x}}(\bm{\beta}_1, \bm{\beta}_0)$. Neither $S^2_{\tau\setminus \bm{w}}$ nor $R^2_{\tau,\bm{x}}(\bm{\beta}_1, \bm{\beta}_0)$ have consistent estimators based on the observed data. The difficulty in estimating  $S^2_{\tau\setminus \bm{w}}$ comes from the fact that for each unit we can observe  at most one potential outcome, but the difficulty in estimating $R^2_{\tau,\bm{x}}(\bm{\beta}_1, \bm{\beta}_0)$ comes from the incomplete information of the design.

Compared to \eqref{eq::samplingdistribution}, 
the estimated distribution \eqref{eq::incompleteinformation} has unnecessarily wider quantile ranges due to the lack of information for ReM, including $R^2_{\tau, \bm{x}}(\bm{\beta}_{1}, \bm{\beta}_0), K$ and $a$. 
In  \eqref{eq::incompleteinformation}, we conduct conservative inference and consider the worse-case scenario, which is the CRE with $\bm{x} = \emptyset$ and $a = \infty$. 
In practice, we can also conduct sensitivity analysis and check how the conclusions change as  $R^2_{\tau, \bm{x}}(\bm{\beta}_{1}, \bm{\beta}_0), K$ and $a$ vary. Without additional information, we still use \eqref{eq::incompleteinformation} to construct confidence intervals in the next section. 

\section{
Optimal adjustment based on the estimated distribution
}\label{sec::Coptimalreg}

$\mathcal{S}$-optimality is based on the uncertainty of the sampling distribution. As discussed in Section \ref{sec::estimatedprecision}, we may not have consistent estimators for the sampling distributions, and often report conservative confidence intervals based on the estimated distributions. Because the lengths of confidence intervals provide important measures of uncertainty in frequentist inference, we focus on the optimal choice of $(\bm{\beta}_1, \bm{\beta}_0)$ based on the estimated distributions proposed in Section \ref{sec::estimatedprecision}.

Under ReM, among the estimators in \eqref{eq:reg}, a choice of $(\bm{\beta}_1, \bm{\beta}_0)$ is optimal  in terms of the estimated precision if the estimated distribution based on \eqref{eq::samplingdistribution} or \eqref{eq::incompleteinformation} has the
shortest asymptotic $1-\alpha$  quantile range for any $\alpha \in (0,1).$ The corresponding estimated distribution then must have the smallest variance. Conversely, if the estimated distributions are Gaussian, then the one with the smallest estimated variance is optimal  in terms of the estimated precision.

Apparently, the optimal adjusted estimator  in terms of the estimated precision
depends on the approach to constructing confidence intervals. 
In the ideal case, we want the probability limits of the estimated distributions to be identical to the true sampling distributions. Section \ref{sec::estimatedprecision}, however, shows that this is generally impossible due to treatment effect heterogeneity or the information only known to the designer. 
In some cases, the confidence intervals based on \eqref{eq::samplingdistribution} and \eqref{eq::incompleteinformation} are asymptotically exact given the analyzer's observed information. For example, if 
the residual from the linear projection of individual treatment effect on covariates $\bm{w}$ is constant
across all units,  
or equivalently $S^2_{\tau \setminus \bm{w}}=0$,  
\eqref{eq::samplingdistribution} provides asymptotically exact confidence intervals; 
if further the designer conducts CRE (i.e., ReM with $\bm{x} = \emptyset$ or $a =\infty$), \eqref{eq::incompleteinformation} provides asymptotically exact confidence intervals. In general, these confidence intervals can be conservative, and can be improved \citep{aronow2014, fogarty2018regression, ding2019decomposition}. 
We focus on \eqref{eq::samplingdistribution} and \eqref{eq::incompleteinformation} for their simplicity,
and divide this section to two subsections in parallel with Section \ref{sec::estimatedprecision}.

\subsection{When the analyzer knows all the information in the design stage}\label{sec:inference_ana_more}

From Theorem \ref{thm:asymp_behavior_var_ci_ana_more}, we can obtain the optimal adjusted estimator in terms of the estimated precision.  

\begin{corollary}\label{cor:Copt_ana_know_all}
Under ReM and Conditions \ref{con:fp} and \ref{con:ana_more_bal_criterion}, based on Theorem \ref{thm:asymp_behavior_var_ci_ana_more}, 
the optimal adjusted estimator in terms of the estimated precision
is attainable when $\bm{\rcoef}=\tilde{\bm{\rcoef}}$ or 
$r_0\bm{\beta}_1+r_1\bm{\beta}_0 = r_0 \tilde{\bm{\beta}}_1+r_1 \tilde{\bm{\beta}}_0 ,$  with the estimated distribution having the same limit as 
\begin{eqnarray}
\label{eq::estimated-optimal}
\{
V_{\tau\tau}(1-R^2_{\tau, \bm{w}})
+
S^2_{\tau\setminus \bm{w}}
\}^{1/2} \cdot \varepsilon.
\end{eqnarray}  
\end{corollary}

Again, from Proposition \ref{prop:three_beta_tilde}, an optimal choice is $(  \bm{\beta}_1, \bm{\beta}_0) = ( \tilde{\bm{\beta}}_1, \tilde{\bm{\beta}}_0 )$ under Condition \ref{con:ana_more_bal_criterion}. Importantly, Corollary \ref{cor:Copt_ana_know_all} does not depend on covariates $\bm{x}$ and the threshold $a$ used in ReM, and thus also holds under the CRE.

\subsection{General scenarios without Condition \ref{con:ana_more_bal_criterion}}\label{sec:Copt_general_rem}
From Theorem \ref{thm:conf_general}, we can obtain the optimal adjusted estimator in terms of the estimated precision. 
 
\begin{corollary}\label{cor:opt_conf_general}
Under ReM and Condition \ref{con:fp},  based on Theorem \ref{thm:conf_general}, the optimal adjusted estimator in terms of the estimated precision
is attainable when $\bm{\rcoef}=\tilde{\bm{\rcoef}}$ or 
$r_0\bm{\beta}_1+r_1\bm{\beta}_0 = r_0 \tilde{\bm{\beta}}_1+r_1 \tilde{\bm{\beta}}_0 ,$  with  the estimated distribution having the same limit as \eqref{eq::estimated-optimal}. 
\end{corollary}

The optimal estimators in terms of the estimated precision are identical in Corollaries \ref{cor:Copt_ana_know_all} and \ref{cor:opt_conf_general}, no matter whether the analyzer knows all the information in the design or not. 
The optimal adjustment $\hat{\tau}(\tilde{\bm{\beta}}_1, \tilde{\bm{\beta}}_0)$ can never hurt the estimated precision. In contrast, Section \ref{sec::generalS} shows that  $\hat{\tau}(\tilde{\bm{\beta}}_1, \tilde{\bm{\beta}}_0)$ may hurt the sampling precision in general. This is an important difference between the two notations of optimality in terms of the sampling precision and estimated precision.

Below we give some intuition for Corollary \ref{cor:opt_conf_general}. 
Under general scenarios without Condition \ref{con:ana_more_bal_criterion}, 
the analyzer does not know the information of the design. 
S/he pretends that the design was a CRE, and estimates the sampling distributions of the estimators under the CRE. Luckily, the resulting confidence intervals are still conservative. 
Dropping the term $S^2_{\tau\setminus \bm{w}}$, the estimated distribution converges to the sampling distribution under the CRE. Based on the discussion of $\mathcal{S}$-optimality under the CRE in Section \ref{sec:sopt_cre}, 
the adjusted estimator 
$\hat{\tau}(\tilde{\bm{\beta}}_1, \tilde{\bm{\beta}}_0)$  is also optimal in terms of the estimated precision.

\begin{table}[t]
\centering
\caption{Sampling standard errors (s.e.) and average estimated standard errors multiplied by $n^{1/2}$ under ReM based on $10^5$ simulated treatment assignments.}\label{eq:table:est}
\begin{tabular}{cccccc}
\toprule
Estimator & \multicolumn{2}{c}{$\rho=0.9$}  & & \multicolumn{2}{c}{$\rho=0$}\\
\cline{2-3} \cline{5-6}
&  sampling s.e. &  average estimated s.e.  & &   sampling s.d.  &  average estimated s.e.\\ 
\midrule
$\hat{\tau}(\tilde{\bm{\beta}}_1, \tilde{\bm{\beta}}_0)$
 & 1.47 & 1.86 & & 2.95 & 3.61  \\
$\hat{\tau}$ 
 & 2.10 & 4.69& & 2.07 & 4.71   \\
\bottomrule
\end{tabular}
\end{table} 

\addtocounter{example}{-1}
\begin{example}[continued]
We revisit Example \ref{eg:sampling}, and study the estimated distributions of $\hat{\tau}(\tilde{\bm{\beta}}_1, \tilde{\bm{\beta}}_0)$ and $\hat{\tau}$ under ReM. 
Because the estimated distributions are Gaussian from Theorem \ref{thm:conf_general}, 
it suffices to report the estimated standard errors.  
Table \ref{eq:table:est} shows the sampling standard errors and the average estimated standard errors.
On average, $\hat{\tau}(\tilde{\bm{\beta}}_1, \tilde{\bm{\beta}}_0)$ results in shorter confidence intervals than $\hat{\tau}$ when $\rho=0$ or $\rho=0.9$. Interestingly, when $\rho=0$, $\hat{\tau}(\tilde{\bm{\beta}}_1, \tilde{\bm{\beta}}_0)$ has less sampling precision as Figure \ref{fig:hist}(b) and Table \ref{eq:table:est} show. 
\end{example}

\section{Gains from the analyzer and the designer}\label{sec:gains}

In the design stage, we can use the CRE or ReM. In the analysis stage, we can use the unadjusted estimator $\hat{\tau}$ or the adjusted estimator $\hat{\tau}(\tilde{\bm{\beta}}_1, \tilde{\bm{\beta}}_0)$. 
Based on the results in previous sections, we now study the additional gains of the designer and analyzer  in 
the sampling precision, the estimated precision, and the coverage probability.

\subsection{Sampling precision}

We first study the additional gains in the sampling precision from the analyzer and the designer, respectively. 
We measure the additional gain of the analyzer by comparing the asymptotic distributions of $\hat{\tau}$ and 
$\hat{\tau}(\tilde{\bm{\beta}}_1, \tilde{\bm{\beta}}_0)$ under ReM. 
We measure the additional gain of the designer by comparing the asymptotic distributions of $\hat{\tau}(\tilde{\bm{\beta}}_1, \tilde{\bm{\beta}}_0)$ under the CRE and ReM. 
Similar to Section \ref{sec:opt_rem}, 
we consider different scenarios based
on the relative amount of covariate information for the analyzer and the designer.
Let $v_{K,a} = P(\chi^2_{K+2}\leq a)/P(\chi^2_{K}\leq a)\leq 1$ be the variance of $L_{k,a}$ \citep{morgan2012rerandomization}, and $q_{1 - \alpha/2}(\rho^2)$ be the $(1-\alpha / 2)$th quantile of $(1-\rho^2)^{1/2}\cdot\varepsilon+|\rho| \cdot L_{K,a}$.

\subsubsection{The analyzer has no less covariate information than the designer}\label{sec:gains_samp_ana_more}

First, we measure the additional gain of the analyzer.

\begin{corollary}\label{cor:ana_more_opt_diff}
	Under ReM and Conditions \ref{con:fp} and \ref{con:ana_more}, compare $\hat{\tau}(\tilde{\bm{\beta}}_1, \tilde{\bm{\beta}}_0)$ to $\hat{\tau}$.
	The percentage reduction in the asymptotic variance is 
	$
	\left\{  R^2_{\tau,\bm{w}}-(1-v_{K,a})R^2_{\tau,\bm{x}}  \right\} / \left\{ 1-(1-v_{K,a})R^2_{\tau,\bm{x}} \right\} .
	$
	For any $\alpha\in (0,1),$ the percentage reduction in the length of the asymptotic $1-\alpha$  quantile range is 
	$
	1 - \left(1-R^2_{\tau, \bm{w}} \right)^{1/2} \cdot  q_{1-\alpha/2}(0) / q_{1-\alpha/2}(R^2_{\tau,\bm{x}}) . 
	$
	Both percentage reductions are nonnegative and nondecreasing in $R^2_{\tau,\bm{w}}$. 
\end{corollary}

From Corollary \ref{cor:ana_more_opt_diff},  the gain from the analyzer through regression adjustment is nondecreasing in the analyzer's covariate information. Both percentage reductions in Corollary \ref{cor:ana_more_opt_diff} converge to 1 as $R^2_{\tau,\bm{w}}$ converges to $ 1$.

Second, we measure the additional gain of the designer. Section \ref{sec:sopt_cre} demonstrates that $\hat{\tau}(\tilde{\bm{\beta}}_1, \tilde{\bm{\beta}}_0)$ has the same asymptotic distribution under the CRE and ReM. Therefore, under Condition \ref{con:ana_more}, the gain from the designer is zero. Nevertheless, this also implies that using ReM in the design will not hurt the sampling precision of $\hat{\tau}(\tilde{\bm{\beta}}_1, \tilde{\bm{\beta}}_0)$. Moreover, under ReM, we can use the additional covariate information, in the same way as in the CRE, to improve the sampling precision. 

\subsubsection{The analyzer has no more covariate information than the designer}\label{sec:gains_samp_ana_less}

First, we measure the additional gain of the analyzer.
Both the asymptotic distributions of $\hat{\tau}(\tilde{\bm{\beta}}_1, \tilde{\bm{\beta}}_0)$ and $\hat{\tau}$ are linear combinations of $\varepsilon$ and $L_{K,a}$. The coefficients of $\varepsilon$ are identical, but the coefficient of $L_{K,a}$ for $\hat{\tau}(\tilde{\bm{\beta}}_1, \tilde{\bm{\beta}}_0)$ is smaller than that for $\hat{\tau}.$

\begin{corollary}\label{cor:ana_less_reg_diff_comp}
	Under ReM and Conditions \ref{con:fp} and \ref{con:des_more}, compare $\hat{\tau}(\tilde{\bm{\beta}}_1, \tilde{\bm{\beta}}_0)$ to $\hat{\tau}$. 
	The percentage reduction in the asymptotic variance is
	$
	v_{K,a}R^2_{\tau, \bm{w}} / \left\{ 1-(1-v_{K,a})R^2_{\tau,\bm{x}} \right\} . 
	$
	For any $\alpha\in (0,1)$, the percentage reduction in the length of the asymptotic $1-\alpha$ quantile range is
	$
	1- \left(    1-R^2_{\tau, \bm{w}}  \right )  ^{1/2} \cdot 
	q_{1-\alpha/2}(\rho^2_{\tau, \bm{x}\setminus \bm{w}}) / q_{1-\alpha/2}(R^2_{\tau, \bm{x}}) . 
	$
	Both percentage reductions are nonnegative and nondecreasing in $R^2_{\tau,\bm{w}}$. 
\end{corollary}

From Corollary \ref{cor:ana_less_reg_diff_comp}, the improvement from regression adjustment is nondecreasing in the analyzer's covariate information. However, this improvment is small when the designer uses a small threshold $a$ for ReM. Both percentage reductions in Corollary \ref{cor:ana_less_reg_diff_comp} converge to 0 as $a$ converges to $ 0$. Intuitively, when the designer uses ReM with a small threshold, s/he has used more covariate information thoroughly in the design, and thus the analyzer has only a small additional gain through regression adjustment. 

Second, we measure the additional gain of the designer.

\begin{corollary}\label{cor:ana_less_CRE_ReM_comp}
	Under Conditions \ref{con:fp} and \ref{con:des_more}, compare $\hat{\tau}(\tilde{\bm{\beta}}_1, \tilde{\bm{\beta}}_0)$ under ReM to
	that under the CRE. The percentage reduction in the asymptotic variance  is
	$
	(1-v_{K,a})\rho^2_{\tau, \bm{x}\setminus \bm{w}}.
	$
	For any $\alpha \in (0,1)$, the percentage reduction in the length of the asymptotic $1-\alpha$ quantile range is
	$
	1-   q_{1-\alpha/2}(\rho^2_{\tau, \bm{x}\setminus \bm{w}}) / q_{1-\alpha/2}(0) .
	$
	Both percentage reductions are nonnegative and nondecreasing in $R^2_{\tau, \bm{x}}$.
\end{corollary}

From Corollary \ref{cor:ana_less_CRE_ReM_comp}, the gain from the designer through ReM is nondecreasing in the designer's covariate information. The gain from the designer is substantial when $R^2_{\tau, \bm{x}}$ is large and the threshold for ReM is small.
Both percentage reductions in Corollary \ref{cor:ana_less_CRE_ReM_comp} converge to 1 as $R^2_{\tau, \bm{x}}\rightarrow 1$ and $a\rightarrow 0$.

\begin{rmk}\label{rmk:gain_same_covariate}
Consider the special case where the designer and the analyzer have the same covariates ($\bm{x} = \bm{w}$). The additional gain from the analyzer is small given that the designer uses ReM with a small threshold $a$, and so is the additional gain from the designer given that the analyzer uses the $\mathcal{S}$-optimal $\hat{\tau}(\tilde{\bm{\beta}}_1, \tilde{\bm{\beta}}_0)$. 
\end{rmk}

\subsubsection{General scenarios}\label{sec:addition_general}

The complexity discussed in Section \ref{sec::generalS} makes it difficult to evaluate the additional gains from the analyzer and the designer. Given any adjusted estimator, the designer can always use ReM to reduce the asymptotic variance and the lengths of asymptotic quantile ranges \citep{asymrerand2106}. For the analyzer, in general, the performance of regression adjustment under ReM depends on the covariates used in the design, and thus the analyzer does not know the optimal adjusted estimator among \eqref{eq:reg}. For instance, without the covariate information used in the design, the analyzer is not sure whether $\hat{\tau}(\tilde{\bm{\beta}}_1, \tilde{\bm{\beta}}_0)$ has smaller asymptotic variance than $\hat{\tau}$. Example \ref{eg:sampling} shows two cases where $\hat{\tau}(\tilde{\bm{\beta}}_1, \tilde{\bm{\beta}}_0)$ increases and decreases the sampling precision, respectively.

\subsection{Estimated precision}\label{sec:gain_est_precision}

We then study the additional gains in the asymptotic estimated precision from the analyzer and the designer, respectively. 
We measure the additional gain of the analyzer by comparing the estimated distributions of $\hat{\tau}(\tilde{\bm{\beta}}_1, \tilde{\bm{\beta}}_0)$ and  $\hat{\tau}$ under ReM.
We measure the additional gain of the designer by comparing the estimated distributions of $\hat{\tau}(\tilde{\bm{\beta}}_1, \tilde{\bm{\beta}}_0)$ under the CRE and ReM.
Similar to Section \ref{sec::Coptimalreg}, we consider two scenarios based on whether the analyzer has full knowledge of the design or not.
Let 
$
\kappa=1+V^{-1}_{\tau\tau}S^2_{\tau\setminus \bm{w}}\geq 1,
$
which reduces to 1 when $S^2_{\tau \setminus \bm{w} } = 0$, i.e., the ``adjusted'' individual treatment effect $\tau_i(\tilde{\bm{\beta}}_1, \tilde{\bm{\beta}}_0) =  Y_i(1; \tilde{\bm{\beta}}_1) - Y_i(0; \tilde{\bm{\beta}}_0)  $ is constant for all units.

\subsubsection{When the analyzer knows all the information in the design stage}\label{sec:gain_est_analyzer_know_all}


First, we measure the additional gain of the analyzer.   

\begin{corollary}\label{cor:infer_ana_more_role_ana}
	Under ReM and Conditions \ref{con:fp} and \ref{con:ana_more_bal_criterion}, compare the probability limit of the estimated distribution of $\hat{\tau}(\tilde{\bm{\beta}}_1, \tilde{\bm{\beta}}_0)$ to that of $\hat{\tau}$ based on \eqref{eq::samplingdistribution}. 
	The percentage reduction in the variance is
	$
	\left\{  R^2_{\tau,\bm{w}} - (1-v_{K,a})R^2_{\tau, \bm{x}}  \right\} / \left\{   \kappa - (1-v_{K,a})R^2_{\tau, \bm{x}} \right\}.
	$
	For $\alpha \in (0,1)$, the percentage reduction in the length of the $1-\alpha$ quantile range is 
	$
	1 - \left( 1-R^2_{\tau, \bm{w}}/\kappa \right)^{1/2}   \cdot 
	q_{1-\alpha/2}(0) / q_{1-\alpha/2}(R^2_{\tau,\bm{x}}/{\kappa}) .
	$
	Both percentage reductions are nonnegative and nondecreasing in $R^2_{\tau,\bm{w}}$. 
\end{corollary}

From Corollary \ref{cor:infer_ana_more_role_ana}, 
the gain from regression adjustment is nondecreasing in the analyzer's covariate information. 
Both percentage reductions in Corollary \ref{cor:infer_ana_more_role_ana} converge to 1 as $R^2_{\tau,\bm{w}}\rightarrow 1$ and $\kappa\rightarrow 1$. 

Second, we measure the additional gain of the designer. From Corollary \ref{cor:Copt_ana_know_all} and the comment after it, the probability limits of
the estimated distributions of $\hat{\tau}(\tilde{\bm{\beta}}_1, \tilde{\bm{\beta}}_0)$ are identical under both designs. 
Therefore, the gain from the designer is zero.

\begin{rmk}\label{rmk:gain_est_same_cov}
	Consider the special case where the analyzer has the same covariate information as the designer and knows the balance criterion in the design.
As discussed above, the designer has no gain.	
	 Based on Corollaries \ref{corr:est_dist_nonadjust} and  \ref{cor:Copt_ana_know_all}, with a small threshold $a$, the estimated distributions of $\hat{\tau}$ and 
	 $\hat{\tau}(\tilde{\bm{\beta}}_1, \tilde{\bm{\beta}}_0)$ 
	 have approximately the same probability limit. Therefore, the analyzer has small additional gain.
\end{rmk}

\subsubsection{General scenarios without Condition \ref{con:ana_more_bal_criterion}}\label{sec:add_gain_est_ana_not_more}

First, 
we measure the additional gain of the analyzer. 

\begin{corollary}\label{cor:infence_inf_reg_diff}
	Under ReM and Condition \ref{con:fp}, compare the probability limit of the estimated distribution of $\hat{\tau}(\tilde{\bm{\beta}}_1, \tilde{\bm{\beta}}_0)$ to that of $\hat{\tau}$ based on \eqref{eq::incompleteinformation}. 
	The percentage reduction in the variance is $R^2_{\tau, \bm{w}}/\kappa.$  For any $\alpha\in(0,1)$, the percentage reduction in the length of the  $1-\alpha$ quantile range is $1-\left( 1-R^2_{\tau, \bm{w}}/\kappa \right)^{1/2} . $ Both percentage reductions are nonnegative and nondecreasing in $R^2_{\tau, \bm{w}}$. 
\end{corollary}

Both percentage reductions in Corollary \ref{cor:infence_inf_reg_diff} converge to 1 as $R^2_{\tau, \bm{w}}\rightarrow 1$ and $\kappa\rightarrow 1$, and they are larger than or equal to those in Corollary \ref{cor:infer_ana_more_role_ana}.

Second, we measure the additional gain of the designer. From Corollary \ref{corr:infer_eff_cre} and Theorem \ref{thm:conf_general}, the estimated distributions of any adjusted estimator in \eqref{eq:reg} have the same probability limit under both designs. Therefore, 
the gain from the designer  
is zero. Nevertheless, using ReM will not hurt the estimated precision of the treatment effect estimators, and we can use covariates in the analysis in the same way as in the CRE to improve the estimated precision.

\subsection{Coverage probabilities}

From Sections \ref{sec::Coptimalreg} and \ref{sec:gain_est_precision},  (a) $\hat{\tau}(\tilde{\bm{\beta}}_1, \tilde{\bm{\beta}}_0)$ is optimal in terms of the estimated precision no matter whether the analyzer knows all the information of the design or not, and (b) the designer provides no gain in the estimated precision of the $\hat{\tau}(\tilde{\bm{\beta}}_1, \tilde{\bm{\beta}}_0)$. 
From (a), under ReM, the analyzer can never increase the asymptotic lengths of the confidence intervals by using $\hat{\tau}(\tilde{\bm{\beta}}_1, \tilde{\bm{\beta}}_0)$ instead of $\hat{\tau}$. Therefore, we do not measure the additional gain of the analyzer in coverage probabilities. From (b), 
the asymptotic lengths of the confidence intervals based on $\hat{\tau}(\tilde{\bm{\beta}}_1, \tilde{\bm{\beta}}_0)$ are the same under the CRE and ReM.   However, we will show shortly that  the designer can help to improve the coverage probabilities of the confidence intervals based on $\hat{\tau}(\tilde{\bm{\beta}}_1, \tilde{\bm{\beta}}_0)$.

\subsubsection{When the analyzer knows all the information in the design stage}
From Sections 
\ref{sec:gains_samp_ana_more} 
and \ref{sec:gain_est_analyzer_know_all}, 
$\hat{\tau}(\tilde{\bm{\beta}}_1, \tilde{\bm{\beta}}_0)$ has the same sampling precision and estimated precision under ReM and the CRE.
Therefore,  the coverage probabilities of the associated confidence intervals  are asymptotically the same under the CRE and ReM.  
This implies that the designer provides no gain for the coverage probabilities of the confidence intervals based on $\hat{\tau}(\tilde{\bm{\beta}}_1, \tilde{\bm{\beta}}_0)$.
We formally state the results as follows. 

\begin{corollary}\label{cor:coverage_ana_more}
	Under Conditions \ref{con:fp} and \ref{con:ana_more_bal_criterion}, compare the confidence intervals based on $\hat{\tau}(\tilde{\bm{\beta}}_1, \tilde{\bm{\beta}}_0)$ and \eqref{eq::samplingdistribution} under the CRE and ReM. 
	Their lengths are asymptotically the same after being scaled by $n^{1/2}$, and they have the same asymptotic coverage probability. 
\end{corollary}


\subsubsection{General scenarios without Condition \ref{con:ana_more_bal_criterion}}

From Sections 
\ref{sec:addition_general}
and 
\ref{sec:add_gain_est_ana_not_more}, 
$\hat{\tau}(\bm{\beta}_1, \bm{\beta}_0)$ in \eqref{eq:reg} with any $\bm{\beta}_1$ and $\bm{\beta}_0$
has better sampling precision under ReM than under the CRE, 
but 
it has the same estimated precision under ReM and the CRE. 
Therefore, 
the confidence intervals based on $\hat{\tau}(\bm{\beta}_1, \bm{\beta}_0)$ under ReM have higher coverage probabilities than that under the CRE. 
We give a formal statement below. 


\begin{corollary}\label{cor:gen_ci_larger_coverage_rem}
	Under Condition \ref{con:fp}, compare the confidence intervals based on $\hat{\tau}(\bm{\beta}_1, \bm{\beta}_0)$ and \eqref{eq::incompleteinformation} under the CRE and ReM. 
	Their lengths are asymptotically the same after being scaled by $n^{1/2}$. However, the asymptotic coverage probability under ReM is larger than or equal to that under the CRE. 
\end{corollary}

In Corollary \ref{cor:gen_ci_larger_coverage_rem}, the confidence intervals under both ReM and the CRE are asymptotically valid and of the same length, 
but the one under ReM has higher coverage probabilities and is more conservative.
In particular,  for any $\alpha \in (0,1)$ and any $(\bm{\beta}_1, \bm{\beta}_0)$, 
as $R^2_{\tau, \bm{x}}(\bm{\beta}_1, \bm{\beta}_0) \rightarrow 1$, $S^2_{\tau \setminus \bm{w}}  \rightarrow 0$ and $a \rightarrow 0$, 
the asymptotic coverage probabilities of the $1-\alpha$ confidence intervals are $1$ and $1-\alpha$ under ReM and the CRE, respectively. 
Corollary \ref{cor:gen_ci_larger_coverage_rem} holds for any adjusted estimator and thus holds for 
$\hat{\tau}(\tilde{\bm{\beta}}_1, \tilde{\bm{\beta}}_0)$. Therefore, under general scenarios without Condition \ref{con:ana_more_bal_criterion}, the designer can provide a substantial gain in coverage probabilities of confidence intervals. This gives another justification for using ReM.

\section{Unification and practical suggestions}
\label{sec::unification}

\subsection{Unification}\label{sec:uni}

In total, there are four combinations in the design and analysis of experiments. Figure \ref{fig:gain} summarizes the sampling distributions and the probability limits of the estimated distributions for all combinations. 

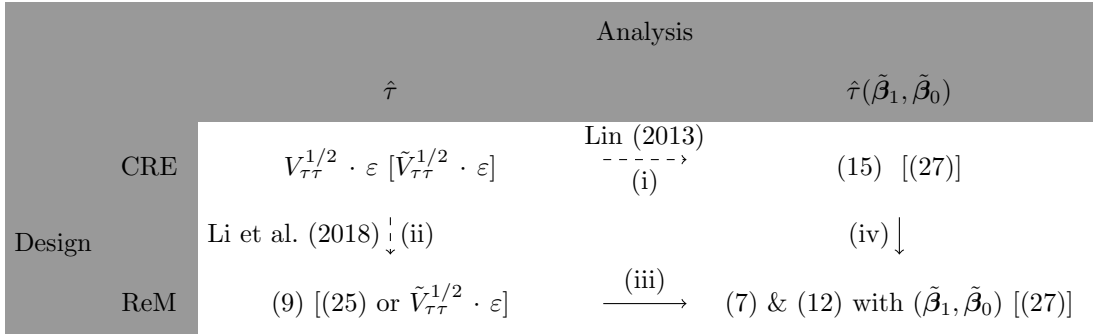
\begin{figure}[t]
	\centering
	\begin{tikzpicture}
	\small
	\matrix (mat) [table]
	{
		& & & Analysis & \\
		& & $\hat{\tau}$ & & $\hat{\tau}(\tilde{\bm{\beta}}_1, \tilde{\bm{\beta}}_0)$  \\
		& CRE & $V_{\tau\tau}^{1/2} \cdot \varepsilon$ [$\tilde{V}_{\tau\tau}^{1/2} \cdot \varepsilon$]  & & \eqref{eq:opt_ana_more}  \  [\eqref{eq::estimated-optimal}]\\
		Design & & &&
		\\
		& ReM & \eqref{eq:diff_rerand} [\eqref{eq:est_unadj_ana_more} or $\tilde{V}_{\tau\tau}^{1/2} \cdot \varepsilon$]  & &   \eqref{eq:adj_rem} \& \eqref{eq:adj_rem_beta} 
		with 
 $(\tilde{\bm{\beta}}_1, \tilde{\bm{\beta}}_0)$
		[\eqref{eq::estimated-optimal}]
		\\
	};
	\begin{scope}[shorten >=7pt,shorten <= 7pt]
	\draw[dashed, ->]  (mat-3-3) edge node[below] {\(\text{(i)}\)} node[above] {\(\text{\citet{lin2013}} \)} (mat-3-5);
	\draw[dashed, ->]  (mat-3-3) edge node[right] {\(\text{(ii)}\)} node[left] {\( \text{\citet{asymrerand2106}} \)}  (mat-5-3);
	\draw[->]  (mat-5-3) edge node[above] {\(\text{(iii)}\)} (mat-5-5);
	\draw[->]  (mat-3-5) edge node[left] {\(\text{(iv)}\)} (mat-5-5);
	\end{scope}
	\end{tikzpicture}
	\caption{
		Design and analysis strategies. 
		The formulas without square brackets correspond to asymptotic distributions, 
		and those with square brackets correspond to probability limits of the estimated distributions with $\tilde{V}_{\tau\tau} = V_{\tau\tau}+S^2_{\tau \setminus \bm{w}}$.  
		The probability limits of the estimated distributions of $\hat{\tau}$ under ReM have two forms, depending on whether Condition \ref{con:ana_more_bal_criterion} holds or not. 
	}\label{fig:gain}
\end{figure}

\citet{Neyman:1923} started the literature by discussing the property of $\hat{\tau}$ under the CRE. \citet{lin2013} showed that $\hat{\tau}(\tilde{\bm{\beta}}_1, \tilde{\bm{\beta}}_0)$ improves $\hat{\tau}$ in terms of the sampling precision and estimated precision under the CRE. Arrow (i) in Figure \ref{fig:gain} illustrates this improvement. \citet{asymrerand2106} showed that ReM improves the CRE in terms of the sampling precision and the estimated precision of $\hat{\tau}$. Arrow (ii) in Figure \ref{fig:gain} illustrates this improvement. Interestingly, $\hat{\tau}(\tilde{\bm{\beta}}_1, \tilde{\bm{\beta}}_0)$ under the CRE and $\hat{\tau}$ under ReM have almost identical asymptotic sampling distributions and estimated distributions, if we use the same sets of covariates and $a\approx 0$ in ReM.

However, both \citet{lin2013} and \citet{asymrerand2106} compared sub-optimal strategies. We evaluated the additional gain from the analyzer given that the designer uses ReM.  Arrow (iii) in Figure \ref{fig:gain} illustrates this improvement. We also evaluated the additional gain from the designer given that the analyzer uses $\hat{\tau}(\tilde{\bm{\beta}}_1, \tilde{\bm{\beta}}_0)$. 
Arrow (iv) in Figure \ref{fig:gain} illustrates this improvement. 
Table \ref{tab:summary} summarizes the results under all scenarios. 
We have the following conclusions. 
\begin{enumerate}[(i)]
	\setlength\itemsep{0em}
	\item
	Compare the analyzer and the designer based on the sampling precision. 
	From the first two rows of Table \ref{tab:summary}, when one has more covariate information than the other, the one with more covariate information provides a substantial additional gain, while the other provides negligible additional gain.

	\item
	Compare the analyzer and the designer based on the estimated precision. 
	From the 6th and 8th
	columns of Table \ref{tab:summary},
	the additional gain from the analyzer can be substantial, 
	while the additional gain from the designer is negligible in general. 
	
	\item
	Consider the special case where the  analyzer has the same covariate information as the designer and knows the balance criterion in the design. 
	From the fourth row of Table \ref{tab:summary},
	the additional gain from either the analyzer or the designer are negligible. 
	
	\item
	From the last row of Table \ref{tab:summary},
	the analyzer may hurt the sampling precision through regression adjustment, but can provide a substantial gain in the estimated precision. 
	The designer can improve sampling precision of any adjusted estimator, and does not hurt the estimated precision. 
	Therefore, although the designer cannot shorten the confidence intervals, s/he can increase the coverage probabilities. 
\end{enumerate}

\begin{table}
	\centering
	\caption{Additional gains.
		In Column 1, A\ $\geq$\ D, A\ $\leq$\ D and A\ $=$\ D represent that the analyzer has 
		no less (i.e., Condition \ref{con:ana_more}), no more (i.e., Condition \ref{con:des_more}) and the same (i.e., both Conditions \ref{con:ana_more} and \ref{con:des_more}) covariate information compared to the designer. 
		Column 2 shows whether the analyzer knows the balance criterion in the design (i.e., Condition \ref{con:ana_more_bal_criterion}).
		Columns 3 and 4 show the optimal coefficients. 
		{\large \smiley} denotes a substantial gain, {\large \neutranie} denotes no gain or negligible gain, and {\large \frownie}  denotes a negative gain. 
	}\label{tab:summary}
	\resizebox{\textwidth}{!}{
		\begin{tabular}{cccccccc}
			\toprule
			Covariates & Balance  & 
			\multicolumn{2}{c}{Optimal adjustment}
			& \multicolumn{2}{c}{Additional gain from analyzer} & \multicolumn{2}{c}{Additional gain from  designer}\\
			information & criterion & Sampling & Estimated & Sampling & Estimated & Sampling & Estimated\\
			\midrule
			A $\geq$ D & (Un-)known & $(\tilde{\bm{\beta}}_1, \tilde{\bm{\beta}}_0)$ & $(\tilde{\bm{\beta}}_1, \tilde{\bm{\beta}}_0)$ & {\ssize \smiley \neutranie } & {\ssize \smiley \neutranie } & {\ssize \neutranie} & {\ssize \neutranie}
			\\
			A $\leq$ D & Unknown & $(\tilde{\bm{\beta}}_1, \tilde{\bm{\beta}}_0)$ & $(\tilde{\bm{\beta}}_1, \tilde{\bm{\beta}}_0)$ & {\ssize \neutranie} & {\ssize \smiley \neutranie}  & {\ssize \smiley \neutranie}  & {\ssize  \neutranie}
			\\
			A $=$ D & Unknown & $(\tilde{\bm{\beta}}_1, \tilde{\bm{\beta}}_0)$ & $(\tilde{\bm{\beta}}_1, \tilde{\bm{\beta}}_0)$ & {\ssize  \neutranie} & {\ssize  \smiley \neutranie}  & {\ssize \neutranie} & {\ssize \neutranie}
			\\
			A $=$ D & Known & $(\tilde{\bm{\beta}}_1, \tilde{\bm{\beta}}_0)$ & $(\tilde{\bm{\beta}}_1, \tilde{\bm{\beta}}_0)$ & {\ssize \neutranie} & {\ssize  \neutranie} & {\ssize \neutranie} & {\ssize \neutranie}
			\\
			General & Unknown & $\bm{\rcoef} \approx \tilde{\bm{\rcoef}}_{\res}$ & $(\tilde{\bm{\beta}}_1, \tilde{\bm{\beta}}_0)$ &  {\ssize \smiley \neutranie \frownie} & {\ssize \smiley \neutranie}  & {\ssize \smiley \neutranie}  & {\ssize \neutranie}\\
			\bottomrule
	\end{tabular}}%
\end{table}

\subsection{ReM, Lin's estimator, and the Huber--White variance estimator}

Based on the summary in Section \ref{sec:uni}, we recommend using ReM in the design and  $\hat{\tau}(\tilde{\bm{\beta}}_1, \tilde{\bm{\beta}}_0)$ in the analysis, which has better estimated precision and coverage property. 
However, some practical issues remain.

First, we need to estimate the population OLS coefficients $\tilde{\bm{\beta}}_1$ and $\tilde{\bm{\beta}}_0$. Under both the CRE and ReM, we can use their sample analogues as consistent estimators, with $\hat{\bm{\beta}}_1$ and $ \hat{\bm{\beta}}_0$ being the coefficients of $\bm{w}$ in the OLS fit of $Y$ on $\bm{w}$ under the treatment and control, respectively. The corresponding adjusted estimator $\hat{\tau}(\hat{\bm{\beta}}_1, \hat{\bm{\beta}}_0)$ numerically equals the coefficient of $Z$ in the OLS fit of $Y$ on $Z$, $\bm{w}$ and $Z\times \bm{w}$, i.e., it is \citet{lin2013}'s estimator. Replacing $\tilde{\bm{\beta}}_1$ and $\tilde{\bm{\beta}}_0$ by their sample analogues does not change the asymptotic distribution. Informally speaking, $\hat{\tau}(\tilde{\bm{\beta}}_1, \tilde{\bm{\beta}}_0)$ and $\hat{\tau}(\hat{\bm{\beta}}_1, \hat{\bm{\beta}}_0)$ have the same asymptotic behavior and optimality. The following corollary is a formal statement.

\begin{proposition}
\label{prop::equivalent}
Under ReM and Condition \ref{con:fp}, $\hat{\tau}(\tilde{\bm{\beta}}_1, \tilde{\bm{\beta}}_0)$ and $\hat{\tau}(\hat{\bm{\beta}}_1, \hat{\bm{\beta}}_0)$ have the same asymptotic distributions and the same probability limits of the estimated distributions. Thus, among \eqref{eq:reg},
$\hat{\tau}(\hat{\bm{\beta}}_1, \hat{\bm{\beta}}_0)$ is $\mathcal{S}$-optimal under Condition \ref{con:ana_more} or \ref{con:des_more}, 
and is
always optimal in terms of the estimated precision based on \eqref{eq::samplingdistribution} and \eqref{eq::incompleteinformation}. 
\end{proposition}

Second, the estimated distributions \eqref{eq::samplingdistribution} and \eqref{eq::incompleteinformation} are identical with or without Condition \ref{con:ana_more_bal_criterion} because $\hat{R}^2_{\tau, \bm{x}}(\hat{\bm{\beta}}_{1}, \hat{\bm{\beta}}_0)$ based on \eqref{eq:R2_beta_hat} equals zero under Condition \ref{con:ana_more_bal_criterion}. 
Moreover,
the variance estimator $\hat{V}_{\tau \tau}(  \hat{\bm{\beta}}_1, \hat{\bm{\beta}}_0 )$ based on \eqref{eq:V_beta_hat} is asymptotically equivalent to the Huber--White variance estimator $\hat{V}_{\textup{HW}}$ of the coefficient of $Z$ from the OLS fit of $Y$ on $Z$, $\bm{w}$ and $Z\times \bm{w}$.

\begin{theorem}
\label{thm::HWequivalent}
Under ReM and Condition \ref{con:fp}, $\hat{V}_{\tau \tau}(  \hat{\bm{\beta}}_1, \hat{\bm{\beta}}_0 ) - \hat{V}_{\textup{HW}} \rightarrow 0$ in probability. 
\end{theorem}

Theorem \ref{thm::HWequivalent} extends \citet{lin2013}'s result for the CRE to ReM. It requires only Condition \ref{con:fp}, but \citet{lin2013} requires higher order moment conditions. 
We finally construct the Wald-type confidence intervals based on a Gaussian approximation. 
The statistical inference based on $\hat{\tau}(\hat{\bm{\beta}}_1, \hat{\bm{\beta}}_0)$, including variance estimation and confidence interval construction, is always the same no matter whether the design is a CRE or ReM and no matter whether the analyzer knows all the information of the design or not. 

From the above, using ReM and $\hat{\tau}(\hat{\bm{\beta}}_1, \hat{\bm{\beta}}_0)$ enjoys the  
optimal estimated precision
and improves the coverage property, 
and the associated statistical inference can be conveniently implemented through the OLS fit and Huber--White variance estimate.

\section{Illustration}\label{sec::illustrations}

\subsection{A simulation study}\label{sec::simulation}
 
We conduct a simulation study to investigate the performance of the asymptotic approximation and the coverage properties of the confidence intervals in finite samples.  We generate the data in the same way as in Example \ref{eg:sampling} with $\rho=0$ and vary the sample size $n$ from 100 to 1000. 
For each simulated dataset, we generate ReM based on covariate $x$, where the the threshold $a$ is the $0.001$th quantile of the $\chi^2_1$ random variable. 
Figure \ref{fig:simu_vary_n}(a)--(b) show the histograms of $\hat{\tau}$ and $\hat{\tau}(\hat{\bm{\beta}}_1, \hat{\bm{\beta}}_0)$ based on covariate $w$.
From Figure \ref{fig:simu_vary_n}(a)--(b), the asymptotic approximation works fairly well. 
We then construct $95\%$ confidence intervals for the average treatment effect, using the estimated distribution \eqref{eq::incompleteinformation} with 
either $\hat{\tau}$ or $\hat{\tau}(\hat{\bm{\beta}}_1, \hat{\bm{\beta}}_0)$. 
From Figure \ref{fig:simu_vary_n}(c)--(d), 
the confidence intervals based on estimator adjusted for $w$ are shorter than that based on $\hat{\tau}$, 
and 
both confidence intervals are conservative with coverage probabilities larger than the nominal level, due to the analyzer's incomplete information of the design. 
We further consider $\hat{\tau}(\hat{\bm{\beta}}_1, \hat{\bm{\beta}}_0)$ based on $(x,w)$, assuming that the analyzer has access to the covariate $x$ in the design. 
The corresponding confidence interval is even shorter and  becomes asymptotically exact, due to the additive treatment effects in the data generating process. 
From Figure \ref{fig:simu_vary_n}(d),  the coverage probabilities are close to the nominal level as the sample size increases. 
Even when the sample size is small, the confidence interval works fairly well with coverage probability at least $94\%$.

\begin{figure}[t]
	\centering
	\begin{subfigure}{.5\textwidth}
		\centering
		\includegraphics[width=0.7\linewidth]{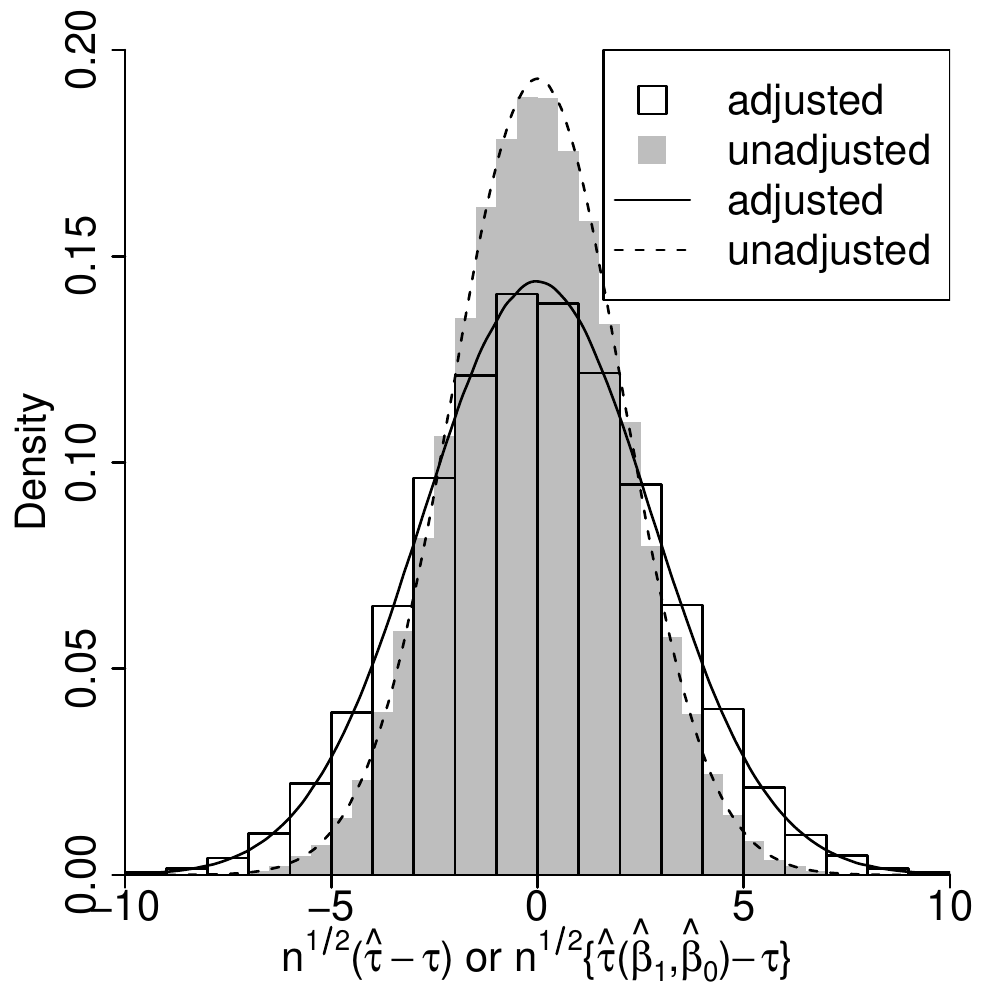}
		\caption{\centering $n = 100$}
	\end{subfigure}%
	\begin{subfigure}{.5\textwidth}
		\centering
				\includegraphics[width=0.7\linewidth]{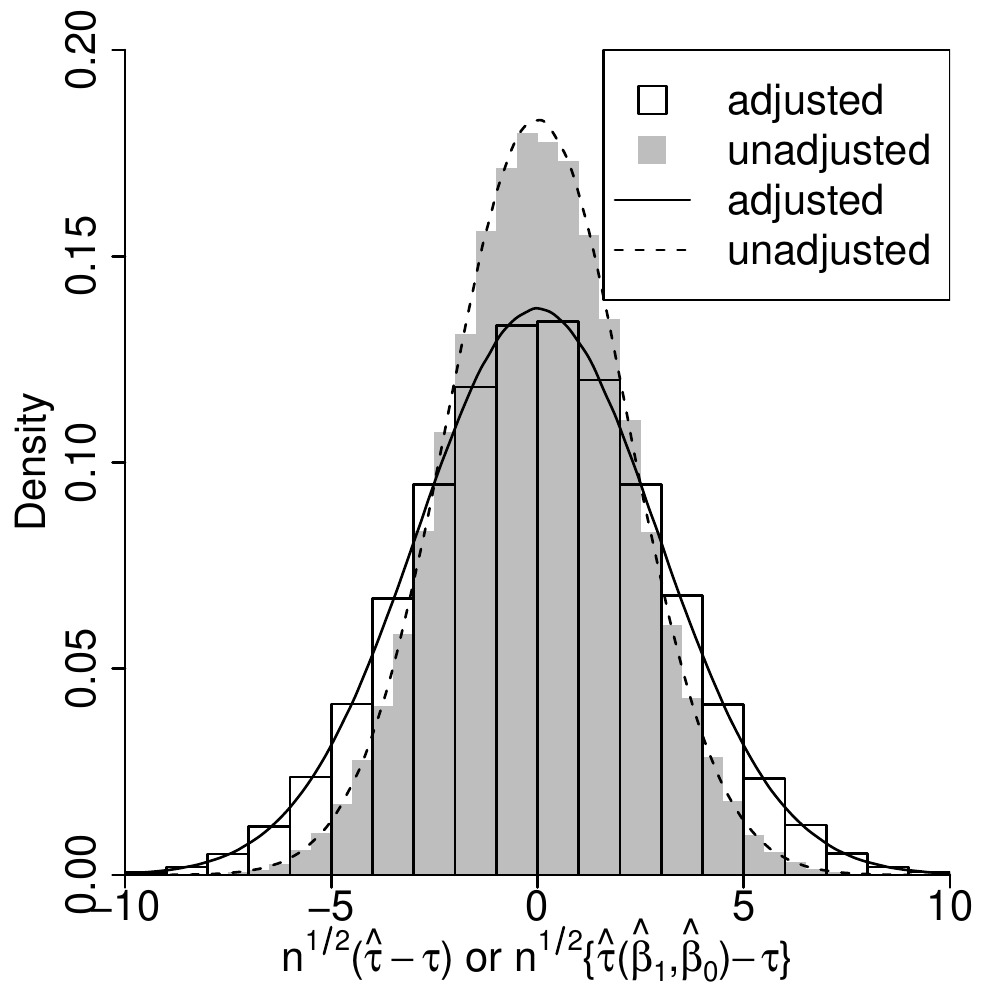}
		\caption{$n = 300$} 
	\end{subfigure}
	\begin{subfigure}{.5\textwidth}
		\centering
		\includegraphics[width=0.7\linewidth]{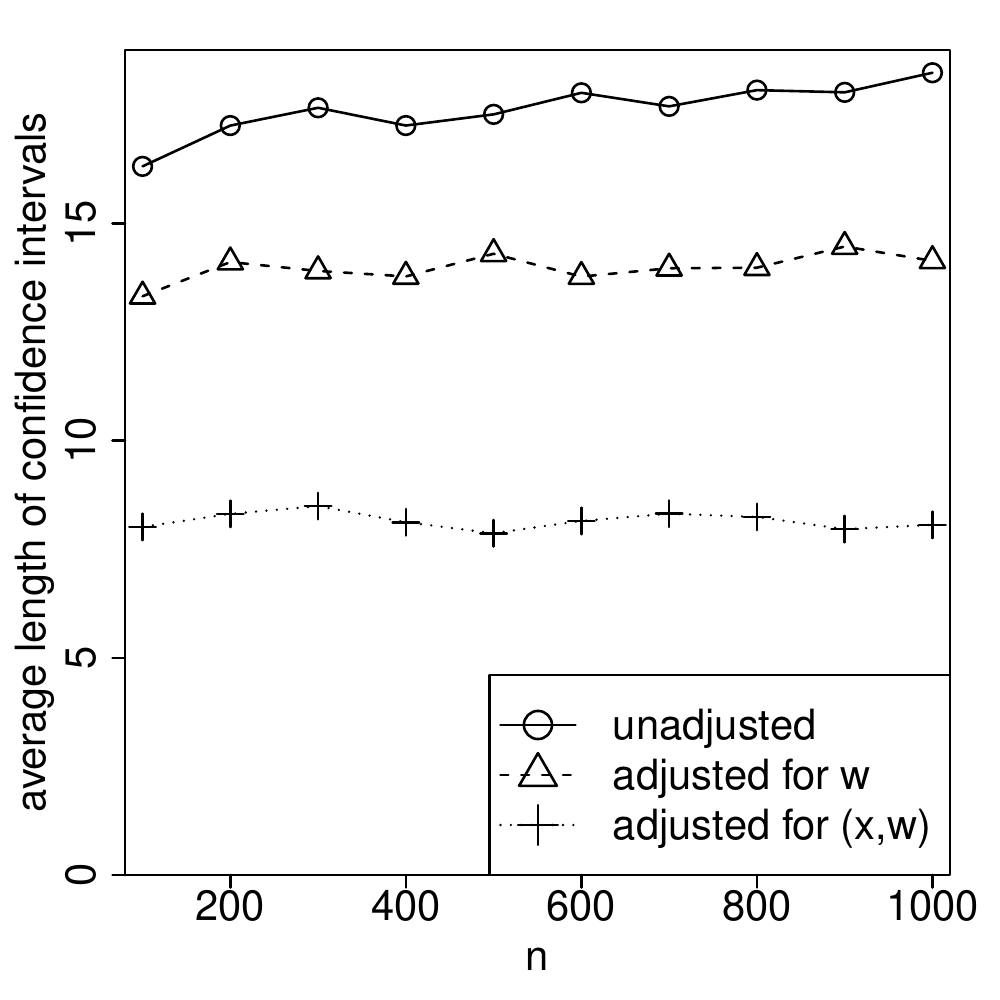}
		\caption{\centering lengths}
	\end{subfigure}%
	\begin{subfigure}{.5\textwidth}
		\centering
				\includegraphics[width=0.7\linewidth]{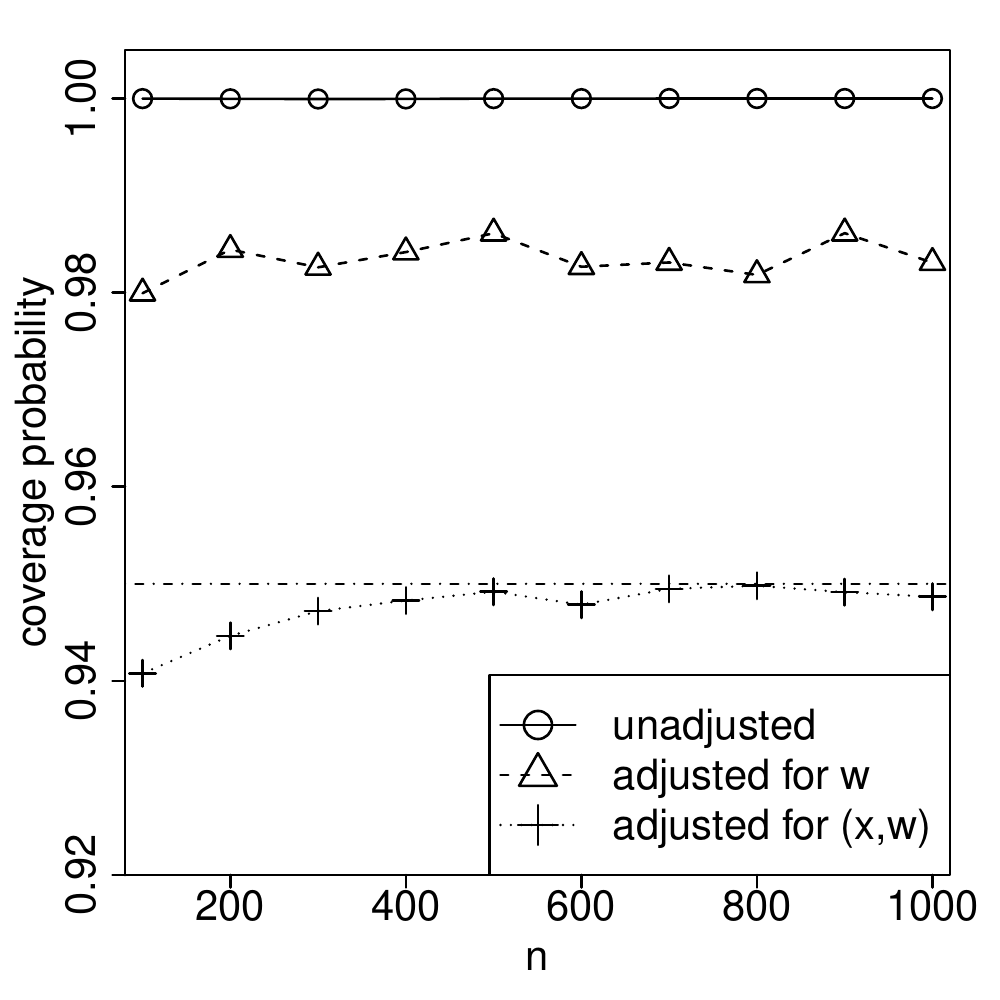}
		\caption{coverage probabilities} 
	\end{subfigure}
	\caption{
	(a) and (b)	show the 
	histograms of  
		$n^{1/2}\{\hat{\tau}(\hat{\bm{\beta}}_1, \hat{\bm{\beta}}_0)-\tau\}$ and 
	 $n^{1/2}(\hat{\tau}-\tau)$ under ReM based on $10^5$ simulated treatment assignments.
	 (c) and (d) show the average lengths and coverage probabilities of $95\%$ confidence intervals constructed from \eqref{eq::incompleteinformation} for three estimators: 
	 $\hat{\tau}$, adjusted estimator based on $w$, and adjusted estimator based on $x$ and $w$.  }\label{fig:simu_vary_n}
\end{figure}

\subsection{The ``Opportunity Knocks'' experiment}
The ``Opportunity Knocks'' experiment \citep{angrist2014opportunity} aims at evaluating the impact of a financial incentive demonstration program on college students' academic performance. The experiment includes first- and second-year students who apply for the financial aid at a large Canadian commuter university. These students were randomly assigned to treatment and control groups. Students in the treated group have peer advisors and receive cash reward for attaining certain grades.

We use this dataset to illustrate rerandomization and regression adjustment. We consider the second-year students, 
and choose the outcome to be the average grade for the semester right after the experiment. 
We exclude students with missing outcomes or covariates, resulting in a treatment group of size 199 and a control group of size 369. 
We evaluate the repeated sampling properties of the adjusted estimators under rerandomization,  which depend on all the potential outcomes. 
However, half of the potential outcomes are missing from the observed data.  
To make the simulation more realistic, we impute all the missing potential outcomes based on a simple model fitting. 
We first fit a linear model of the observed outcome on the treatment indicator and covariates within each stratum classified by sex and high school GPA. We then impute the missing potential outcomes using the fitted linear model. 

We conduct ReM with two covariates  
and choose threshold $a$ to be the $0.005$th quantile of $\chi^2_2$.
For the covariates in the design and the analysis, 
we consider the following two cases: 
\begin{enumerate}
	[(i)]
	\item
	the covariates in the design are sex and high school grade, and
	the  covariates in the analysis are 
	whether mother/father is a college graduate, whether correctly answer the first/second question in a survey, whether mother tongue is English, and  GPA in the previous year;
	\item the covariates in the design and analysis are the same as in case (i), except that we switch the high school grade to the analysis stage and switch the GPA in the previous year to the design stage.
\end{enumerate}

We first consider the sampling precision. 
Table \ref{tab:OK_samp_est} shows the coefficients of $\varepsilon$ in the asymptotic distributions. We omit the coefficients of $L_{K,a}$ because the $\varepsilon$ components are the dominating terms in the asymptotic distributions. 
Figure \ref{fig:OKsampling} shows the histograms of $\hat{\tau}(\hat{\bm{\beta}}_1, \hat{\bm{\beta}}_0)$ and $\hat{\tau}$ under ReM.
In Table \ref{tab:OK_samp_est},  compared to the second column, the reduction in coefficients in the first column shows the gain from the designer alone, and the reduction in the last column shows the gain from the analyzer alone. The magnitude of the reduction suggests the relative amount of the covariate information of the designer and analyzer.
In case (i), 
the first row of Table \ref{tab:OK_samp_est} shows that $\hat{\tau}(\hat{\bm{\beta}}_1, \hat{\bm{\beta}}_0)$ under the CRE is more precise than $\hat{\tau}$ under ReM. This holds because
$R^2_{\tau,\bm{w}} = 0.65 \geq R^2_{\tau,\bm{x}} = 0.31$. 
Figure \ref{fig:OKsampling}(a) shows that 
$\hat{\tau}(\hat{\bm{\beta}}_1, \hat{\bm{\beta}}_0)$ outperforms $\hat{\tau}$, coherent with Theorem \ref{thm:gen_ana_enough}. 
In case (ii), 
the second row of Table \ref{tab:OK_samp_est} shows that $\hat{\tau}(\hat{\bm{\beta}}_1, \hat{\bm{\beta}}_0)$ under the CRE is less precise than $\hat{\tau}$ under ReM. This is because 
$R^2_{\tau,\bm{w}} = 0.32 < R^2_{\tau,\bm{x}} = 0.65$. 
Figure \ref{fig:OKsampling}(b) shows that 
$\hat{\tau}$ outperforms $\hat{\tau}(\hat{\bm{\beta}}_1, \hat{\bm{\beta}}_0)$ under ReM.

We then consider the estimated precision. Because the estimated distributions in \eqref{eq::incompleteinformation} are Gaussian, it suffices to compare the estimated standard errors. Table \ref{tab:OK_samp_est} shows the average estimated standard errors and the coverage probabilities of the $95\%$ confidence intervals. 
In both cases, $\hat{\tau}(\hat{\bm{\beta}}_1, \hat{\bm{\beta}}_0)$ has almost the same estimated precision under ReM and the CRE. So does $\hat{\tau}$. However, the $95\%$ confidence intervals under ReM have higher coverage probability, which is coherent with Corollary \ref{cor:gen_ci_larger_coverage_rem}.
Moreover, $\hat{\tau}(\hat{\bm{\beta}}_1, \hat{\bm{\beta}}_0)$ always has higher estimated precision than $\hat{\tau}$, which is coherent with the 
optimality results in Corollary \ref{cor:opt_conf_general}.

\begin{table}[th]
	\centering
	\caption{
		Sampling precision, estimated precision and coverage probabilities for $\hat{\tau}$ and $\hat{\tau}(\hat{\bm{\beta}}_1, \hat{\bm{\beta}}_0)$ based on $10^5$ simulated treatment assignments. 
		The first two rows show 
		the coefficients of $\varepsilon$ in the asymptotic distributions of  $\hat{\tau}$ and $\hat{\tau}(\hat{\bm{\beta}}_1, \hat{\bm{\beta}}_0)$ in \eqref{eq:adj_rem}. The last two rows show the average estimated standard errors multiplied by $n^{1/2}$ with the coverage probabilities of the $95\%$ confidence intervals  in the parentheses. 
	}\label{tab:OK_samp_est}
	\begin{tabular}{ccccccccc}
		\toprule
		& Estimator & \multicolumn{3}{c}{ $\hat{\tau}$}  & & \multicolumn{3}{c}{ $\hat{\tau}(\hat{\bm{\beta}}_1, \hat{\bm{\beta}}_0)$}\\
		\hline
		& Design & \multicolumn{1}{c}{ReM}  & & \multicolumn{1}{c}{CRE} & & \multicolumn{1}{c}{ReM}  & & \multicolumn{1}{c}{CRE}
		\\
		\midrule 
		Sampling & case (i) & $13.88$ &  &  $16.73$ & & $9.83$ & & $9.86$\\
		& case (ii) & $9.90$  & & $16.73$ & & $11.62$ & & $13.80$\\
		\midrule
		Estimated & case (i) & $18.56$ ($98.5\%$) &  &  $18.56$ ($96.0\%$) & & $12.75$ ($97.1\%$) & & $12.75$ ($97.0\%$)\\
		& case (ii) & $18.57$ ($99.9\%$) & & $18.56$ ($96.0\%$) & & $15.95$ ($98.3\%$) & & $15.94$ ($96.1\%$)\\
		\bottomrule
	\end{tabular}
\end{table}


\begin{figure}[th]
	\centering
	\begin{subfigure}{.5\textwidth}
		\centering
		\includegraphics[width=0.7\linewidth]{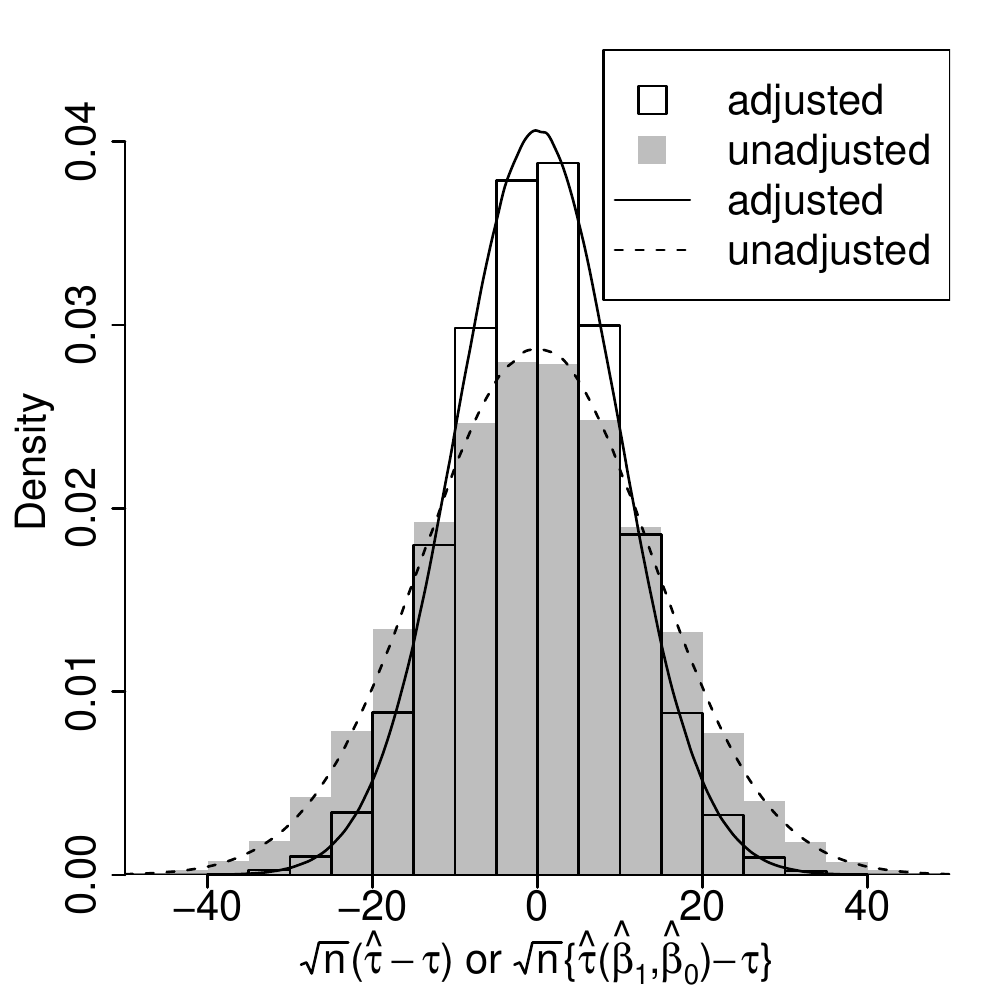}
		\caption{\centering case (i)}
	\end{subfigure}%
	\begin{subfigure}{.5\textwidth}
		\centering
		\includegraphics[width=0.7\linewidth]{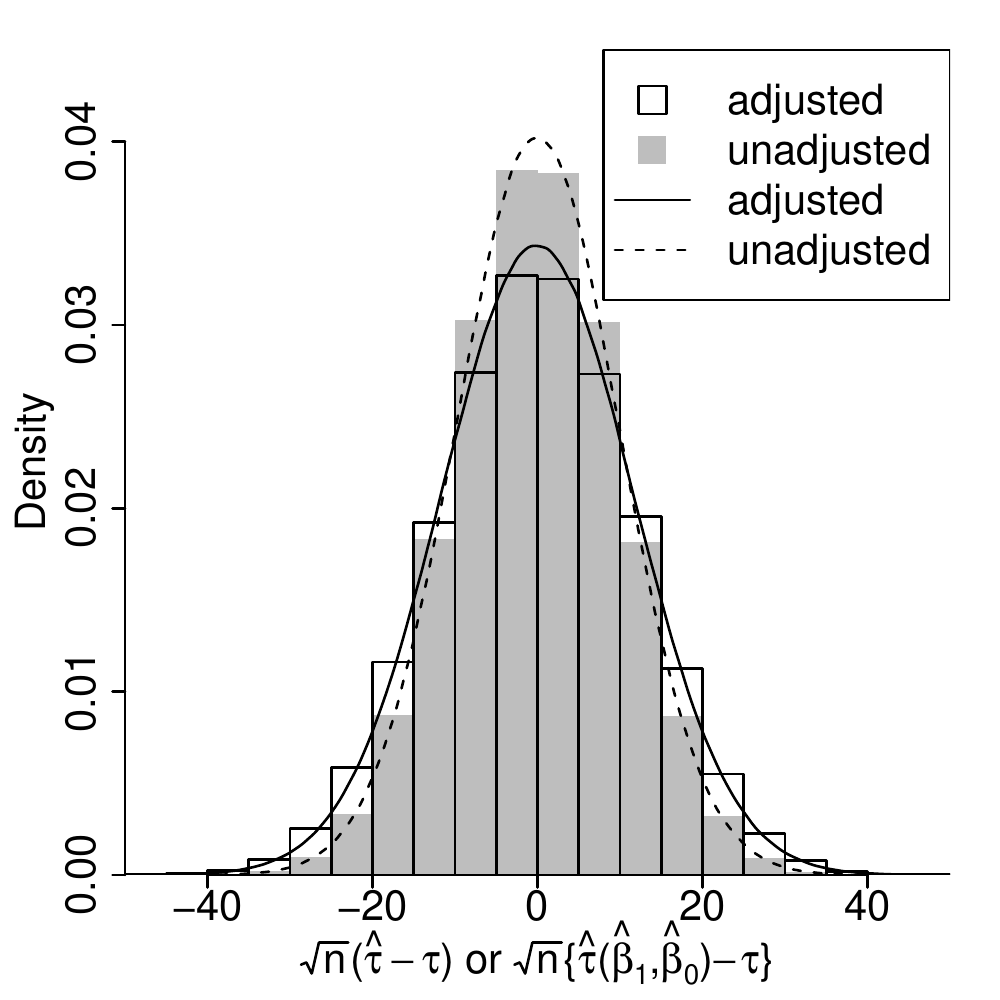}
		\caption{\centering case (ii)} 
	\end{subfigure}
	\caption{Histograms of $\hat{\tau}$ and $\hat{\tau}(\hat{\bm{\beta}}_1, \hat{\bm{\beta}}_0)$
		under ReM based on $10^5$ simulated treatment assignments.
	}
	\label{fig:OKsampling}
\end{figure}

\section{Discussion}\label{sec::discussion}

In sum, regression adjustment can improve the estimated precision but may hurt the sampling precision, and ReM can improve the sampling precision and  never hurts the estimated precision. 
The resulting adjusted estimator is optimal 
in terms of the estimated precision among all linearly adjusted estimators in \eqref{eq:reg}, 
has lower sampling variability under ReM than it would have had under the CRE, 
and the corresponding confidence intervals have higher coverage probabilities than that under the CRE.
Therefore, in practice, we recommend using ReM in the design and using \citet{lin2013}'s estimator in the analysis followed by the Huber--White robust standard error. Importantly, the analyzer should communicate with the designer, asking for detailed covariate information and assignment mechanism in the design stage.

For the analysis, we focused on inferring the average treatment effect using regression adjustment. 
It is interesting to extend the discussion to covariate adjustment in more complicated settings, such as high dimensional covariates \citep{bloniarz2015lasso, wager2016high,lei2018regression}, 
logistic regression for binary outcomes \citep{zhang2008improving, freedman2008randomization, moore2009covariate, moore2011robust}, 
and adjustment using machine learning methods \citep{bloniarz2015lasso, wager2016high, wu2018loop}. 
It is also important to consider covariate adjustment for general nonlinear estimands \citep{zhang2008improving, jiang2019robust, tian2019moving} and general designs \citep{middleton2018unified}, 
such as blocking \citep{miratrix2013adjusting,bugni2018inference}, matched pairs \citep{fogarty2018regression}, and factorial designs \citep{lu2016covariate}.

For the design, we focused on rerandomization using the Mahalanobis distance. It is conceptually straightforward to extend the results to rerandomization with tiers of covariates \citep{morgan2015rerandomization, asymrerand2106}.  Recently, \citet{zhou2018sequential} discussed sequential rerandomization, and \citet{li2018rerandomization} discussed rerandomization in $2^K$ factorial experiments with tiers of both covariates and factorial effects.  It is important to discuss regression adjustment after these rerandomizations.

The relationship between blocking and post-stratification for discrete covariates is analogous to the relationship between rerandomization and regression adjustment for general covariates. 
When the number of blocks is small compared to the sample size, our results suggest conducting post-stratification, or equivalently an OLS fit of the outcome on treatment, block indicator and their interactions, followed by the Huber--White robust standard error. 
	When the number of blocks is large, 
\citet{miratrix2013adjusting} showed that post-stratification can be worse than blocking, 
which sheds light on the possible advantage of rerandomization over regression adjustment with a large number of covariates. In this case, although deriving the asymptotic properties of rerandomization is challenging, it is still straightforward to conduct Fisher randomization tests.

\section*{Acknowledgments}
We thank the Associate Editor and two reviewers for constructive comments. Peng Ding gratefully acknowledges financial support from the National Science Foundation (DMS grant \# 1713152).

\bibliographystyle{plainnat}
\bibliography{causal}

\newpage
\setcounter{equation}{0}
\setcounter{section}{0}
\setcounter{figure}{0}
\setcounter{example}{0}
\setcounter{proposition}{0}
\setcounter{corollary}{0}
\setcounter{theorem}{0}
\setcounter{table}{0}
\setcounter{condition}{0}

\renewcommand {\theproposition} {A\arabic{proposition}}
\renewcommand {\theexample} {A\arabic{example}}
\renewcommand {\thefigure} {A\arabic{figure}}
\renewcommand {\thetable} {A\arabic{table}}
\renewcommand {\theequation} {A\arabic{equation}}
\renewcommand {\thelemma} {A\arabic{lemma}}
\renewcommand {\thesection} {A\arabic{section}}
\renewcommand {\thetheorem} {A\arabic{theorem}}
\renewcommand {\thecorollary} {A\arabic{corollary}}
\renewcommand {\thecondition} {A\arabic{condition}}

\setcounter{page}{1}

\begin{center}
\bf \Large 
Supplementary Material
\end{center}


\bigskip

Appendix \ref{appendix::sampling} proves   the results related to the sampling distributions.
 
Appendix \ref{appendix::soptimal} proves   the results related to $\mathcal{S}$-optimality.

Appendix \ref{appendix::estimationsampling} proves   the results related to the confidence intervals and the optimal adjusted estimators in terms of the estimated precision.

Appendix \ref{sec:gains_proof} proves   the results related to the gains from the analyzer and the designer.

Appendix \ref{sec::hwappendix} proves the asymptotic equivalence of $\hat{\tau}(\tilde{\bm{\beta}}_1, \tilde{\bm{\beta}}_0)$ and $\hat{\tau}(\hat{\bm{\beta}}_1, \hat{\bm{\beta}}_0)$ as well as the asymptotic equivalence of $\hat{V}_{\tau\tau}(\hat{\bm{\beta}}_1, \hat{\bm{\beta}}_0)$ and the Huber--White variance estimator.



\section{Sampling distributions of adjusted estimators}
\label{appendix::sampling}

\begin{proof}[{\bf Proof of Proposition \ref{prop:three_beta_tilde}}]
By definition,  
\begin{align*}
r_0\tilde{\bm{\beta}}_1+r_1 \tilde{\bm{\beta}}_0 & = 
r_0 \left( \bm{S}_{\bm{w}}^2 \right)^{-1}
\bm{S}_{\bm{w}, Y(1)} + r_1 \left(\bm{S}_{\bm{w}}^2 \right)^{-1}
\bm{S}_{\bm{w}, Y(0)}
\\
& = \left\{ \left(r_1r_0\right)^{-1} \bm{S}^2_{\bm{w}} \right\}^{-1}
\left\{  
r_1^{-1}\bm{S}_{\bm{w}, Y(1)} + r_0^{-1}\bm{S}_{\bm{w}, Y(0)}
\right\} 
 = \bm{V}_{\bm{w} \bm{w}}^{-1}\bm{V}_{\bm{w}\tau} = \tilde{\bm{\rcoef}}.
\end{align*}
\end{proof}

\begin{proof}[{\bf Proof of Theorem \ref{thm:adj_rem}}]
The regression adjustment coefficients $\bm{\beta}_1$ and $\bm{\beta}_0$ can depend on sample size $n$ implicitly and have finite limits as $n\rightarrow \infty$. 
Recall that $Y_i(z;\bm{\beta}_z) = Y_i(z) - \bm{\beta}_z' \bm{w}_i$ is the ``adjusted'' potential outcome  under treatment $z$, and  $\tau_i(\bm{\beta}_1, \bm{\beta}_0) = \tau_i - (\bm{\beta}_1- \bm{\beta}_0)'\bm{w}_i$ is the ``adjusted'' individual treatment effect. 
Under Condition \ref{con:fp}(ii),  
the finite population variances and covariances 
$
S^2_{Y(z;\bm{\beta}_z)} = S^2_{Y(z)} + \bm{\beta}_z' \bm{S}^2_{\bm{w}} \bm{\beta}_z - 2\bm{\beta}_z'  \bm{S}_{\bm{w}, Y(z)}, 
$
$S^2_{\tau(\bm{\beta}_1, \bm{\beta}_0)} = S^2_{\tau} + (\bm{\beta}_1- \bm{\beta}_0)' \bm{S}^2_{\bm{w}} (\bm{\beta}_1- \bm{\beta}_0) - 2 (\bm{\beta}_1- \bm{\beta}_0)' \bm{S}_{\bm{w},\tau},$
$\bm{S}_{Y(z;\bm{\beta}_z), \bm{x}} = \bm{S}_{Y(z),\bm{x}} - \bm{\beta}_z'\bm{S}_{\bm{w}, \bm{x}},$
and 
$
\bm{S}^2_{\bm{x}}
$
have finite limiting values. Under Condition  \ref{con:fp}(iii), the maximum squared distances satisfy that as $n\rightarrow \infty$, 
$\max_{1\leq i\leq n}\|\bm{x}_i\|_2^2/n \rightarrow 0$, and 
\begin{align*}
n^{-1}\max_{1\leq i\leq n} \left|Y_i(z;\bm{\beta}_z) - \bar{Y}(z;\bm{\beta}_z)\right|^2 
& = n^{-1}\max_{1\leq i\leq n} \left|Y_i(z) - \bar{Y}(z) - \bm{\beta}_z'\bm{w}_i\right|^2 \\
& \leq 
n^{-1}   (  1+\bm{\beta}_z'\bm{\beta}_z ) \max_{1\leq i\leq n} \left( \left|Y_i(z) - \bar{Y}(z)\right|^2 + \left\|\bm{w}_i\right\|_2^2 \right) \rightarrow 0, 
\end{align*}
where the inequality follows from the Cauchy--Schwarz inequality. 
Using \citet[][Theorem 1]{asymrerand2106}, we can show that, under ReM, $n^{1/2}\left\{\hat{\tau}(\bm{\beta}_1, \bm{\beta}_0)-\tau \right\}$ has the asymptotic distribution \eqref{eq:adj_rem}. 
\end{proof}

\begin{proof}[{\bf Proof of Corollary \ref{corr:reg_cre}}]
Corollary \ref{corr:reg_cre} follows from Theorem \ref{thm:adj_rem} with $a=\infty$.
\end{proof}

\begin{proof}[{\bf Proof of Corollary \ref{corr:diff_rem}}]
Corollary \ref{corr:diff_rem} follows from Theorem \ref{thm:adj_rem} with $( \bm{\beta}_1, \bm{\beta}_0 ) = (\bm{0},\bm{0})$.
\end{proof}

\begin{proof}[\bf Proof of Proposition \ref{lemma:relation_three}]
By definition, Cov$(\hat{\bm{\tau}}_{\bm{w}}, \hat{\tau}  - \tilde{\bm{\rcoef}}'\hat{\bm{\tau}}_{\bm{w}} ) =0$ under the CRE. We have 
\begin{align}\label{eq:beta_proj_beta_res}
\hat{\tau}  - \tilde{\bm{\rcoef}}'\hat{\bm{\tau}}_{\bm{w}} & = 
\proj(\hat{\tau} \mid \hat{\bm{\tau}}_{\bm{x}}) + 
\res(\hat{\tau} \mid \hat{\bm{\tau}}_{\bm{x}}) - 
\tilde{\bm{\rcoef}}'
\left\{
\proj( \hat{\bm{\tau}}_{\bm{w}} \mid \hat{\bm{\tau}}_{\bm{x}}) + 
\res( \hat{\bm{\tau}}_{\bm{w}} \mid \hat{\bm{\tau}}_{\bm{x}})
\right\} \nonumber \\
& =  
\left\{\proj(\hat{\tau} \mid \hat{\bm{\tau}}_{\bm{x}}) - \tilde{\bm{\rcoef}}_{\proj}'\proj( \hat{\bm{\tau}}_{\bm{w}} \mid \hat{\bm{\tau}}_{\bm{x}})\right\} + 
\left\{
\res(\hat{\tau} \mid \hat{\bm{\tau}}_{\bm{x}})- \tilde{\bm{\rcoef}}_{\res}'\res( \hat{\bm{\tau}}_{\bm{w}} \mid \hat{\bm{\tau}}_{\bm{x}})
\right\} \nonumber\\
& \quad \   
- (\tilde{\bm{\rcoef}} - \tilde{\bm{\rcoef}}_{\proj} )'\proj( \hat{\bm{\tau}}_{\bm{w}} \mid \hat{\bm{\tau}}_{\bm{x}}) - 
(\tilde{\bm{\rcoef}} - \tilde{\bm{\rcoef}}_{\res})'\res( \hat{\bm{\tau}}_{\bm{w}} \mid \hat{\bm{\tau}}_{\bm{x}}). 
\end{align}
Thus, the covariances between $\hat{\bm{\tau}}_{\bm{w}}$ and the four terms in \eqref{eq:beta_proj_beta_res} sum to 0. Below we consider the four covariances separately. 

First, by definition, 
$\proj(\hat{\tau} \mid \hat{\bm{\tau}}_{\bm{x}}) - \tilde{\bm{\rcoef}}_{\proj}'\proj( \hat{\bm{\tau}}_{\bm{w}} \mid \hat{\bm{\tau}}_{\bm{x}})$ is uncorrelated with $\proj( \hat{\bm{\tau}}_{\bm{w}} \mid \hat{\bm{\tau}}_{\bm{x}})$. 
Moreover, 
because $\proj(\hat{\tau} \mid \hat{\bm{\tau}}_{\bm{x}}) - \tilde{\bm{\rcoef}}_{\proj}'\proj( \hat{\bm{\tau}}_{\bm{w}} \mid \hat{\bm{\tau}}_{\bm{x}})$ is a linear function of $\hat{\bm{\tau}}_{\bm{x}}$, it must also be uncorrelated with $\res( \hat{\bm{\tau}}_{\bm{w}} \mid \hat{\bm{\tau}}_{\bm{x}})$. 
Thus, $\proj(\hat{\tau} \mid \hat{\bm{\tau}}_{\bm{x}}) - \tilde{\bm{\rcoef}}_{\proj}'\proj( \hat{\bm{\tau}}_{\bm{w}} \mid \hat{\bm{\tau}}_{\bm{x}})$ is uncorrelated with 
$\proj( \hat{\bm{\tau}}_{\bm{w}} \mid \hat{\bm{\tau}}_{\bm{x}}) + \res( \hat{\bm{\tau}}_{\bm{w}} \mid \hat{\bm{\tau}}_{\bm{x}}) = \hat{\bm{\tau}}_{\bm{w}}$.

Second, 
by definition, 
$\res(\hat{\tau}\mid \hat{\bm{\tau}}_{\bm{x}})- \tilde{\bm{\rcoef}}_{\res}'\res( \hat{\bm{\tau}}_{\bm{w}} \mid \hat{\bm{\tau}}_{\bm{x}})$ is uncorrelated with $\res( \hat{\bm{\tau}}_{\bm{w}} \mid \hat{\bm{\tau}}_{\bm{x}})$. 
Moreover, because 
$\res(\hat{\tau}\mid \hat{\bm{\tau}}_{\bm{x}})- \tilde{\bm{\rcoef}}_{\res}'\res( \hat{\bm{\tau}}_{\bm{w}} \mid \hat{\bm{\tau}}_{\bm{x}})$ is uncorrelated with $\hat{\bm{\tau}}_{\bm{x}}$, 
it must also be uncorrelated with $\proj(\hat{\bm{\tau}}_{\bm{w}} \mid \hat{\bm{\tau}}_{\bm{x}})$. 
Thus, $\res(\hat{\tau}\mid \hat{\bm{\tau}}_{\bm{x}})- \tilde{\bm{\rcoef}}_{\res}'\res( \hat{\bm{\tau}}_{\bm{w}} \mid \hat{\bm{\tau}}_{\bm{x}})$ is uncorrelated with 
$
\res( \hat{\bm{\tau}}_{\bm{w}} \mid \hat{\bm{\tau}}_{\bm{x}}) + 
\proj( \hat{\bm{\tau}}_{\bm{w}} \mid \hat{\bm{\tau}}_{\bm{x}}) = \hat{\bm{\tau}}_{\bm{w}}. 
$

Third, 
because $\proj( \hat{\bm{\tau}}_{\bm{w}} \mid \hat{\bm{\tau}}_{\bm{x}})$ is uncorrelated with $\res( \hat{\bm{\tau}}_{\bm{w}} \mid \hat{\bm{\tau}}_{\bm{x}})$, 
we can simplify 
the covariance between  $(\tilde{\bm{\rcoef}} - \tilde{\bm{\rcoef}}_{\proj} )'\proj( \hat{\bm{\tau}}_{\bm{w}} \mid \hat{\bm{\tau}}_{\bm{x}})$ and $\hat{\bm{\tau}}_{\bm{w}} = \proj( \hat{\bm{\tau}}_{\bm{w}} \mid \hat{\bm{\tau}}_{\bm{x}}) + \res( \hat{\bm{\tau}}_{\bm{w}} \mid \hat{\bm{\tau}}_{\bm{x}})$ as 
$ (\tilde{\bm{\rcoef}} - \tilde{\bm{\rcoef}}_{\proj} )' \Cov\{
 \proj( \hat{\bm{\tau}}_{\bm{w}} \mid \hat{\bm{\tau}}_{\bm{x}}) 
\}$. 
The covariance of $\proj( \hat{\bm{\tau}}_{\bm{w}} \mid \hat{\bm{\tau}}_{\bm{x}})$ has the following equivalent forms: 
\begin{align}\label{eq:cov_W_proj_X}
\Cov\left\{
\proj
(
\hat{\bm{\tau}}_{\bm{w}}
\mid \hat{\bm{\tau}}_{\bm{x}}
)
\right\} & = n^{-1}\bm{V}_{\bm{wx}}\bm{V}_{\bm{xx}}^{-1}\bm{V}_{\bm{xw}} = (n r_1r_0)^{-1} \bm{S}^2_{\bm{w}\mid \bm{x}}. 
\end{align}

Fourth, because $\res( \hat{\bm{\tau}}_{\bm{w}} \mid \hat{\bm{\tau}}_{\bm{x}})$ is uncorrelated with $\proj( \hat{\bm{\tau}}_{\bm{w}} \mid \hat{\bm{\tau}}_{\bm{x}})$, we can simplify the covariance between $(\tilde{\bm{\rcoef}} - \tilde{\bm{\rcoef}}_{\res})'\res( \hat{\bm{\tau}}_{\bm{w}} \mid \hat{\bm{\tau}}_{\bm{x}})$ and $\hat{\bm{\tau}}_{\bm{w}} = \proj( \hat{\bm{\tau}}_{\bm{w}} \mid \hat{\bm{\tau}}_{\bm{x}}) + \res( \hat{\bm{\tau}}_{\bm{w}} \mid \hat{\bm{\tau}}_{\bm{x}})$ as 
$(\tilde{\bm{\rcoef}} - \tilde{\bm{\rcoef}}_{\res})' \Cov\{
\res( \hat{\bm{\tau}}_{\bm{w}} \mid \hat{\bm{\tau}}_{\bm{x}})
\}.$ 
The covariance of $\res( \hat{\bm{\tau}}_{\bm{w}} \mid \hat{\bm{\tau}}_{\bm{x}})$ has the following equivalent forms: 
\begin{align}\label{eq:cov_W_res_X}
\Cov\left\{
\res
(
\hat{\bm{\tau}}_{\bm{w}}
\mid \hat{\bm{\tau}}_{\bm{x}}
)
\right\} & = \Cov(\hat{\bm{\tau}}_{\bm{w}}) - \Cov\left\{
\proj
(
\hat{\bm{\tau}}_{\bm{w}}
\mid \hat{\bm{\tau}}_{\bm{x}}
)
\right\}  = (n r_1r_0)^{-1} \bm{S}^2_{\bm{w}\setminus \bm{x}}.
\end{align}

From the above, the zero covariance between $\hat{\bm{\tau}}_{\bm{w}}$ and \eqref{eq:beta_proj_beta_res} implies that 
\begin{align*}
\bm{0} 
& = -(nr_1r_0)^{-1} \left\{(\tilde{\bm{\rcoef}} - \tilde{\bm{\rcoef}}_{\proj} )' \bm{S}^2_{\bm{w}\mid \bm{x}} + 
(\tilde{\bm{\rcoef}} - \tilde{\bm{\rcoef}}_{\res} )' \bm{S}^2_{\bm{w}\setminus \bm{x}}\right\}. 
\end{align*}
Therefore, Proposition \ref{lemma:relation_three} holds. 
\end{proof}

\begin{proof}[{\bf Proof of Theorem \ref{thm:reg_rem_beta}}]
First, by the definitions of $V_{\tau\tau}(\bm{\beta}_1, \bm{\beta}_0)$  in \eqref{eq:V_tau_adj} and $R^2_{\tau,\bm{x}}(\bm{\beta}_{1}, \bm{\beta}_0)$ in \eqref{eq:R2_tau_x_beta}, the squared coefficients of $L_{K,a}$ and $\varepsilon$ in \eqref{eq:adj_rem} have the following equivalent forms:
\begin{align}\label{eq:var_proj_proof_rem}
V_{\tau\tau}(\bm{\beta}_1, \bm{\beta}_0) R^2_{\tau,\bm{x}}(\bm{\beta}_{1}, \bm{\beta}_0) & =
n\Var\left\{ \proj\left(\hat{\tau}(\bm{\beta}_1,\bm{\beta}_0)\mid \hat{\bm{\tau}}_{\bm{x}} \right) \right\}, \\
\label{eq:var_res_proof_rem}
V_{\tau\tau}(\bm{\beta}_1, \bm{\beta}_0)\left\{1- R^2_{\tau,\bm{x}}(\bm{\beta}_{1}, \bm{\beta}_0) \right\} & = n\Var\left\{
\res\left(\hat{\tau}(\bm{\beta}_1,\bm{\beta}_0)\mid \hat{\bm{\tau}}_{\bm{x}}\right) 
\right\}.
\end{align}

Second, 
because $\hat{\tau}(\bm{\beta}_1,\bm{\beta}_0) = \hat{\tau} - \bm{\rcoef}'\hat{\bm{\tau}}_{\bm{w}}$ by \eqref{eq:reg}, the linear projection of $\hat{\tau}(\bm{\beta}_1,\bm{\beta}_0)$ on $\hat{\bm{\tau}}_{\bm{x}}$ under the CRE and the corresponding residual have the following equivalent forms: 
\begin{align*}
\proj\left(\hat{\tau}(\bm{\beta}_1,\bm{\beta}_0)\mid \hat{\bm{\tau}}_{\bm{x}} \right)
& = 
\proj
\left(
\hat{\tau} - \bm{\rcoef}'\hat{\bm{\tau}}_{\bm{w}}
\mid \hat{\bm{\tau}}_{\bm{x}}
\right) = 
\proj\left(
\hat{\tau} \mid \hat{\bm{\tau}}_{\bm{x}}
\right)
- 
\bm{\rcoef}'
\proj
\left(
\hat{\bm{\tau}}_{\bm{w}}
\mid \hat{\bm{\tau}}_{\bm{x}}
\right),
\\
\res\left(\hat{\tau}(\bm{\beta}_1,\bm{\beta}_0)\mid \hat{\bm{\tau}}_{\bm{x}}\right)
& = \res
\left(
\hat{\tau} - \bm{\rcoef}'\hat{\bm{\tau}}_{\bm{w}}
\mid \hat{\bm{\tau}}_{\bm{x}}
\right)
= 
\res\left(
\hat{\tau} \mid \hat{\bm{\tau}}_{\bm{x}}
\right)
- 
\bm{\rcoef}'
\res
\left(
\hat{\bm{\tau}}_{\bm{w}}
\mid \hat{\bm{\tau}}_{\bm{x}}
\right). 
\end{align*}
Using the definitions of $\tilde{\bm{\rcoef}}_\proj$ in \eqref{eq:beta_proj} and $\tilde{\bm{\rcoef}}_\res$ in \eqref{eq:beta_res}, we can express the above quantities as 
\begin{align}
\proj\left(\hat{\tau}(\bm{\beta}_1,\bm{\beta}_0)\mid \hat{\bm{\tau}}_{\bm{x}} \right)
& = \proj\left(
\hat{\tau} \mid \hat{\bm{\tau}}_{\bm{x}}
\right)
- 
\tilde{\bm{\rcoef}}_\proj'
\proj
\left(
\hat{\bm{\tau}}_{\bm{w}}
\mid \hat{\bm{\tau}}_{\bm{x}}
\right) - 
\left(
\bm{\rcoef}
-\tilde{\bm{\rcoef}}_\proj
\right)' 
\proj
\left(
\hat{\bm{\tau}}_{\bm{w}}
\mid \hat{\bm{\tau}}_{\bm{x}}
\right)
\nonumber
\\
& =
\tau + 
\res\left\{
\proj(
\hat{\tau} \mid \hat{\bm{\tau}}_{\bm{x}}
) \mid 
\proj(
\hat{\bm{\tau}}_{\bm{w}}
\mid \hat{\bm{\tau}}_{\bm{x}}
)
\right\} - 
(
\bm{\rcoef}
-\tilde{\bm{\rcoef}}_\proj
)' 
\proj
(
\hat{\bm{\tau}}_{\bm{w}}
\mid \hat{\bm{\tau}}_{\bm{x}}
), 
\label{eq:tau_beta_proj}
\\
\res\left(\hat{\tau}(\bm{\beta}_1,\bm{\beta}_0)\mid \hat{\bm{\tau}}_{\bm{x}}\right)
& = 
\res\left(
\hat{\tau} \mid \hat{\bm{\tau}}_{\bm{x}}
\right)
- 
\tilde{\bm{\rcoef}}_\res'
\res
\left(
\hat{\bm{\tau}}_{\bm{w}}
\mid \hat{\bm{\tau}}_{\bm{x}}
\right)
- 
\left(\bm{\rcoef} - \tilde{\bm{\rcoef}}_\res\right)'
\res
\left(
\hat{\bm{\tau}}_{\bm{w}}
\mid \hat{\bm{\tau}}_{\bm{x}}
\right)
\nonumber
\\
& = 
\res
\left\{
\res\left(
\hat{\tau} \mid \hat{\bm{\tau}}_{\bm{x}}
\right) \mid 
\res
\left(
\hat{\bm{\tau}}_{\bm{w}}
\mid \hat{\bm{\tau}}_{\bm{x}}
\right)
\right\} - 
\left(\bm{\rcoef} - \tilde{\bm{\rcoef}}_\res\right)'
\res
\left(
\hat{\bm{\tau}}_{\bm{w}}
\mid \hat{\bm{\tau}}_{\bm{x}}
\right). 
\label{eq:tau_beta_res}
\end{align}

Third, 
because the two terms in \eqref{eq:tau_beta_proj} excluding the constant term $\tau$ are uncorrelated, the variance of 
the linear projection of $\hat{\tau}(\bm{\beta}_1,\bm{\beta}_0)$ on $\hat{\bm{\tau}}_{\bm{x}}$ is the summation of the variances of these two terms in \eqref{eq:tau_beta_proj}. 
Using \eqref{eq:cov_W_proj_X} and the definitions of $R^2_{\tau,\bm{x}}$ and $R^2_{\proj}$ in \eqref{eq:R2_tau_x} and \eqref{eq:beta_proj}, we have 
\begin{eqnarray}\label{eq:var_tau_beta_proj}
& & \Var\left\{
\proj\left(\hat{\tau}(\bm{\beta}_1,\bm{\beta}_0)\mid \hat{\bm{\tau}}_{\bm{x}} \right)
\right\}
\nonumber
 \\
& = & 
\Var\left\{
\proj\left(
\hat{\tau} \mid \hat{\bm{\tau}}_{\bm{x}}
\right)
\right\}\cdot 
\left(
1-R^2_{\proj}
\right) + 
\left(
\bm{\rcoef}
-\tilde{\bm{\rcoef}}_\proj
\right)' 
\Cov\left\{
\proj
\left(
\hat{\bm{\tau}}_{\bm{w}}
\mid \hat{\bm{\tau}}_{\bm{x}}
\right)\right\}
\left(
\bm{\rcoef}
-\tilde{\bm{\rcoef}}_\proj
\right)
\nonumber
\\
& = & 
n^{-1}V_{\tau\tau} \cdot R^2_{\tau,\bm{x}}\left(
1-R^2_{\proj}
\right) + 
(n r_1r_0)^{-1} 
\left(
\bm{\rcoef}
-\tilde{\bm{\rcoef}}_\proj
\right)' 
\bm{S}^2_{\bm{w}\mid \bm{x}}
\left(
\bm{\rcoef}
-\tilde{\bm{\rcoef}}_\proj
\right).
\end{eqnarray}
Similarly, because the two terms in \eqref{eq:tau_beta_res} are uncorrelated, the variance of the residual of 
the linear projection of $\hat{\tau}(\bm{\beta}_1,\bm{\beta}_0)$ on $\hat{\bm{\tau}}_{\bm{x}}$ is the summation of the variances of these two terms in \eqref{eq:tau_beta_res}. 
Using \eqref{eq:cov_W_res_X} and the definitions of $R^2_{\tau,\bm{w}}$ in \eqref{eq:R2_mid_w} and $R^2_{\res}$ in \eqref{eq:beta_res}, we have 
\begin{eqnarray}\label{eq:var_tau_beta_res}
& & \Var\left\{
\res\left(\hat{\tau}(\bm{\beta}_1,\bm{\beta}_0)\mid \hat{\bm{\tau}}_{\bm{x}}\right)
\right\} 
\nonumber
\\
& = & 
\Var\left\{
\res\left(
\hat{\tau} \mid \hat{\bm{\tau}}_{\bm{x}}
\right) 
\right\}\cdot
\left(
1-R^2_{\res}
\right) + 
\left(\bm{\rcoef} - \tilde{\bm{\rcoef}}_\res\right)'
\Cov\left\{\res
\left(
\hat{\bm{\tau}}_{\bm{w}}
\mid \hat{\bm{\tau}}_{\bm{x}}
\right)\right\}
\left(\bm{\rcoef} - \tilde{\bm{\rcoef}}_\res\right)
\nonumber
\\
& = &  
n^{-1}V_{\tau\tau} \cdot \left(
1 - R^2_{\tau,\bm{x}}
\right)\left(
1-R^2_{\res}
\right) + 
(n r_1r_0)^{-1} 
\left(\bm{\rcoef} - \tilde{\bm{\rcoef}}_\res\right)'
\bm{S}^2_{\bm{w}\setminus \bm{x}}
\left(\bm{\rcoef} - \tilde{\bm{\rcoef}}_\res\right).
\end{eqnarray}

Fourth, 
using \eqref{eq:var_proj_proof_rem}, \eqref{eq:var_res_proof_rem}, \eqref{eq:var_tau_beta_proj} and \eqref{eq:var_tau_beta_res}, we have   
\begin{align*}
V_{\tau\tau}(\bm{\beta}_1, \bm{\beta}_0) R^2_{\tau,\bm{x}}(\bm{\beta}_{1}, \bm{\beta}_0) & =
V_{\tau\tau} R^2_{\tau,\bm{x}}\left(
1-R^2_{\proj}
\right) + 
(r_1r_0)^{-1}
\left(
\bm{\rcoef}
-\tilde{\bm{\rcoef}}_\proj
\right)' 
 \bm{S}^2_{\bm{w}\mid \bm{x}}
\left(
\bm{\rcoef}
-\tilde{\bm{\rcoef}}_\proj
\right), \\
V_{\tau\tau}(\bm{\beta}_1, \bm{\beta}_0)\left\{1- R^2_{\tau,\bm{x}}(\bm{\beta}_{1}, \bm{\beta}_0) \right\} & = 
V_{\tau\tau}\left(
1 - R^2_{\tau,\bm{x}}
\right)\left(
1-R^2_{\res}
\right) + 
(r_1r_0)^{-1}\left(\bm{\rcoef} - \tilde{\bm{\rcoef}}_\res\right)'
\bm{S}^2_{\bm{w}\setminus \bm{x}}
\left(\bm{\rcoef} - \tilde{\bm{\rcoef}}_\res\right).
\end{align*}
These coupled with Theorem \ref{thm:adj_rem} imply
Theorem \ref{thm:reg_rem_beta}. 
\end{proof}

\begin{proof}[{\bf Proof of Corollary \ref{cor:analysis_more}}]
First, we prove that $\tilde{\bm{\rcoef}}_{\proj}=\tilde{\bm{\rcoef}}_{\res}=\tilde{\bm{\rcoef}}$ and $R^2_{\proj}=1$.
Under Condition \ref{con:ana_more}, $\hat{\bm{\tau}}_{\bm{w}}$ can linearly represent $\hat{\bm{\tau}}_{\bm{x}}$ as 
$\hat{\bm{\tau}}_{\bm{x}} = \bm{B}_1\hat{\bm{\tau}}_{\bm{w}}$. Thus using the linearity of the projection operator, we have 
\begin{eqnarray}\label{eq:proj_beta_tilde}
\proj(\hat{\tau} \mid \hat{\bm{\tau}}_{\bm{x}}) - \tau -  \tilde{\bm{\rcoef}}'\proj(\hat{\bm{\tau}}_{\bm{w}} \mid \hat{\bm{\tau}}_{\bm{x}}) = 
\proj(\hat{\tau}  - \tau - \tilde{\bm{\rcoef}}'\hat{\bm{\tau}}_{\bm{w}} \mid \hat{\bm{\tau}}_{\bm{x}})
= 
\proj\{
\res(\hat{\tau} \mid \hat{\bm{\tau}}_{\bm{w}}) \mid \hat{\bm{\tau}}_{\bm{x}}
\}.
\end{eqnarray}
By definition, $\res(\hat{\tau} \mid \hat{\bm{\tau}}_{\bm{w}})$ is uncorrelated with $\hat{\bm{\tau}}_{\bm{w}}$, and then it is also uncorrelated with $\hat{\bm{\tau}}_{\bm{x}} = \bm{B}_1\hat{\bm{\tau}}_{\bm{w}}$. Therefore, 
\eqref{eq:proj_beta_tilde} equals zero, 
implying that 
(i) $\proj(\hat{\tau} \mid \hat{\bm{\tau}}_{\bm{x}}) = \tau +  \tilde{\bm{\rcoef}}'\proj(\hat{\bm{\tau}}_{\bm{w}} \mid \hat{\bm{\tau}}_{\bm{x}})$ and $\tilde{\bm{\rcoef}}$ equals the linear projection coefficient of $\proj(\hat{\tau} \mid \hat{\bm{\tau}}_{\bm{x}})$ on $\proj(\hat{\bm{\tau}}_{\bm{w}} \mid \hat{\bm{\tau}}_{\bm{x}})$,  i.e., $\tilde{\bm{\rcoef}}=\tilde{\bm{\rcoef}}_\proj$; 
(ii) the squared multiple correlation between $\proj(\hat{\tau} \mid \hat{\bm{\tau}}_{\bm{x}})$ and $\proj(\hat{\bm{\tau}}_{\bm{w}} \mid \hat{\bm{\tau}}_{\bm{x}})$ equals 1, i.e., $R^2_{\proj}=1$.  
Moreover, (i) and Proposition \ref{lemma:relation_three} imply that $\tilde{\bm{\rcoef}}= \tilde{\bm{\rcoef}}_{\proj}=\tilde{\bm{\rcoef}}_{\res}$. 

Second, we prove that $R^2_{\res} =  (R^2_{\tau,\bm{w}} - R^2_{\tau,\bm{x}})/(1-R^2_{\tau,\bm{x}}).$ 
Because $\tilde{\bm{\rcoef}}_{\res}=\tilde{\bm{\rcoef}},$ 
the residual from the linear projection of $\res(\hat{\tau} \mid \hat{\bm{\tau}}_{\bm{x}})$ on $\res(\hat{\bm{\tau}}_{\bm{w}} \mid \hat{\bm{\tau}}_{\bm{x}})$
reduces to  
\begin{align}\label{eq:res_res_ana_more}
\res\left(\hat{\tau} \mid \hat{\bm{\tau}}_{\bm{x}} \right) - \tilde{\bm{\rcoef}}_{\res}'\res\left(\hat{\bm{\tau}}_{\bm{w}} \mid \hat{\bm{\tau}}_{\bm{x}}\right) & = 
\res\left(
\hat{\tau}- \tilde{\bm{\rcoef}}'\hat{\bm{\tau}}_{\bm{w}} \mid \hat{\bm{\tau}}_{\bm{x}}
\right) = 
\res\left\{
\res\left(\hat{\tau}\mid \hat{\bm{\tau}}_{\bm{w}} \right)
\mid \hat{\bm{\tau}}_{\bm{x}}
\right\}. 
\end{align}
Because, under Condition \ref{con:ana_more}, $\hat{\bm{\tau}}_{\bm{x}} = \bm{B}_1 \hat{\bm{\tau}}_{\bm{w}}$ is uncorrelated with $\res(\hat{\tau}\mid \hat{\bm{\tau}}_{\bm{w}}),$ 
\eqref{eq:res_res_ana_more} reduces to $\res(\hat{\tau}\mid \hat{\bm{\tau}}_{\bm{w}} )$. 
Thus, the squared multiple correlation between 
$\res(\hat{\tau} \mid \hat{\bm{\tau}}_{\bm{x}} )$  and $\res(\hat{\bm{\tau}}_{\bm{w}} \mid \hat{\bm{\tau}}_{\bm{x}})$ reduces to 
\begin{align*}
R^2_{\res} & =  
1 - \frac{\Var\left\{
\res(\hat{\tau}\mid \hat{\bm{\tau}}_{\bm{w}} )
\}
\right\}}{\Var\left\{\res(\hat{\tau} \mid \hat{\bm{\tau}}_{\bm{x}}) \right\}}
 = 1 - \frac{1-R^2_{\tau,\bm{w}}}{1-R^2_{\tau,\bm{x}}} = 
\frac{R^2_{\tau,\bm{w}}-R^2_{\tau,\bm{x}}}{1-R^2_{\tau,\bm{x}}}.
\end{align*}

Corollary \ref{cor:analysis_more} then follows immediately from Theorem \ref{thm:reg_rem_beta}.
\end{proof}

\begin{proof}[{\bf Proof of Corollary \ref{cor:design_more}}]
First, we prove that $\tilde{\bm{\beta}}_{\proj}=\tilde{\bm{\beta}}$, $R^2_{\proj}= R^2_{\tau, \bm{w}}/R^2_{\tau, \bm{x}}$, $\bm{S}^2_{\bm{w}\mid \bm{x}}=\bm{S}^2_{\bm{w}}$, and $\bm{S}^2_{\bm{w}\setminus \bm{x}}=\bm{0}$. 
Under Condition \ref{con:des_more}, 
$\bm{S}^2_{\bm{w}\mid \bm{x}}=\bm{S}^2_{\bm{w}}$,  $\bm{S}^2_{\bm{w}\setminus \bm{x}}=\bm{0}$, 
and 
$\hat{\bm{\tau}}_{\bm{w}} = \bm{B}_2 \hat{\bm{\tau}}_{\bm{x}}$. 
Then using Proposition \ref{lemma:relation_three}, we have $\tilde{\bm{\rcoef}}_{\proj}=\tilde{\bm{\rcoef}}.$ 
Moreover, 
$\proj(\hat{\bm{\tau}}_{\bm{w}} \mid \hat{\bm{\tau}}_{\bm{x}}) = \hat{\bm{\tau}}_{\bm{w}}$, and thus the linear projection of 
$\proj(\hat{\tau} \mid \hat{\bm{\tau}}_{\bm{x}})$ on $\proj(\hat{\bm{\tau}}_{\bm{w}} \mid \hat{\bm{\tau}}_{\bm{x}})$ under the CRE reduces to  
\begin{align*}
\proj\left\{
\proj(\hat{\tau} \mid \hat{\bm{\tau}}_{\bm{x}}) \mid \proj(\hat{\bm{\tau}}_{\bm{w}} \mid \hat{\bm{\tau}}_{\bm{x}})
\right\}
= \tau + \tilde{\bm{\rcoef}}_{\proj}'
\proj(\hat{\bm{\tau}}_{\bm{w}} \mid \hat{\bm{\tau}}_{\bm{x}}) 
= 
\tau + \tilde{\bm{\rcoef}}'
\hat{\bm{\tau}}_{\bm{w}} = \proj(\hat{\tau} \mid \hat{\bm{\tau}}_{\bm{w}} ).
\end{align*}
Consequently, 
the squared multiple correlation between $\proj(\hat{\tau} \mid \hat{\bm{\tau}}_{\bm{x}})$ and  $\proj(\hat{\bm{\tau}}_{\bm{w}} \mid \hat{\bm{\tau}}_{\bm{x}})$ equals 
\begin{align*}
R^2_{\proj} & = 
\frac{
\Var\left\{
\proj(\hat{\tau} \mid \hat{\bm{\tau}}_{\bm{w}} )
\right\}
}{
\Var\left\{
\proj\left(\hat{{\tau}} \mid \hat{\bm{\tau}}_{\bm{x}}\right)
\right\}
}
= 
\frac{\Var(\hat{\tau})R^2_{\tau, \bm{w}}}{\Var(\hat{\tau})R^2_{\tau, \bm{x}}}
= \frac{
R^2_{\tau, \bm{w}}
}{R^2_{\tau, \bm{x}}}.
\end{align*}

Second, we prove that $R^2_{\res} = 0.$
Under Condition \ref{con:des_more}, 
$
\res(\hat{\bm{\tau}}_{\bm{w}} \mid \hat{\bm{\tau}}_{\bm{x}}) = \bm{0},
$
and thus the squared multiple correlation between $\res(\hat{\tau} \mid \hat{\bm{\tau}}_{\bm{x}})$ and 
$\res(\hat{\bm{\tau}}_{\bm{w}} \mid \hat{\bm{\tau}}_{\bm{x}})$ reduces to zero.

Corollary \ref{cor:design_more} then follows immediately from Theorem \ref{thm:reg_rem_beta}.
\end{proof}

\section{$\mathcal{S}$-optimality}
\label{appendix::soptimal}

\subsection{Lemmas}

\begin{lemma}\label{lemma:L_ka_unimodal}
Let $\varepsilon \sim \mathcal{N}(0,1)$, and 
$L_{K,a} \sim D_1 \mid \bm{D}'\bm{D}\leq a$, where $\bm{D} = (D_1,\ldots, D_K)\sim \mathcal{N}(\bm{0}, \bm{I}_K)$. Both $\varepsilon$ and $L_{K,a}$ are symmetric and unimodal around zero. 
\end{lemma}

\begin{proof}[\bf Proof of Lemma \ref{lemma:L_ka_unimodal}]
It follows from \citet[][Proposition 2]{asymrerand2106}.
\end{proof}

\begin{lemma}\label{lemma:order_sum}
Let 
$\zeta_0,\zeta_1$ and $\zeta_2$ be three mutually independent random variables. 
If 
\begin{itemize}
\item[(1)] $\zeta_0$ is symmetric and unimodal around zero; 
\item[(2)] $\zeta_1$ and $\zeta_2$ are symmetric around 0;
\item[(3)] $P(|\zeta_1| \leq c) \geq P(|\zeta_2 | \leq c)$ for any $c \geq 0$; 
\end{itemize}
then $P( | \zeta_0+\zeta_1 | \leq c) \geq P( | \zeta_0+\zeta_2 | \leq c)$ for any $c \geq 0$.
\end{lemma}

\begin{proof}[\bf Proof of Lemma \ref{lemma:order_sum}]
It follows from \citet[][Theorem 7.5]{dharmadhikari1988}.
\end{proof}

\begin{lemma}\label{lemma:qr_linear_comb_epsilon_L_Ka}
Let $\varepsilon \sim \mathcal{N}(0,1)$, 
$L_{K,a} \sim D_1 \mid \bm{D}'\bm{D}\leq a$, where $\bm{D} = (D_1,\ldots, D_K)\sim \mathcal{N}(\bm{0}, \bm{I}_K)$, and $\varepsilon$ and $L_{K,a}$ be mutually independent. 
For any nonnegative constants $b_1\leq c_1$, $b_2\leq c_2$, and any $\alpha \in (0,1)$, 
the $1-\alpha$ quantile range of $b_1\varepsilon + b_2 L_{K,a}$ is narrower than or equal to that of $c_1\varepsilon + c_2 L_{K,a}$. 
\end{lemma}

\begin{proof}[\bf Proof of Lemma \ref{lemma:qr_linear_comb_epsilon_L_Ka}]
From Lemma \ref{lemma:L_ka_unimodal}, $b_1\varepsilon$ is symmetric and unimodal. Because $b_2 \leq c_2$,
$P(| b_2L_{K,a} | \leq c) \geq P( |c_2L_{K,a}| \leq c)$ for any $c\geq 0$. Then from Lemma \ref{lemma:order_sum}, 
$
P( | b_1\varepsilon + b_2 L_{K,a} | \leq c) \geq P( | b_1\varepsilon + c_2L_{K,a} | \leq c)$ for any $c\geq 0.$

From Lemma \ref{lemma:L_ka_unimodal}, $c_2L_{K,a}$ is symmetric and unimodal. Because $b_1 \leq c_1$, $
P( | b_1\varepsilon | \leq c) \geq P( | c_1\varepsilon | \leq c)$ for any $c\geq 0$. Then from Lemma \ref{lemma:order_sum}, 
$P( | b_1\varepsilon + c_2L_{K,a} | \leq c) \geq P( | c_1 \varepsilon + c_2L_{K,a} | \leq c)$ for any $c\geq 0.$

From the above two results, for any $c\geq 0$,
\begin{align*}
P( | b_1\varepsilon + b_2 L_{K,a} | \leq c) \geq 
P( | b_1\varepsilon + c_2L_{K,a} | \leq c) \geq 
P( | c_1 \varepsilon + c_2L_{K,a} | \leq c),
\end{align*}
which implies Lemma \ref{lemma:qr_linear_comb_epsilon_L_Ka}.
\end{proof}

\begin{lemma}\label{lemma:quantile_in_rho}
	For any $\alpha\geq 1/2$, the $\alpha$th quantile of $(1-\rho^2)^{1/2}\cdot\varepsilon+ | \rho | \cdot L_{K,a}$ is nonincreasing in $\rho^2$.  
\end{lemma}

\begin{proof}[\bf Proof of Lemma \ref{lemma:quantile_in_rho}]
It follows from \citet[][Lemma A3]{asymrerand2106}. 
\end{proof}

\subsection{Proofs}

\begin{proof}[{\bf Proof of Theorem \ref{thm:analysis_more_opt}}]
In the asymptotic distribution of $\hat{\tau}(\bm{\beta}_1, \bm{\beta}_0)$ in Corollary \ref{cor:analysis_more}, both coefficients of $\varepsilon$ and $L_{K,a}$ attain their minimum values at $r_1\bm{\beta}_1 + r_1\bm{\beta}_0\equiv \bm{\gamma}  = \tilde{\bm{\gamma}} .$ From Lemma \ref{lemma:qr_linear_comb_epsilon_L_Ka} and Corollary \ref{cor:analysis_more},
the $\mathcal{S}$-optimal adjusted estimator is attainable when $\bm{\rcoef}=\tilde{\bm{\rcoef}}$ or 
$r_0\bm{\beta}_1+r_1\bm{\beta}_0 = r_0 \tilde{\bm{\beta}}_1+r_1 \tilde{\bm{\beta}}_0 ,$ 
with the asymptotic distribution \eqref{eq:opt_ana_more}. 
From Proposition \ref{prop:three_beta_tilde}, $\hat{\tau}(\tilde{\bm{\beta}}_1, \tilde{\bm{\beta}}_0)$ is $\mathcal{S}$-optimal.
\end{proof}

\begin{proof}[{\bf Proof of Theorem \ref{thm:design_more_opt}}]
From Corollary \ref{cor:design_more}, 
in the asymptotic distribution of $\hat{\tau}(\bm{\beta}_1, \bm{\beta}_0)$ under ReM, 
the coefficient of $\varepsilon$ does not depend on $\bm{\rcoef}$, and the coefficient of $L_{K,a}$ attains its minimum when $\bm{\rcoef} = \tilde{\bm{\rcoef}}$. 
From Lemma \ref{lemma:qr_linear_comb_epsilon_L_Ka} and Corollary \ref{cor:design_more}, 
the $\mathcal{S}$-optimal adjusted estimator is 
attainable when $\bm{\rcoef}=\tilde{\bm{\rcoef}}$ or 
$r_0\bm{\beta}_1+r_1\bm{\beta}_0 = r_0 \tilde{\bm{\beta}}_1+r_1 \tilde{\bm{\beta}}_0 ,$  with asymptotic distribution \eqref{eq:design_more_opt}. 
Moreover, from Proposition \ref{prop:three_beta_tilde}, 
$\hat{\tau}(\tilde{\bm{\beta}}_1, \tilde{\bm{\beta}}_0)$ is $\mathcal{S}$-optimal. 
\end{proof}

\begin{proof}[{\bf Proof of Theorem \ref{thm:gen_ana_enough}}]
From Corollary \ref{corr:diff_rem}, the squared coefficient of $\varepsilon$ in the asymptotic distribution of $\hat{\tau}$ is 
$
V_{\tau\tau}(1-R^2_{\tau,\bm{x}}). 
$
From Theorem \ref{thm:adj_rem}, the squared coefficient of $\varepsilon$ in the asymptotic distribution of $\hat{\tau}(\tilde{\bm{\beta}}_1, \tilde{\bm{\beta}}_0)$ is 
$
V_{\tau\tau}(\tilde{\bm{\beta}}_1, \tilde{\bm{\beta}}_0) \{
1-R^2_{\tau, \bm{x}}(\tilde{\bm{\beta}}_1, \tilde{\bm{\beta}}_0)
\}.
$
Because 
$$
V_{\tau\tau}(\tilde{\bm{\beta}}_1, \tilde{\bm{\beta}}_0)
= n\Var\{
\hat{\tau}(\tilde{\bm{\beta}}_1, \tilde{\bm{\beta}}_0) 
\} 
= 
n\Var(
\hat{\tau} - \tilde{\bm{\rcoef}} \hat{\bm{\tau}}_{\bm{w}}
)
= 
n\Var\left\{\res(\hat{\tau} \mid \hat{\bm{\tau}}_{\bm{w}})\right\}
=
V_{\tau\tau}\left( 1-R^2_{\tau, \bm{w}} \right),
$$
the squared coefficient of $\varepsilon$ in the asymptotic distribution of $\hat{\tau}(\tilde{\bm{\beta}}_1, \tilde{\bm{\beta}}_0)$ under ReM reduces to 
$V_{\tau\tau}
\left(1-R^2_{\tau,\bm{w}}\right)\{1 - R^2_{\tau,\bm{x}}(\tilde{\bm{\beta}}_1, \tilde{\bm{\beta}}_0)\}.
$
Therefore, under ReM, the squared coefficient of $\varepsilon$ in the asymptotic distribution of $\hat{\tau}(\tilde{\bm{\beta}}_1, \tilde{\bm{\beta}}_0)$  is smaller than or equal to that of $\hat{\tau}$ if and only if 
\begin{eqnarray*}
& & V_{\tau\tau}
\left(1-R^2_{\tau,\bm{w}}\right)\left\{1 - R^2_{\tau,\bm{x}}(\tilde{\bm{\beta}}_1, \tilde{\bm{\beta}}_0)\right\}
\leq 
V_{\tau\tau}\left(1 - R^2_{\tau,\bm{x}}\right)\\
& \Longleftrightarrow & 
\left(1-R^2_{\tau,\bm{w}}\right)
- \left(1-R^2_{\tau,\bm{w}}\right)R^2_{\tau,\bm{x}}(\tilde{\bm{\beta}}_1, \tilde{\bm{\beta}}_0)
\leq 
\left(1 - R^2_{\tau,\bm{x}}\right)\\
& \Longleftrightarrow & 
R^2_{\tau,\bm{w}} 
+ \left(1-R^2_{\tau,\bm{w}}\right)R^2_{\tau,\bm{x}}(\tilde{\bm{\beta}}_1, \tilde{\bm{\beta}}_0)
\geq R^2_{\tau,\bm{x}}.
\end{eqnarray*}
\end{proof}

\subsection{The adjusted estimator with the smallest asymptotic variance}
From Theorem \ref{thm:reg_rem_beta}, the asymptotic variance of $\hat{\tau}(\bm{\beta}_1, \bm{\beta}_0)$ under ReM is 
\begin{align}\label{eq:var_reg_rem}
&\left\{V_{\tau\tau}\left( 1-R^2_{\tau, \bm{x}} \right)
\left( 1-R^2_{\res} \right)
+ (r_1r_0)^{-1}
\left( \bm{\rcoef}-\tilde{\bm{\rcoef}}_{\res} \right)'
\bm{S}^2_{\bm{w}\setminus \bm{x}}
\left( \bm{\rcoef}-\tilde{\bm{\rcoef}}_{\res} \right)\right\}
\nonumber\\
& \quad \ + 
\left\{
V_{\tau\tau}R^2_{\tau, \bm{x}}
\left( 1 - R^2_{\proj} \right) + (r_1r_0)^{-1}
\left( \bm{\rcoef}-\tilde{\bm{\rcoef}}_{\proj} \right)'
\bm{S}^2_{\bm{w}\mid \bm{x}}
\left( \bm{\rcoef}-\tilde{\bm{\rcoef}}_{\proj} \right)
\right\} v_{K,a}
\nonumber\\
 =  &
V_{\tau\tau} \left( 1-R^2_{\tau, \bm{x}} \right)
\left( 1-R^2_{\res} \right) + 
V_{\tau\tau}R^2_{\tau, \bm{x}}
\left( 1 - R^2_{\proj} \right) v_{K,a} 
\nonumber\\
 & \quad \  + 
(r_1r_0)^{-1}
\left( \bm{\rcoef}-\tilde{\bm{\rcoef}}_{\res} \right)'
\bm{S}^2_{\bm{w}\setminus \bm{x}}
\left( \bm{\rcoef}-\tilde{\bm{\rcoef}}_{\res} \right) + 
(r_1r_0)^{-1}
\left( \bm{\rcoef}-\tilde{\bm{\rcoef}}_{\proj} \right)'
\bm{S}^2_{\bm{w}\mid \bm{x}}
\left( \bm{\rcoef}-\tilde{\bm{\rcoef}}_{\proj} \right) v_{K,a} . 
\end{align}
It is a quadratic form of $\bm{\rcoef}.$ 
The derivative of \eqref{eq:var_reg_rem} with respect to $\bm{\rcoef}$ is   
\begin{eqnarray*}
&&2 (r_1r_0)^{-1}\bm{S}^2_{\bm{w}\setminus \bm{x}}
\left( \bm{\rcoef}-\tilde{\bm{\rcoef}}_{\res} \right) + 
2v_{K,a} (r_1r_0)^{-1} 
\bm{S}^2_{\bm{w}\mid \bm{x}}
\left( \bm{\rcoef}-\tilde{\bm{\rcoef}}_{\proj} \right)\\
& = & 2 (r_1r_0)^{-1}\left\{
\left( 
\bm{S}^2_{\bm{w}\setminus \bm{x}} + v_{K,a}\bm{S}^2_{\bm{w}\mid \bm{x}}
\right) 
\bm{\rcoef} - 
\left(
\bm{S}^2_{\bm{w}\setminus \bm{x}}\tilde{\bm{\rcoef}}_{\res} +
v_{K,a}\bm{S}^2_{\bm{w}\mid \bm{x}}\tilde{\bm{\rcoef}}_{\proj}
\right)
\right\}. 
\end{eqnarray*}
Therefore, under ReM, $\hat{\tau}(\bm{\beta}_1, \bm{\beta}_0)$ with the smallest asymptotic variance is attainable when 
\begin{align*}
r_0 \bm{\beta}_{1} + r_1 \bm{\beta}_{0}
\equiv 
\bm{\rcoef}
=
\left(
\bm{S}^2_{\bm{w}\setminus \bm{x}}  + v_{K,a}\bm{S}^2_{\bm{w}\mid \bm{x}}
\right)^{-1}
\left(
\bm{S}^2_{\bm{w}\setminus \bm{x}} 
\tilde{\bm{\rcoef}}_{\res} + 
v_{K,a}\bm{S}^2_{\bm{w}\mid \bm{x}}\tilde{\bm{\rcoef}}_{\proj}
\right).
\end{align*}
When $a \approx 0$, the above coefficient is close to $\tilde{\bm{\rcoef}}_{\res}.$

 \subsection{Technical details for Remark \ref{rmk:reg_worse}}
We first give equivalent forms for the squared coefficients of $\varepsilon$ in the asymptotic distributions of $\hat{\tau}$ and $\hat{\tau}(\tilde{\bm{\beta}}_1, \tilde{\bm{\beta}}_0)$ under ReM. 
		From Corollary \ref{corr:diff_rem} and the definition of $R^2_{\tau, \bm{x}}$ in \eqref{eq:R2_tau_x}, the squared coefficient of $\varepsilon$ in the asymptotic distribution of $\hat{\tau}$ under ReM has the following equivalent forms:
		\begin{align}\label{eq:var_diff_reg_worse}
		V_{\tau\tau} (1-R^2_{\tau, \bm{x}}) & = n \cdot \Var(\hat{\tau}) \cdot (1-R^2_{\tau, \bm{x}}) = n \cdot \Var\left\{\res\left(\hat{\tau} \mid \hat{\bm{\tau}}_{\bm{x}} \right)\right\}. 
		\end{align}
		From Theorem \ref{thm:adj_rem} and the definition of $R^2_{\tau, \bm{x}}(\bm{\beta}_{1}, \bm{\beta}_0)$ in \eqref{eq:R2_tau_x_beta}, the squared coefficient of $\varepsilon$ in the asymptotic distribution of $\hat{\tau}(\tilde{\bm{\beta}}_1, \tilde{\bm{\beta}}_0)$ under ReM has the following equivalent forms:
		\begin{align}\label{eq:var_adj_reg_worse}
		& \quad \ V_{\tau\tau}(\tilde{\bm{\beta}}_1, \tilde{\bm{\beta}}_0) \left\{1-R^2_{\tau, \bm{x}}(\tilde{\bm{\beta}}_{1}, \tilde{\bm{\beta}}_0)\right\}
		\nonumber
		\\
		& = 
		n \cdot \Var\left\{\hat{\tau}(\tilde{\bm{\beta}}_1, \tilde{\bm{\beta}}_0)\right\} \cdot  \left\{1-R^2_{\tau, \bm{x}}(\tilde{\bm{\beta}}_{1}, \tilde{\bm{\beta}}_0)\right\}
		= n \cdot \Var\left\{
		\res\left(
		\hat{\tau}(\tilde{\bm{\beta}}_1, \tilde{\bm{\beta}}_0) \mid \hat{\bm{\tau}}_{\bm{x}} 
		\right)
		\right\}
		\nonumber
		\\
		& = 
		n \cdot \Var\left\{
		\res\left(
		\hat{\tau}  - \tilde{\bm{\rcoef}}' \hat{\bm{\tau}}_{\bm{w}} \mid \hat{\bm{\tau}}_{\bm{x}} 
		\right)
		\right\}
		\nonumber
		\\
		& = 
		n \cdot \Var\left\{
		\res\left(
		\hat{\tau} \mid \hat{\bm{\tau}}_{\bm{x}} 
		\right)
		-
		\tilde{\bm{\rcoef}}' 
		\res\left(
		\hat{\bm{\tau}}_{\bm{w}} \mid \hat{\bm{\tau}}_{\bm{x}} 
		\right)
		\right\}. 
		\end{align}
		
We then study the covariance between $\res\left(
		\hat{\tau} \mid \hat{\bm{\tau}}_{\bm{x}} 
		\right)$ and $\res\left(
		\hat{\bm{\tau}}_{\bm{w}} \mid \hat{\bm{\tau}}_{\bm{x}} 
		\right)$ under the CRE. 
		For $z=0,1$, let $Y^{\perp}_i(z)$ be the residual from the linear projection of $Y_i(z)$ on $\bm{x}_i$, and $\bm{w}^{\perp}_i$ be the residual from the linear projection of $\bm{w}_i$ on $\bm{x}_i$. 
		We have 
		\begin{align*}
		\res\left(
		\hat{\tau} \mid \hat{\bm{\tau}}_{\bm{x}} 
		\right) & = n_1^{-1} \sum_{i=1}^{n} Z_i Y^{\perp}_i(1) - n_0^{-1} \sum_{i=1}^{n} (1-Z_i) Y^{\perp}_i(0), \\
		\res\left(
		\hat{\bm{\tau}}_{\bm{w}} \mid \hat{\bm{\tau}}_{\bm{x}} 
		\right) & = n_1^{-1} \sum_{i=1}^{n} Z_i \bm{w}_i^{\perp} - n_0^{-1} \sum_{i=1}^{n} (1-Z_i) \bm{w}_i^{\perp},
		\end{align*}
		and thus 
		\begin{align}\label{eq:cov_res_x}
		\Cov\left\{
		\res\left(
		\hat{\tau} \mid \hat{\bm{\tau}}_{\bm{x}} 
		\right), 
		\res\left(
		\hat{\bm{\tau}}_{\bm{w}} \mid \hat{\bm{\tau}}_{\bm{x}} 
		\right)
		\right\}
		& = 
		n_1^{-1} \bm{S}_{Y(1), \bm{w} \mid \bm{x} }  + n_0^{-1} \bm{S}_{Y(0), \bm{w} \mid \bm{x} } ,
		\end{align}
		where $\bm{S}_{Y(z), \bm{w} \mid \bm{x} }$ is the finite population covariance between $Y^{\perp}(z)$ and $\bm{w}^{\perp}$, or, equivalently, the finite population partial covariance between $Y(z)$ and $\bm{w}$ given $\bm{x}$, for $z=0,1$. 
		
We finally prove Remark \ref{rmk:reg_worse}. 
		When $\bm{S}_{Y(z), \bm{w} \mid \bm{x} } = \bm{0}$ for $z=0,1$, from \eqref{eq:cov_res_x}, 
		$\res\left(
		\hat{\tau} \mid \hat{\bm{\tau}}_{\bm{x}} 
		\right)$
		is uncorrelated with 
		$
		\res\left(
		\hat{\bm{\tau}}_{\bm{w}} \mid \hat{\bm{\tau}}_{\bm{x}} 
		\right).$ 
		This further implies that 
		\begin{align*}
		\Var\left\{
		\res\left(
		\hat{\tau} \mid \hat{\bm{\tau}}_{\bm{x}} 
		\right)
		-
		\tilde{\bm{\rcoef}}' 
		\res\left(
		\hat{\bm{\tau}}_{\bm{w}} \mid \hat{\bm{\tau}}_{\bm{x}} 
		\right)
		\right\}
		& = 
		\Var\left\{
		\res\left(
		\hat{\tau} \mid \hat{\bm{\tau}}_{\bm{x}} 
		\right)
		\right\}
		+ 
		\Var\left\{
		\tilde{\bm{\rcoef}}' 
		\res\left(
		\hat{\bm{\tau}}_{\bm{w}} \mid \hat{\bm{\tau}}_{\bm{x}} 
		\right)
		\right\}
		\ge 
		\Var\left\{
		\res\left(
		\hat{\tau} \mid \hat{\bm{\tau}}_{\bm{x}} 
		\right)
		\right\}. 
		\end{align*}
		From \eqref{eq:var_diff_reg_worse} and \eqref{eq:var_adj_reg_worse}, 
		the coefficient of $\varepsilon$ 
		for $\hat{\tau}(\tilde{\bm{\beta}}_1, \tilde{\bm{\beta}}_0)$ is larger than or equal to that for $\hat{\tau}$. 

\section{Estimated distributions and the associated optimality
}\label{app:est_rem}
\label{appendix::estimationsampling}

\subsection{Lemmas}
For treatment arm $z$ $(z=0,1)$, 
let $s^2_{Y(z)}$ be the sample variance of observed outcome, 
$\bm{s}^2_{\bm{w},z}$ be the sample covariance of covariates $\bm{w}$, and $\bm{s}_{Y(z),\bm{w}}$ be the sample covariance between observed outcome and covariates $\bm{w}$.

\begin{lemma}\label{lemma:samp_var_consistent}
Under ReM and Condition \ref{con:fp}, for $z=0,1$, 
\begin{align}\label{eq:samp_var_consistent1}
& s^2_{Y(z)} - S^2_{Y(z)} = o_P(1), \quad 
\bm{s}^2_{\bm{w},z} - \bm{S}^2_{\bm{w}} = o_P(1), \quad 
\bm{s}_{Y(z),\bm{w}} - \bm{S}_{Y(z),\bm{w}} = o_P(1), 
\nonumber
\\ 
& \bm{s}^2_{\bm{x},z} - \bm{S}^2_{\bm{x}} = o_P(1), \quad 
\bm{s}_{Y(z),\bm{x}} - \bm{S}_{Y(z),\bm{x}} = o_P(1), \quad 
\bm{s}_{\bm{w},\bm{x}} - \bm{S}_{\bm{w},\bm{x}} = o_P(1);
\end{align}
and for any $\bm{\beta}_1$ and $\bm{\beta}_0$ that can depend implicitly on sample size $n$ but have finite limits, 
\begin{align}\label{eq:samp_var_consistent2}
s^2_{Y(z;\bm{\beta}_z)} - S^2_{Y(z;\bm{\beta}_z)} = o_P(1), \quad & s_{Y(z;\bm{\beta}_z)\mid \bm{x}}^2  - S_{Y(z;\bm{\beta}_z)\mid \bm{x}}^2 = o_P(1), 
\nonumber
\\
\bm{s}_{Y(z;\bm{\beta}_z),\bm{x}} - \bm{S}_{Y(z;\bm{\beta}_z),\bm{x}} = o_P(1), \quad &
\bm{s}_{Y(z;\bm{\beta}_z),\bm{w}} - \bm{S}_{Y(z;\bm{\beta}_z),\bm{w}} = o_P(1).
\end{align}
\end{lemma}

\begin{proof}[\bf Proof of Lemma \ref{lemma:samp_var_consistent}]
First, 
we can view covariates $\bm{w}$ as ``outcomes'' unaffected by the treatment. 
Thus, 
\eqref{eq:samp_var_consistent1}
follows immediately from \citet[][Lemma A16]{asymrerand2106}.

Second, 
the observed sample variances have the following equivalent forms:
\begin{align*}
s^2_{Y(z;\bm{\beta}_z)}  = s^2_{Y(z)}  + \bm{\beta}_z'\bm{s}^2_{\bm{w},z} \bm{\beta}_z - 2\bm{\beta}_z' \bm{s}_{\bm{w}, Y(z)}, \quad & 
\bm{s}_{Y(z;\bm{\beta}_z),\bm{x}}  = \bm{s}_{Y(z),\bm{x}} - \bm{\beta}_z'\bm{s}_{\bm{w},\bm{x}},
\\
\bm{s}_{Y(z;\bm{\beta}_z),\bm{w}}  = \bm{s}_{Y(z),\bm{w}} - \bm{\beta}_z'\bm{s}_{\bm{w},z}^2, \quad & 
s_{Y(z;\bm{\beta}_z)\mid \bm{x}}^2 = 
s_{Y(z;\bm{\beta}_z), \bm{x}} 
\left( \bm{s}_{\bm{x},z}^2 \right)^{-1}
s_{\bm{x}, Y(z;\bm{\beta}_z)}.
\end{align*}
From 
\eqref{eq:samp_var_consistent1},
$$
s^2_{Y(z;\bm{\beta}_z)} - S^2_{Y(z;\bm{\beta}_z)} = o_P(1), \quad 
\bm{s}_{Y(z;\bm{\beta}_z),\bm{x}} - \bm{S}_{Y(z;\bm{\beta}_z),\bm{x}} = o_P(1), \quad 
\bm{s}_{Y(z;\bm{\beta}_z),\bm{w}} - \bm{S}_{Y(z;\bm{\beta}_z),\bm{w}} = o_P(1).
$$
These further imply that 
$s_{Y(z;\bm{\beta}_z)\mid \bm{x}}^2 - S_{Y(z;\bm{\beta}_z)\mid \bm{x}}^2 = o_P(1)$. 
Thus,  
\eqref{eq:samp_var_consistent2} holds. 

\end{proof}

\begin{lemma}\label{lemma:inf_rem}
Under ReM and Condition \ref{con:fp}, 
$$
\hat{V}_{\tau\tau}(\bm{\beta}_1, \bm{\beta}_0) - \tilde{V}_{\tau\tau}(\bm{\beta}_1, \bm{\beta}_0) = o_P(1),\quad
\hat{R}^2_{\tau, \bm{x}}(\bm{\beta}_{1}, \bm{\beta}_0) 
- 
\tilde{R}^2_{\tau, \bm{x}}(\bm{\beta}_{1}, \bm{\beta}_0) = o_P(1),
$$
where 
\begin{align*}
	\tilde{V}_{\tau\tau}(\bm{\beta}_1, \bm{\beta}_0)\equiv V_{\tau\tau}(\bm{\beta}_1, \bm{\beta}_0)+S^2_{\tau \setminus \bm{w}},\quad 
	\tilde{R}^2_{\tau, \bm{x}}(\bm{\beta}_{1}, \bm{\beta}_0)  \equiv   
	\tilde{V}_{\tau\tau}^{-1}(\bm{\beta}_1, \bm{\beta}_0)
	V_{\tau\tau}(\bm{\beta}_1, \bm{\beta}_0)
	R^2_{\tau, \bm{x}}(\bm{\beta}_1, \bm{\beta}_0).
\end{align*} 
\end{lemma}

\begin{proof}[\bf Proof of Lemma \ref{lemma:inf_rem}]
From \eqref{eq:V_beta_hat}, \eqref{eq:R2_beta_hat} and Lemma \ref{lemma:samp_var_consistent}, 
\begin{align*}
\hat{V}_{\tau\tau}(\bm{\beta}_1, \bm{\beta}_0)  & =  
r_1^{-1}S^2_{Y(1;\bm{\beta}_1)} + r_0^{-1}S^2_{Y(0;\bm{\beta}_0)} - 
\bm{S}^2_{\tau(\bm{\beta}_1, \bm{\beta}_0) \mid \bm{w}}+ o_P(1) \\
&= V_{\tau\tau}(\bm{\beta}_1, \bm{\beta}_0) + \bm{S}^2_{\tau(\bm{\beta}_1, \bm{\beta}_0) \setminus \bm{w}} + o_P(1)\\
& = V_{\tau\tau}(\bm{\beta}_1, \bm{\beta}_0) + \bm{S}^2_{\tau \setminus \bm{w}} + o_P(1)
 = \tilde{V}_{\tau\tau}(\bm{\beta}_1, \bm{\beta}_0) + o_P(1), \\
\hat{R}^2_{\tau, \bm{x}}(\bm{\beta}_{1}, \bm{\beta}_0)
& = 
\tilde{V}_{\tau\tau}^{-1}(\bm{\beta}_1, \bm{\beta}_0)^{-1}
\left(
r_1^{-1}S_{Y(1;\bm{\beta}_1)\mid \bm{x}}^2+r_0^{-1}S_{Y(0;\bm{\beta}_0)\mid \bm{x}}^2
 - 
\bm{S}^2_{\tau(\bm{\beta}_1, \bm{\beta}_0) \mid \bm{x}}\right) + o_P(1)\\
& = 
\tilde{V}_{\tau\tau}^{-1}(\bm{\beta}_1, \bm{\beta}_0)
V_{\tau\tau}(\bm{\beta}_1, \bm{\beta}_0) \cdot
V_{\tau\tau}^{-1}(\bm{\beta}_1, \bm{\beta}_0)\left(
r_1^{-1}S_{Y(1;\bm{\beta}_1)\mid \bm{x}}^2+r_0^{-1}S_{Y(0;\bm{\beta}_0)\mid \bm{x}}^2
 - 
\bm{S}^2_{\tau(\bm{\beta}_1, \bm{\beta}_0) \mid \bm{x}}\right)+ o_P(1)\\
& = \tilde{V}_{\tau\tau}^{-1}(\bm{\beta}_1, \bm{\beta}_0)
V_{\tau\tau}(\bm{\beta}_1, \bm{\beta}_0)
R^2_{\tau, \bm{x}}(\bm{\beta}_{1}, \bm{\beta}_0)+ o_P(1) 
= 
\tilde{R}^2_{\tau, \bm{x}}(\bm{\beta}_{1}, \bm{\beta}_0) + o_P(1).
\end{align*}
\end{proof}

\begin{lemma}\label{lemma:V_tau_beta_decomp}
$
V_{\tau\tau}(\bm{\beta}_1, \bm{\beta}_0) = 
V_{\tau\tau}(1-R^2_{\tau, \bm{w}})
+ 
(r_1r_0)^{-1}(\bm{\rcoef}-\tilde{\bm{\rcoef}})'
\bm{S}^2_{\bm{w}}
	(\bm{\rcoef}-\tilde{\bm{\rcoef}}).
$
\end{lemma}

\begin{proof}[{\bf Proof of Lemma \ref{lemma:V_tau_beta_decomp}}]
It follows from Theorem \ref{thm:adj_rem} and Corollary \ref{cor:analysis_more} with $a=\infty$ and $\bm{x} = \emptyset$. 
We give a more direct proof below:
\begin{align*}
V_{\tau\tau}(\bm{\beta}_1, \bm{\beta}_0) & = 
n\Var\left(
\hat{\tau} - \tau - \bm{\rcoef}'\hat{\bm{\tau}}_{\bm{w}}
\right)
= 
n\Var\left\{
\hat{\tau} - \tau - \tilde{\bm{\rcoef}}'\hat{\bm{\tau}}_{\bm{w}} - 
\left(
\bm{\rcoef} -  \tilde{\bm{\rcoef}}
\right)'\hat{\bm{\tau}}_{\bm{w}}
\right\}\\
& = 
n\Var\left\{
\res(\hat{\tau} \mid \hat{\bm{\tau}}_{\bm{w}}) - 
\left(
\bm{\rcoef} -  \tilde{\bm{\rcoef}}
\right)'\hat{\bm{\tau}}_{\bm{w}}
\right\}
= n\Var\left\{
\res(\hat{\tau} \mid \hat{\bm{\tau}}_{\bm{w}}) 
\right\}
+ 
n\left(
\bm{\rcoef} -  \tilde{\bm{\rcoef}}
\right)'\Cov\left(
\hat{\bm{\tau}}_{\bm{w}}
\right)\left(
\bm{\rcoef} -  \tilde{\bm{\rcoef}}
\right)\\
& = V_{\tau\tau}(1-R^2_{\tau, \bm{w}})
+ 
(r_1r_0)^{-1}(\bm{\rcoef}-\tilde{\bm{\rcoef}})'
\bm{S}^2_{\bm{w}}
	(\bm{\rcoef}-\tilde{\bm{\rcoef}}).
\end{align*}
\end{proof}

\subsection{Proofs}


\begin{proof}[{\bf Proof of Theorem \ref{thm:asymp_behavior_var_ci_ana_more}}]
From Lemma \ref{lemma:inf_rem}, under ReM, the probability limit of the estimated distribution of $\hat{\tau}(\bm{\beta}_1, \bm{\beta}_0)$ is 
$$
\tilde{V}_{\tau\tau}^{1/2}(\bm{\beta}_1, \bm{\beta}_0)
\left[
\left\{1-\tilde{R}^2_{\tau, \bm{x}}(\bm{\beta}_{1}, \bm{\beta}_0) \right\}^{1/2}\cdot \varepsilon + \left\{\tilde{R}^2_{\tau, \bm{x}}(\bm{\beta}_{1}, \bm{\beta}_0) \right\}^{1/2} \cdot  L_{K,a}
\right].
$$
By definition, it has the following equivalent forms:
\begin{align*}
&  
\left\{\tilde{V}_{\tau\tau}(\bm{\beta}_1, \bm{\beta}_0)-\tilde{V}_{\tau\tau}(\bm{\beta}_1, \bm{\beta}_0)\tilde{R}^2_{\tau, \bm{x}}(\bm{\beta}_{1}, \bm{\beta}_0) \right\}^{1/2}\cdot \varepsilon + 
\left\{\tilde{V}_{\tau\tau}(\bm{\beta}_1, \bm{\beta}_0)\tilde{R}^2_{\tau, \bm{x}}(\bm{\beta}_{1}, \bm{\beta}_0)\right\}^{1/2} \cdot  L_{K,a}
\\
 \sim &
\left\{V_{\tau\tau}(\bm{\beta}_1, \bm{\beta}_0) + S^2_{\tau \setminus \bm{w}}-{V}_{\tau\tau}(\bm{\beta}_1, \bm{\beta}_0){R}^2_{\tau, \bm{x}}(\bm{\beta}_{1}, \bm{\beta}_0) \right\}^{1/2} \cdot \varepsilon + \left\{{V}_{\tau\tau}(\bm{\beta}_1, \bm{\beta}_0){R}^2_{\tau, \bm{x}}(\bm{\beta}_{1}, \bm{\beta}_0)\right\}^{1/2} \cdot  L_{K,a}
\\
  \sim  &
\left\{V_{\tau\tau}(\bm{\beta}_1, \bm{\beta}_0)\left\{
1 - {R}^2_{\tau, \bm{x}}(\bm{\beta}_{1}, \bm{\beta}_0)
\right\} + S^2_{\tau \setminus \bm{w}} \right\}^{1/2} \cdot \varepsilon + \left\{{V}_{\tau\tau}(\bm{\beta}_1, \bm{\beta}_0){R}^2_{\tau, \bm{x}}(\bm{\beta}_{1}, \bm{\beta}_0) \right\}^{1/2} \cdot  L_{K,a}. 
\end{align*}
Because Condition \ref{con:ana_more_bal_criterion} implies Condition \ref{con:ana_more},
from Corollary \ref{cor:analysis_more} and its proof, we can further write the probability limit of the estimated distribution of $\hat{\tau}(\bm{\beta}_1, \bm{\beta}_0)$ as
$$
  \left\{
V_{\tau\tau}(1-R^2_{\tau, \bm{w}})
+
S^2_{\tau\setminus \bm{w}}
+ 
(r_1r_0)^{-1}(\bm{\rcoef}-\tilde{\bm{\rcoef}})'
\bm{S}^2_{\bm{w}\setminus \bm{x}}
	(\bm{\rcoef}-\tilde{\bm{\rcoef}})
\right\}^{1/2} \cdot \varepsilon   
 + 
\left\{ 
(r_1r_0)^{-1}
(\bm{\rcoef}-\tilde{\bm{\rcoef}})'
\bm{S}^2_{\bm{w}\mid \bm{x}}
(\bm{\rcoef}-\tilde{\bm{\rcoef}})
\right\}^{1/2} \cdot L_{K,a}.
$$
\end{proof}

\begin{proof}[{\bf Proof of Corollary \ref{corr:infer_eff_cre}}]
It follows  from Theorem \ref{thm:asymp_behavior_var_ci_ana_more} with $a=\infty$ and $\bm{x} = \emptyset$. 
\end{proof}

\begin{proof}[{\bf Proof of Corollary \ref{corr:est_dist_nonadjust}}]
With $\bm{\beta}_1 = \bm{\beta}_0 = \bm{0}$,
Theorem \ref{thm:asymp_behavior_var_ci_ana_more} implies that the the probability limit of the estimated distribution of $\hat{\tau}$ is 
\begin{align*}
&\ \tilde{V}_{\tau\tau}^{1/2}
\left(
1-\tilde{R}^2_{\tau, \bm{x}}
\right)^{1/2}
\cdot \varepsilon + 
\tilde{V}_{\tau\tau}^{1/2} \left(\tilde{R}^2_{\tau, \bm{x}} \right)^{1/2} \cdot  L_{K,a}\\ 
  \sim & \ 
\left( \tilde{V}_{\tau\tau}
-\tilde{V}_{\tau\tau}\tilde{R}^2_{\tau, \bm{x}}
\right)^{1/2} \cdot \varepsilon + 
\left( \tilde{V}_{\tau\tau}\tilde{R}^2_{\tau, \bm{x}} \right)^{1/2} \cdot  L_{K,a}\\
  \sim &\ 
\left( V_{\tau\tau} + S^2_{\tau\setminus \bm{w}}
-V_{\tau\tau} R^2_{\tau, \bm{x}}
\right)^{1/2} \cdot \varepsilon + 
\left( V_{\tau\tau} R^2_{\tau, \bm{x}} \right)^{1/2} \cdot  L_{K,a}\\
  \sim & \ 
\left\{V_{\tau\tau}(1-R^2_{\tau,\bm{x}}) + S^2_{\tau\setminus \bm{w}} \right\}^{1/2} \cdot \varepsilon + 
\left(V_{\tau\tau} R^2_{\tau,\bm{x}} \right)^{1/2} \cdot L_{K,a}. 
\end{align*}
\end{proof}

\begin{proof}[{\bf Proof of Theorem \ref{thm:conf_general}}]
From  Lemma \ref{lemma:inf_rem}, the probability limit of the estimated distribution of $\hat{\tau}(\bm{\beta}_1, \bm{\beta}_0)$ in \eqref{eq::incompleteinformation} is $ \tilde{V}_{\tau\tau}^{1/2}(\bm{\beta}_1, \bm{\beta}_0) \cdot \varepsilon$. 
From Lemma \ref{lemma:V_tau_beta_decomp}, this probability limit has the following equivalent forms:
\begin{align*}
\tilde{V}_{\tau\tau}^{1/2}(\bm{\beta}_1, \bm{\beta}_0) \cdot \varepsilon & 
\sim \left\{  V_{\tau\tau}(\bm{\beta}_1, \bm{\beta}_0) +  S^2_{\tau\setminus \bm{w}} \right\}^{1/2}\cdot \varepsilon \\
& \sim 
\left\{ V_{\tau\tau}(1-R^2_{\tau, \bm{w}}) +  S^2_{\tau\setminus \bm{w}}
+ 
(r_1r_0)^{-1}(\bm{\rcoef}-\tilde{\bm{\rcoef}})'
\bm{S}^2_{\bm{w}}
	(\bm{\rcoef}-\tilde{\bm{\rcoef}})  \right\}^{1/2} \cdot \varepsilon.
\end{align*}
\end{proof}

\begin{proof}[{\bf Proof of Corollary \ref{cor:Copt_ana_know_all}}]
From Theorem \ref{thm:asymp_behavior_var_ci_ana_more}, in the probability limit of the estimated distribution of $\hat{\tau}(\bm{\beta}_1, \bm{\beta}_0)$ in \eqref{eq::samplingdistribution}, both coefficients of $\varepsilon$ and $L_{K,a}$ attain their minimum values at $r_0\bm{\beta}_1+r_1\bm{\beta}_0 \equiv \bm{\rcoef} = \tilde{\bm{\rcoef}}$. 
Lemma \ref{lemma:qr_linear_comb_epsilon_L_Ka} then implies that the optimal
adjusted estimator among \eqref{eq:reg} in terms of the estimated precision is attainable when $r_0\bm{\beta}_1+r_1\bm{\beta}_0 \equiv \bm{\rcoef} = \tilde{\bm{\rcoef}}$. The corresponding probability limit of the estimated distribution is 
$ 
\{
V_{\tau\tau}(1-R^2_{\tau, \bm{w}})
+
S^2_{\tau\setminus \bm{w}}
\}^{1/2} \cdot \varepsilon.
$
\end{proof}

\begin{proof}[{\bf Proof of Corollary \ref{cor:opt_conf_general}}]
It follows immediately from Theorem \ref{thm:conf_general}. 
\end{proof}

\subsection{Additional comments on the asymptotic conservativeness under ReM}

First, we consider the scenario under Condition \ref{con:ana_more_bal_criterion}. From Lemma \ref{lemma:qr_linear_comb_epsilon_L_Ka}, Corollary \ref{cor:analysis_more} and Theorem \ref{thm:asymp_behavior_var_ci_ana_more}, the probability limit of the estimated distribution of $\hat{\tau}(\bm{\beta}_1, \bm{\beta}_0)$ has larger variance and wider quantile ranges than the asymptotic distribution of $\hat{\tau}(\bm{\beta}_1, \bm{\beta}_0)$. Therefore, both the variance estimator and confidence intervals are asymptotically conservative. 

Second, we consider general scenario without Condition \ref{con:ana_more_bal_criterion}.
Theorem \ref{thm:conf_general} implies that the probability limit of the estimated distribution of $\hat{\tau}(\bm{\beta}_1, \bm{\beta}_0)$ has larger variance and wider quantile ranges than $V^{1/2}_{\tau\tau}(\bm{\beta}_1, \bm{\beta}_0) \cdot \varepsilon$.
Lemma \ref{lemma:quantile_in_rho} implies that $V^{1/2}_{\tau\tau}(\bm{\beta}_1, \bm{\beta}_0) \cdot \varepsilon$ has larger variance and wider quantile ranges than
the asymptotic distribution of $\hat{\tau}(\bm{\beta}_1, \bm{\beta}_0)$ in \eqref{eq:adj_rem}. 
Therefore, both the variance estimator and the Wald-type confidence intervals are asymptotically conservative.

\section{Gains from the analyzer and the designer}\label{sec:gains_proof}

\begin{proof}[{\bf Proof of Corollary \ref{cor:ana_more_opt_diff}}]
	First, we compare the asymptotic variances. 
	From Corollary \ref{corr:diff_rem}, 
	the asymptotic variance of $\hat{\tau}$ is 
	$V_{\tau\tau}\{
	1 - (1-v_{K,a})R^2_{\tau, \bm{x}}
	\}$. From Theorem \ref{thm:analysis_more_opt}, the asymptotic variance of $\hat{\tau}(\tilde{\bm{\beta}}_1, \tilde{\bm{\beta}}_0)$ is
	$
	V_{\tau\tau}(1-R^2_{\tau,\bm{w}}).
	$
	Compared to $\hat{\tau}$, the percentage reduction in the asymptotic variance of $\hat{\tau}(\tilde{\bm{\beta}}_1, \tilde{\bm{\beta}}_0)$ is 
	\begin{align*}
	1 - 
	\frac{1-R^2_{\tau,\bm{w}}
	}{
		1 - (1-v_{K,a})R^2_{\tau, \bm{x}}
	}
	= 
	\frac{
		R^2_{\tau,\bm{w}} - (1-v_{K,a})R^2_{\tau, \bm{x}}
	}{
		1 - (1-v_{K,a})R^2_{\tau, \bm{x}}
	}.
	\end{align*}

	Second, we compare the asymptotic quantile ranges. From Corollary \ref{corr:diff_rem}, the length of the asymptotic $1-\alpha$ quantile range of 
	$\hat{\tau}$ is $2V_{\tau\tau}^{1/2} \cdot q_{1-\alpha/2}(R^2_{\tau,\bm{x}})$. From Theorem \ref{thm:analysis_more_opt}, the length of the asymptotic $1-\alpha$ quantile range of $\hat{\tau}(\tilde{\bm{\beta}}_1, \tilde{\bm{\beta}}_0)$ is $2V_{\tau\tau}^{1/2}(1-R^2_{\tau,\bm{w}})^{1/2} \cdot q_{1-\alpha/2}(0)$. Compared to $\hat{\tau}$, the percentage reduction in the length of the asymptotic $1-\alpha$ quantile range of $\hat{\tau}(\tilde{\bm{\beta}}_1, \tilde{\bm{\beta}}_0)$ is 
	\begin{align*}
	1 - \frac{
		2V_{\tau\tau}^{1/2}(1-R^2_{\tau,\bm{w}})^{1/2} \cdot q_{1-\alpha/2}(0)
	}{
		2V_{\tau\tau}^{1/2} \cdot q_{1-\alpha/2}(R^2_{\tau,\bm{x}})
	}
	= 1 - \left(1-R^2_{\tau,\bm{w}} \right)^{1/2}
	\cdot 
	\frac{q_{1-\alpha/2}(0)}{q_{1-\alpha/2}(R^2_{\tau,\bm{x}})}.
	\end{align*}
	
	Third, because $\hat{\tau}(\tilde{\bm{\beta}}_1, \tilde{\bm{\beta}}_0)$ is $\mathcal{S}$-optimal, both percentage reductions in the variance and the $1-\alpha$ quantile range are nonnegative.  
	It is easy to verify that they are both nondecreasing in $R^2_{\tau,\bm{w}}$. 
\end{proof}

\begin{proof}[{\bf Proof of Corollary \ref{cor:ana_less_reg_diff_comp}}]
	First, we compare the asymptotic variances. 
	From Corollary \ref{corr:diff_rem},  
	the asymptotic variance of $\hat{\tau}$ is $V_{\tau\tau}\{
	1 - \left(1-v_{K,a}\right)R^2_{\tau, \bm{x}}
	\}.$
	From Theorem \ref{thm:design_more_opt}, the asymptotic variance of $\hat{\tau}(\tilde{\bm{\beta}}_1, \tilde{\bm{\beta}}_0)$ is $V_{\tau\tau}
	\{
	1-(1-v_{K,a})R^2_{\tau, \bm{x}} - v_{K,a}R^2_{\tau, \bm{w}}
	\}$.
	Compared to $\hat{\tau}$, the percentage reduction in the asymptotic variance of $\hat{\tau}(\tilde{\bm{\beta}}_1, \tilde{\bm{\beta}}_0)$ is 
	\begin{align*}
	1 - \frac{
		V_{\tau\tau}
		\{
		1-(1-v_{K,a})R^2_{\tau, \bm{x}} - v_{K,a}R^2_{\tau, \bm{w}}
		\}
	}{
		V_{\tau\tau}\{
		1 - \left(1-v_{K,a}\right)R^2_{\tau, \bm{x}}
		\}
	}
	= 
	\frac{v_{K,a}R^2_{\tau, \bm{w}}}{
		1 - \left(1-v_{K,a}\right)R^2_{\tau, \bm{x}}
	}.
	\end{align*}

	Second, we compare the asymptotic quantile ranges. 
	From Corollary \ref{corr:diff_rem}, 
	the length of the asymptotic $1-\alpha$  quantile range of $\hat{\tau}$ is $2 V_{\tau\tau}^{1/2} \cdot q_{1-\alpha/2}(R^2_{\tau,\bm{x}}).$ 
	From Theorem \ref{thm:design_more_opt}, 
	the length of the asymptotic $1-\alpha$  quantile range of $\hat{\tau}(\tilde{\bm{\beta}}_1, \tilde{\bm{\beta}}_0)$ is $2 V_{\tau\tau}^{1/2} (1-R^2_{\tau, \bm{w}})^{1/2} \cdot q_{1-\alpha/2}(\rho^2_{\tau, \bm{x}\setminus \bm{w}})$. 
	Compared to $\hat{\tau}$, 
	the percentage reduction in the asymptotic $1-\alpha$ quantile range of $\hat{\tau}(\tilde{\bm{\beta}}_1, \tilde{\bm{\beta}}_0)$ is 
	\begin{align*}
	1 - \left( 1-R^2_{\tau, \bm{w}} \right)^{1/2} \cdot 
		q_{1-\alpha/2}(\rho^2_{\tau, \bm{x}\setminus \bm{w}})
/
		q_{1-\alpha/2}(R^2_{\tau,\bm{x}}).
	\end{align*}
	
	Third, because $\hat{\tau}(\tilde{\bm{\beta}}_1, \tilde{\bm{\beta}}_0)$ is $\mathcal{S}$-optimal, both percentage reductions in variance and $1-\alpha$ quantile range are nonnegative.  
	It is easy to verify that the percentage reduction in the variance is nondecreasing in $R^2_{\tau,\bm{w}}$. 
	For the quantile range, from Lemma \ref{lemma:qr_linear_comb_epsilon_L_Ka}, 
	$(1-R^2_{\tau, \bm{w}})^{1/2} \cdot q_{1-\alpha/2}(\rho^2_{\tau, \bm{x}\setminus \bm{w}})$, the $(1-\alpha/2)$th quantile of $(1-R^2_{\tau, \bm{x}})^{1/2} \cdot  \varepsilon + 
	(R^2_{\tau, \bm{x}}-R^2_{\tau, \bm{w}})^{1/2} \cdot  L_{K,a}$, is nonincreasing  in $R^2_{\tau,\bm{w}}$. 
	Hence the percentage reduction in the $1-\alpha$ quantile range is nondecreasing in $R^2_{\tau,\bm{w}}$.
\end{proof}

\begin{proof}[{\bf Proof of Corollary \ref{cor:ana_less_CRE_ReM_comp}}]
	First, we compare the asymptotic variances. 
	From Section \ref{sec:sopt_cre}, 
	the asymptotic variance of $\hat{\tau}(\tilde{\bm{\beta}}_1, \tilde{\bm{\beta}}_0)$ under the CRE is 
	$
	V_{\tau\tau}(1-R^2_{\tau, \bm{w}}).
	$
	From Theorem \ref{thm:design_more_opt}, the asymptotic variance of $\hat{\tau}(\tilde{\bm{\beta}}_1, \tilde{\bm{\beta}}_0)$ under ReM is 
	$V_{\tau\tau}(1-R^2_{\tau, \bm{w}})\{
	1 - (1-v_{K,a})\rho^2_{\tau, \bm{x}\setminus \bm{w}}
	\}.$
	Compared to the CRE, the percentage reduction in the asymptotic variance under ReM is 
	\begin{align*}
	1 - \frac{
		V_{\tau\tau}(1-R^2_{\tau, \bm{w}})\{
		1 - (1-v_{K,a})\rho^2_{\tau, \bm{x}\setminus \bm{w}}
		\}
	}{
		V_{\tau\tau}(1-R^2_{\tau, \bm{w}})
	}
	= (1-v_{K,a})\rho^2_{\tau, \bm{x}\setminus \bm{w}}.
	\end{align*}
	
	Second, we compare the asymptotic quantile ranges. 
	From Section \ref{sec:sopt_cre}, 
	the length of the asymptotic $1-\alpha$ quantile range of $\hat{\tau}(\tilde{\bm{\beta}}_1, \tilde{\bm{\beta}}_0)$ under the CRE is 
	$2 V_{\tau\tau}^{1/2}(1 - R^2_{\tau,  \bm{w}})^{1/2} \cdot q_{1-\alpha/2}(0)$.
	From Theorem \ref{thm:design_more_opt}, 
	the length of the asymptotic $1-\alpha$ quantile range of $\hat{\tau}(\tilde{\bm{\beta}}_1, \tilde{\bm{\beta}}_0)$ under ReM is 
	$2 V_{\tau\tau}^{1/2}(1-R^2_{\tau, \bm{w}})^{1/2} \cdot q_{1-\alpha/2}(\rho^2_{\tau, \bm{x}\setminus \bm{w}})$. 
	Compared to the CRE,  
	the percentage reduction in the length of the asymptotic $1-\alpha$ quantile range under ReM is 
	$ 
	1-
	q_{1-\alpha/2}(\rho^2_{\tau, \bm{x}\setminus \bm{w}})/
	q_{1-\alpha/2}(0).
	$ 
	
	Third, 
	from Lemma \ref{lemma:quantile_in_rho}, 
	both percentage reductions in variance and $1-\alpha$ quantile range are nonnegative and nondecreasing in $\rho^2_{\tau, \bm{x}\setminus \bm{w}} = (R^2_{\tau, \bm{x}}-R^2_{\tau, \bm{w}})/(1 - R^2_{\tau, \bm{w}})$. 
	Consequently, both percentage reductions are nondecreasing in $R^2_{\tau, \bm{x}}$. 
\end{proof}

\begin{proof}[{\bf Proof of Corollary \ref{cor:infer_ana_more_role_ana}}]
	Recall that $\kappa=1+V^{-1}_{\tau\tau}S^2_{\tau\setminus \bm{w}}\geq 1$. 
	From Corollary \ref{corr:est_dist_nonadjust}, under ReM and Conditions \ref{con:fp} and \ref{con:ana_more_bal_criterion}, 
	the probability limit of the estimated distribution of $\hat{\tau}$ is 
	\begin{eqnarray*}
		& & \left\{V_{\tau\tau}(1-R^2_{\tau,\bm{x}}) + S^2_{\tau\setminus \bm{w}} \right\}^{1/2} \cdot \varepsilon + \left\{V_{\tau\tau}R^2_{\tau,\bm{x}} \right\}^{1/2} \cdot L_{K,a}\\
		& \sim &
		V_{\tau\tau}^{1/2}
		\left\{
		\left(\kappa-R^2_{\tau,\bm{x}}\right)^{1/2} \cdot \varepsilon + \left( R^2_{\tau,\bm{x}} \right)^{1/2} \cdot L_{K,a}
		\right\}
		\\
		& \sim &
		\kappa^{1/2} V_{\tau\tau}^{1/2}
		\left\{
		\left( 1-R^2_{\tau,\bm{x}}/\kappa \right)^{1/2} \cdot \varepsilon + \left( R^2_{\tau,\bm{x}}/\kappa\right)^{1/2} \cdot L_{K,a}
		\right\}.
	\end{eqnarray*}
	From Corollary \ref{cor:Copt_ana_know_all}, under ReM and Conditions \ref{con:fp} and \ref{con:ana_more_bal_criterion}, 
	the probability limit of the estimated distribution of $\hat{\tau}(\tilde{\bm{\beta}}_1, \tilde{\bm{\beta}}_0)$ is 
	$$
	\left\{
	V_{\tau\tau}(1-R^2_{\tau, \bm{w}})
	+
	S^2_{\tau\setminus \bm{w}}
	\right\}^{1/2} \cdot \varepsilon \sim 
	V_{\tau\tau}^{1/2}
	\left(
	\kappa - R^2_{\tau, \bm{w}}
	\right)^{1/2} \cdot \varepsilon
	\sim 
	\kappa^{1/2} V_{\tau\tau}^{1/2}
	\left(
	1 - R^2_{\tau, \bm{w}}/\kappa
	\right)^{1/2} \cdot \varepsilon.
	$$
	
	First, we compare the variances. 
	The variance of the probability limit of the estimated distributions of $\hat{\tau}$ is
	$
	\kappa V_{\tau\tau} 
	\{
	1 - (1-v_{K,a})R^2_{\tau,\bm{x}}/\kappa
	\}. 
	$
	The variance of the probability limit of the estimated distributions of  $\hat{\tau}(\tilde{\bm{\beta}}_1, \tilde{\bm{\beta}}_0)$ is 
	$\kappa V_{\tau\tau}(1 - R^2_{\tau, \bm{w}}/\kappa)$.
	Compared to $\hat{\tau}$, 
	the percentage reduction in variance of the probability limit of the estimated distribution of $\hat{\tau}(\tilde{\bm{\beta}}_1, \tilde{\bm{\beta}}_0)$ is 
	\begin{align*}
	1 - \frac{
		1 - R^2_{\tau, \bm{w}}/\kappa
	}{
		1  - (1-v_{K,a})R^2_{\tau,\bm{x}}/\kappa
	}
	= \frac{
		R^2_{\tau, \bm{w}} - (1-v_{K,a})R^2_{\tau,\bm{x}}
	}{
		\kappa  - (1-v_{K,a})R^2_{\tau,\bm{x}}
	}.
	\end{align*}

	Second, we compare the quantile ranges. 
	The length of $1-\alpha$ quantile range of the probability limit of the estimated distribution of  $\hat{\tau}$ is 
	$2\kappa^{1/2} V_{\tau\tau}^{1/2}\cdot q_{1-\alpha/2}(R^2_{\tau,\bm{x}}/\kappa)$. 
	The length of $1-\alpha$ quantile range of the probability limit of the estimated distribution of  $\hat{\tau}(\tilde{\bm{\beta}}_1, \tilde{\bm{\beta}}_0)$ is 
	$2\kappa^{1/2} V_{\tau\tau}^{1/2}
	(
	1 - R^2_{\tau, \bm{w}}/\kappa
	)^{1/2}\cdot q_{1-\alpha/2}(0)$. 
	Compared to $\hat{\tau}$, 
	the percentage reduction in $1-\alpha$ quantile range of the probability limit of the estimated distribution of $\hat{\tau}(\tilde{\bm{\beta}}_1, \tilde{\bm{\beta}}_0)$ is 
	\begin{align*}
	1 - \frac{
		q_{1-\alpha/2}(0) \left( 1 - R^2_{\tau, \bm{w}}/\kappa \right)^{1/2}
	}{
		q_{1-\alpha/2}(R^2_{\tau,\bm{x}}/\kappa)
	}
	= 
	1 - \left( 1 - R^2_{\tau, \bm{w}}/\kappa \right)^{1/2} \cdot \frac{q_{1-\alpha/2}(0)}{q_{1-\alpha/2}(R^2_{\tau,\bm{x}}/\kappa)}.
	\end{align*}

	Third, the 	optimality of $\hat{\tau}(\tilde{\bm{\beta}}_1, \tilde{\bm{\beta}}_0)$ 
in terms of the estimated precision 
	implies that both percentage reductions are nonnegative. It is easy to verify that they are both nondecreasing in $R^2_{\tau, \bm{w}}$.
\end{proof}

\begin{proof}[{\bf Proof of Corollary \ref{cor:infence_inf_reg_diff}}]
	From
	Theorem \ref{thm:conf_general}, 
	the probability limit of the estimated distribution of $\hat{\tau}$ is
	$
	\tilde{V}_{\tau\tau}^{1/2} \cdot\varepsilon \sim 
 (V_{\tau\tau} + S^2_{\tau\setminus \bm{w}} )^{1/2} \cdot \varepsilon. 
	$
	From Corollary \ref{cor:opt_conf_general}, 
	the probability limit of the estimated distribution of $\hat{\tau}(\tilde{\bm{\beta}}_1, \tilde{\bm{\beta}}_0)$ is 
	$
	\{V_{\tau\tau}(1-R^2_{\tau, \bm{w}})
	+
	S^2_{\tau\setminus \bm{w}} \}^{1/2} \cdot \varepsilon.
	$
	Compared to $\hat{\tau}$, the percentage reduction in variance 
	of the probability limit of the estimated distribution of $\hat{\tau}(\tilde{\bm{\beta}}_1, \tilde{\bm{\beta}}_0)$ is 
	\begin{align*}
	1 - \frac{V_{\tau\tau}(1-R^2_{\tau, \bm{w}})
		+
		S^2_{\tau\setminus \bm{w}}}{V_{\tau\tau} + S^2_{\tau\setminus \bm{w}}}
	= 1 - \frac{
		\kappa - R^2_{\tau,\bm{w}}
	}{\kappa} =  \frac{   R^2_{\tau,\bm{w}} }{ \kappa } ,
	\end{align*}
	and the percentage reduction in length of the $1-\alpha$  quantile range of the probability limit of the estimated distribution of $\hat{\tau}(\tilde{\bm{\beta}}_1, \tilde{\bm{\beta}}_0)$ is 
	\begin{align*}
	1 - \frac{
		2q_{1-\alpha/2}(0)\left\{V_{\tau\tau}(1-R^2_{\tau, \bm{w}})
		+
		S^2_{\tau\setminus \bm{w}}\right\}^{1/2}
	}{2q_{1-\alpha/2}(0)
		\left( V_{\tau\tau} + S^2_{\tau\setminus \bm{w}} \right)^{1/2}
	}
	= 
	1 - \frac{
		\left( \kappa - R^2_{\tau,\bm{w}} \right)^{1/2}
	}{\kappa^{1/2}}
	= 1 - \left(1 - R^2_{\tau,\bm{w}}/\kappa \right)^{1/2}.
	\end{align*} 
	It is easy to show that both percentage reductions are nonnegative and nondecreasing in $R^2_{\tau,\bm{w}}$. 
\end{proof}

In the following two proofs, we recall that $q_{1-\alpha/2}(0)$ is the $(1-\alpha/2)$th quantile of a standard Gaussian distribution, and use the fact that under either the CRE or ReM,
		the $1-\alpha$ confidence interval covers the average treatment effect if and only if
		\begin{align}\label{eq:std_est_dist_optimal}
		q_{1-\alpha/2}(0) \leq
		\hat{V}^{-1/2}_{\tau\tau}( \tilde{\bm{\beta}}_1, \tilde{\bm{\beta}}_0)  \times  n^{1/2}
		\left\{\hat{\tau}(\tilde{\bm{\beta}}_1, \tilde{\bm{\beta}}_0) - \tau \right\}
		\leq q_{1-\alpha/2}(0). 
		\end{align}
Therefore, the limit of the probability that \eqref{eq:std_est_dist_optimal} holds is the asymptotic coverage probability of the confidence interval.

\begin{proof}[{\bf Proof of Corollary \ref{cor:coverage_ana_more}}]
		From 
		Corollary \ref{cor:Copt_ana_know_all} and the comment after it,
	 the probability limits of the estimated distributions of $\hat{\tau}(\tilde{\bm{\beta}}_1, \tilde{\bm{\beta}}_0)$ are the same under  the CRE and ReM, and so are the lengths of confidence intervals after being scaled by $n^{1/2}$. 

		From Lemma \ref{lemma:inf_rem}, 
		$\hat{V}_{\tau\tau}( \tilde{\bm{\beta}}_1, \tilde{\bm{\beta}}_0)$ in \eqref{eq:std_est_dist_optimal} has the same probability under the CRE and ReM.  
		From Theorem \ref{thm:analysis_more_opt} and Section \ref{sec:sopt_cre}, 
		$n^{1/2}
		\{\hat{\tau}(\tilde{\bm{\beta}}_1, \tilde{\bm{\beta}}_0) - \tau \}$ in \eqref{eq:std_est_dist_optimal}
		converges weakly to the same distribution under the CRE and ReM. 
		From Slutsky's theorem, 
		the quantity in the middle of \eqref{eq:std_est_dist_optimal} converges weakly to the same distribution under the CRE and ReM.
		Therefore, for any $\alpha\in (0,1)$, the limit of the probability that \eqref{eq:std_est_dist_optimal} holds is the same under the CRE and ReM, and so is
the asymptotic coverage probability of the $1-\alpha$ confidence interval.
\end{proof}

\begin{proof}[{\bf Proof of Corollary \ref{cor:gen_ci_larger_coverage_rem}}]
	From Theorem \ref{thm:conf_general} and Corollary \ref{corr:infer_eff_cre}, the probability limits of the estimated distributions of $\hat{\tau}(\bm{\beta}_1, \bm{\beta}_0)$ are the same under  the CRE and ReM, and so are the lengths of confidence intervals after being scaled by $n^{1/2}$. 

	Using Lemma \ref{lemma:inf_rem}, Theorem \ref{thm:adj_rem} and Slutsky's theorem, we have that under ReM, the quantity in the middle of \eqref{eq:std_est_dist_optimal} is asymptotically equal to
	\begin{eqnarray}\label{eq:general_scaled_rem}
	\tilde{V}^{-1/2}_{\tau\tau}(\bm{\beta}_1, \bm{\beta}_0)
	V_{\tau\tau}^{1/2}(\bm{\beta}_1, \bm{\beta}_0)
	\left[
	\left\{1-R^2_{\tau, \bm{x}}(\bm{\beta}_{1}, \bm{\beta}_0)\right\}^{1/2}\cdot \varepsilon + \left\{R^2_{\tau, \bm{x}}(\bm{\beta}_{1}, \bm{\beta}_0)\right\}^{1/2} \cdot  L_{K,a}\right]. 
	\end{eqnarray}
	Using Lemma  \ref{lemma:inf_rem}, Corollary \ref{corr:reg_cre} and Slutsky's theorem, we have that under the CRE, the quantity in the middle of \eqref{eq:std_est_dist_optimal} is asymptotically equal to
	\begin{eqnarray}\label{eq:general_scaled_cre}
	\tilde{V}^{-1/2}_{\tau\tau}(\bm{\beta}_1, \bm{\beta}_0) V_{\tau\tau}^{1/2}(\bm{\beta}_1, \bm{\beta}_0) \cdot \varepsilon. 
	\end{eqnarray}
	From Lemma \ref{lemma:quantile_in_rho}, the distribution \eqref{eq:general_scaled_rem} has shorter quantile ranges than \eqref{eq:general_scaled_cre}. 
	Therefore, for any $\alpha\in (0,1)$, the limit of the probability that \eqref{eq:std_est_dist_optimal} holds under ReM is larger than or equal to that under the CRE. 
\end{proof}

\section{$\hat{\tau}(\hat{\bm{\beta}}_1, \hat{\bm{\beta}}_0)$ and variance estimators under ReM}
\label{sec::hwappendix}

We need additional notation.
Let $\bm{U}_i = (1, Z_i, \bm{w}_i', Z_i \bm{w}_i')' \in \mathbb{R}^{2J +2}$. 
Let $\bar{\bm{w}}_1$ and $\bar{\bm{w}}_0$ be the averages of covariates, and $\bar{Y}_1$ and $\bar{Y}_0$ be the averages of observed outcomes in treatment and control groups. 
We can verify that in the OLS fit of $Y$ on $\bm{U}$, the coefficient of $Z$ is $\hat{\tau}(\hat{\bm{\beta}}_1, \hat{\bm{\beta}}_0)$, and the residual for unit $i$ is $\hat{e}_i =  Y_i - \hat{\bm{\beta}}_1' \bm{w}_i - (\bar{Y}_1 -  \hat{\bm{\beta}}_1' \bar{\bm{w}}_1 )$ for treated units with $Z_i = 1$ and $\hat{e}_i =Y_i - \hat{\bm{\beta}}_0' \bm{w}_i - (\bar{Y}_0 -  \hat{\bm{\beta}}_0' \bar{\bm{w}}_0 )$ for control units with $Z_i = 0.$
For $z=0,1$, let 
$
\hat{\sigma}^2_{e,z} =n_z^{-1} \sum_{i: Z_i=z} \hat{e}_i^2
$
be the average of squared residuals, and 
$\bm{m}^2_{\bm{w},z} = n_z^{-1} \sum_{i:Z_i=z}\bm{w}_i\bm{w}_i'$ be the second sample moment of $\bm{w}$. 
Define
\begin{align}
\bm{G} &  =
n^{-1}\sum_{i=1}^n \bm{U}_i \bm{U}_i'
 = 
\begin{pmatrix}
\bm{G}_{11} & \bm{G}_{12}\\
\bm{G}_{21} & \bm{G}_{22}
\end{pmatrix}   \label{eq:G}
\\
& = 
n^{-1}
\sum_{i=1}^n 
\begin{pmatrix}
1 & Z_i & \bm{w}_i' & Z_i \bm{w}_i'\\
Z_i & Z_i & Z_i \bm{w}_i' & Z_i \bm{w}_i'\\
\bm{w}_i & Z_i \bm{w}_i & \bm{w}_i \bm{w}_i' & Z_i \bm{w}_i \bm{w}_i'\\
Z_i \bm{w}_i & Z_i \bm{w}_i & Z_i \bm{w}_i\bm{w}_i' & Z_i \bm{w}_i\bm{w}_i'
\end{pmatrix}
=
\left(
\begin{array}{cc; {2pt/2pt}cc}
1 & r_1 & \bar{\bm{w}}' & r_1 \bar{\bm{w}}_1'\\
r_1 & r_1 &  r_1 \bar{\bm{w}}_1' &  r_1 \bar{\bm{w}}_1'\\ \hdashline[2pt/2pt]
\bar{\bm{w}} &  r_1 \bar{\bm{w}}_1 & \bm{S}^2_{\bm{w}} & r_1 \bm{m}^2_{\bm{w},1}\\
r_1 \bar{\bm{w}}_1 & r_1 \bar{\bm{w}}_1 & r_1 \bm{m}^2_{\bm{w},1} & r_1 \bm{m}^2_{\bm{w},1}
\end{array}
\right)  \nonumber \\
\bm{H} & = 
n^{-1}\sum_{i=1}^n \hat{e}_i^2 \bm{U}_i \bm{U}_i'
\equiv
\begin{pmatrix}
\bm{H}_{11} & \bm{H}_{12}\\
\bm{H}_{21} & \bm{H}_{22}
\end{pmatrix}  \label{eq:H}  \\
&=n^{-1}
\sum_{i=1}^{n}
\begin{pmatrix}
 \hat{e}_i^2 & Z_i  \hat{e}_i^2 &  \hat{e}_i^2 \bm{w}_i' & Z_i  \hat{e}_i^2 \bm{w}_i'\\
Z_i  \hat{e}_i^2 & Z_i  \hat{e}_i^2 & Z_i  \hat{e}_i^2 \bm{w}_i' & Z_i  \hat{e}_i^2 \bm{w}_i'\\
 \hat{e}_i^2 \bm{w}_i & Z_i  \hat{e}_i^2 \bm{w}_i &  \hat{e}_i^2 \bm{w}_i \bm{w}_i' & Z_i  \hat{e}_i^2 \bm{w}_i \bm{w}_i'\\
Z_i  \hat{e}_i^2 \bm{w}_i & Z_i \hat{e}_i^2 \bm{w}_i & Z_i \hat{e}_i^2 \bm{w}_i\bm{w}_i' & Z_i  \hat{e}_i^2 \bm{w}_i\bm{w}_i'
\end{pmatrix}
\nonumber
\\
& = 
\left(
\begin{array}{cc; {2pt/2pt}cc}
\hat{\sigma}^2_{e} & r_1 \hat{\sigma}^2_{e,1} & n^{-1}\sum_{i=1}^{n} \hat{e}_i^2 \bm{w}_i' & n^{-1}\sum_{i=1}^{n} Z_i  \hat{e}_i^2 \bm{w}_i'\\
 r_1 \hat{\sigma}^2_{e,1} &  r_1 \hat{\sigma}^2_{e,1} & n^{-1}\sum_{i=1}^{n} Z_i  \hat{e}_i^2 \bm{w}_i' & n^{-1}\sum_{i=1}^{n} Z_i  \hat{e}_i^2 \bm{w}_i'\\ \hdashline[2pt/2pt]
n^{-1}\sum_{i=1}^{n} \hat{e}_i^2 \bm{w}_i & n^{-1}\sum_{i=1}^{n}Z_i  \hat{e}_i^2 \bm{w}_i &  n^{-1}\sum_{i=1}^{n} \hat{e}_i^2 \bm{w}_i \bm{w}_i' & n^{-1}\sum_{i=1}^{n} Z_i  \hat{e}_i^2 \bm{w}_i \bm{w}_i'\\
n^{-1}\sum_{i=1}^{n} Z_i  \hat{e}_i^2 \bm{w}_i & n^{-1}\sum_{i=1}^{n} Z_i \hat{e}_i^2 \bm{w}_i & n^{-1}\sum_{i=1}^{n} Z_i \hat{e}_i^2 \bm{w}_i\bm{w}_i' & n^{-1}\sum_{i=1}^{n} Z_i  \hat{e}_i^2 \bm{w}_i\bm{w}_i'
\end{array}
\right) .
\nonumber
\end{align}

The Huber--White variance estimator for $n^{1/2} \{  \hat{\tau}(\hat{\bm{\beta}}_1, \hat{\bm{\beta}}_0) - \tau  \}$ is 
$
\hat{V}_{\text{HW}}  = \left[\bm{G}^{-1} \bm{H} \bm{G}^{-1}\right]_{(2,2)},
$
the $(2,2)$th element of $\bm{G}^{-1} \bm{H} \bm{G}^{-1}$. 

\subsection{Lemmas}

\begin{lemma}\label{lemma:tau_W}
Under ReM and Condition \ref{con:fp}, 
$
\hat{\bm{\tau}}_{\bm{w}} = r_0^{-1} \bar{\bm{w}}_1 = - r_1^{-1} \bar{\bm{w}}_0 = O_P(n^{-1/2}). 
$
\end{lemma}
\begin{proof}[\bf Proof of Lemma \ref{lemma:tau_W}]
For $1\leq j\leq J$, define pseudo potential outcomes $(\tilde{Y}(1), \tilde{Y}(0)) = (W_j, W_j)$. We can verify that Condition \ref{con:fp} also holds if we replace the original potential outcomes by the pseudo ones. Corollary \ref{corr:diff_rem} implies that $\hat{\tau}_{W_j}= O_P(n^{-1/2})$ and thus $\hat{\bm{\tau}}_{\bm{w}} = O_P(n^{-1/2})$. 
\end{proof}

\begin{lemma}\label{lemma:var_adjust_outcome_est}
	Under ReM and Condition \ref{con:fp}, for $z=0,1$, we have $\hat{\bm{\beta}}_z - \tilde{\bm{\beta}}_z = o_P(1)$, and
	\begin{align*}
	s^2_{Y(z;\hat{\bm{\beta}}_z)} - S^2_{Y(z;\tilde{\bm{\beta}}_z)} = o_P(1), \quad & s_{Y(z;\hat{\bm{\beta}}_z)\mid \bm{x}}^2 - S_{Y(z;\tilde{\bm{\beta}}_z)\mid \bm{x}}^2 = o_P(1)\\
	\bm{s}_{Y(z; \hat{\bm{\beta}}_z),\bm{x}} - \bm{S}_{Y(z; \tilde{\bm{\beta}}_z),\bm{x}} = o_P(1), \quad &
	\bm{s}_{Y(z; \hat{\bm{\beta}}_z),\bm{w}} - \bm{S}_{Y(z;\tilde{\bm{\beta}}_z),\bm{w}} = o_P(1).
	\end{align*}
\end{lemma}

\begin{proof}[\bf Proof of Lemma \ref{lemma:var_adjust_outcome_est}]
The results follow directly from 
Lemma \ref{lemma:samp_var_consistent}.
\end{proof}

\begin{lemma}\label{lemma:matrix}
For any two matrices $\bm{A}$ and $\bm{B}$, if  both $\bm{A}$ and $\bm{A}+\bm{B}$ are nonsingular, then 
\begin{align*}
 \left(\bm{A} + \bm{B}\right)^{-1} - \bm{A}^{-1} = 
 \bm{A}^{-1} \bm{B}\left(\bm{A} + \bm{B}\right)^{-1}\bm{B}\bm{A}^{-1} - 
\bm{A}^{-1} \bm{B} \bm{A}^{-1}.
\end{align*}
\end{lemma}

\begin{proof}[\bf Proof of Lemma \ref{lemma:matrix}]
Lemma \ref{lemma:matrix} is known, but we give a direct proof for completeness. 
From 
\begin{align*}
\bm{0} & = \bm{A} - \bm{A} =  \bm{A} \left(\bm{A}+\bm{B}\right)^{-1} \left(\bm{A}+\bm{B}\right) - \bm{A} = \bm{A} \left(\bm{A}+\bm{B}\right)^{-1} \bm{A} +  \bm{A} \left(\bm{A}+\bm{B}\right)^{-1} \bm{B} - \bm{A},
\\
\bm{0} & = \bm{B} - \bm{B} =   \left(\bm{A}+\bm{B}\right) \left(\bm{A}+\bm{B}\right)^{-1} \bm{B}  - \bm{B} = \bm{A} \left(\bm{A}+\bm{B}\right)^{-1} \bm{B} + \bm{B} \left(\bm{A}+\bm{B}\right)^{-1} \bm{B}  - \bm{B},
\end{align*}
we have 
$
\bm{A} \left(\bm{A}+\bm{B}\right)^{-1} \bm{A}  - \bm{A} =  \bm{B} \left(\bm{A}+\bm{B}\right)^{-1} \bm{B}  - \bm{B},
$
which further implies 
\begin{align*}
\left(\bm{A}+\bm{B}\right)^{-1} - \bm{A}^{-1} & = 
\bm{A}^{-1} \left\{
\bm{A} \left(\bm{A}+\bm{B}\right)^{-1} \bm{A}  - \bm{A}
\right\} \bm{A}^{-1} = 
\bm{A}^{-1} \left\{
\bm{B} \left(\bm{A}+\bm{B}\right)^{-1} \bm{B}  - \bm{B}
\right\} \bm{A}^{-1}\\
& = \bm{A}^{-1} \bm{B} \left(\bm{A}+\bm{B}\right)^{-1} \bm{B}  \bm{A}^{-1}
-\bm{A}^{-1} \bm{B} \bm{A}^{-1}. 
\end{align*}
\end{proof}

\begin{lemma}\label{lemma:G22}
Under ReM and Condition \ref{con:fp}, $\bm{m}_{\bm{w},z}^2 = \bm{S}_{\bm{w}}^2 + o_P(1)$ for $z=0,1$. 
Both $\bm{G}_{11}$ and $\bm{G}_{22}$ in \eqref{eq:G} converge in probability to nonsingular matrices. 
\end{lemma}

\begin{proof}[\bf Proof of Lemma \ref{lemma:G22}]
First, we consider $\bm{m}_{\bm{w},z}^2$. 
By definition, 
\begin{align*}
\bm{m}_{\bm{w},1}^2  = 
n_1^{-1} \sum_{i:Z_i=1} \bm{w}_i \bm{w}_i' = 
n_1^{-1} \sum_{i:Z_i=1} 
\left( \bm{w}_i - \bar{\bm{w}}_1 \right) \left( \bm{w}_i - \bar{\bm{w}}_1 \right)' + 
\bar{\bm{w}}_1 \bar{\bm{w}}_1' 
 = n_1^{-1} (n_1-1) \bm{s}^2_{\bm{w},1} + \bar{\bm{w}}_1 \bar{\bm{w}}_1'.
\end{align*}
From Lemmas \ref{lemma:samp_var_consistent} and \ref{lemma:tau_W}, 
$\bm{s}^2_{\bm{w},1} = \bm{S}^2_{\bm{w}} + o_P(1)$ and $\bar{\bm{w}}_1 \bar{\bm{w}}_1'= o_P(1)$. 
Thus, $\bm{m}_{\bm{w},1}^2 = \bm{S}^2_{\bm{w}} + o_P(1).$ 
Similarly, $\bm{m}_{\bm{w},0}^2 = \bm{S}^2_{\bm{w}} + o_P(1).$ 

Second, we consider $\bm{G}_{11}$. 
By definition, $\bm{G}_{11}$ has a limit as $n\rightarrow \infty$. 
Because
$$
\bm{G}_{11}^{-1} = 
\begin{pmatrix}
1 & r_1\\
r_1 & r_1
\end{pmatrix}^{-1}
= \frac{1}{r_1 r_0}
\begin{pmatrix}
r_1 & -r_1\\
-r_1 & 1
\end{pmatrix},
$$ 
the limit of $\bm{G}_{11}$ is nonsingular.

Third, we consider $\bm{G}_{22}$. From the above and by definition,  
\begin{align*}
\bm{G}_{22} = 
\begin{pmatrix}
\bm{S}^2_{\bm{w}} & r_1 \bm{m}^2_{\bm{w},1}\\
r_1 \bm{m}^2_{\bm{w},1} & r_1 \bm{m}^2_{\bm{w},1}
\end{pmatrix}
= 
\begin{pmatrix}
\bm{S}^2_{\bm{w}} & r_1 \bm{S}^2_{\bm{w}}\\
r_1 \bm{S}^2_{\bm{w}} & r_1 \bm{S}^2_{\bm{w}}
\end{pmatrix} + o_P(1).
\end{align*}
Thus, $\bm{G}_{22}$ has a probability limit as $n\rightarrow \infty$. 
Because the limit of $\bm{S}^2_{\bm{w}}$ is nonsingular, 
and 
\begin{align*}
\begin{pmatrix}
\bm{S}^2_{\bm{w}} & r_1 \bm{S}^2_{\bm{w}}\\
r_1 \bm{S}^2_{\bm{w}} & r_1 \bm{S}^2_{\bm{w}}
\end{pmatrix}^{-1}
=
( \bm{G}_{11} \otimes \bm{S}^2_{\bm{w}} )^{-1}
= \bm{G}_{11}^{-1} \otimes ( \bm{S}^2_{\bm{w}} )^{-1},
\end{align*}
the probability limit of $\bm{G}_{22}$ is nonsingular.

\end{proof}

\begin{lemma}\label{lemma:H}
Under ReM and Condition \ref{con:fp},  
\begin{itemize}
	\item[(i)] for $z=0,1$, 
	$\hat{\sigma}^2_{e,z}  = s^2_{Y(z;\hat{\bm{\beta}}_z)} + o_P(1) = S^2_{Y(z;\tilde{\bm{\beta}}_z)} + o_P(1) = O_P(1)$, 
	and 
	$\bm{H}_{11} = O_P(1)$; 
	\item[(ii)]
	$
	n^{-1}\sum_{i=1}^{n} Z_i  \hat{e}_i^2 \bm{w}_i' = o_P(n^{1/2}),
	$
	$ 
	n^{-1}\sum_{i=1}^{n} (1-Z_i)  \hat{e}_i^2 \bm{w}_i'  = o_P(n^{1/2}),
	$
	and 
	$
	\bm{H}_{12} = \bm{H}_{21}' = o_P(n^{1/2});
	$ 
	\item[(iii)]
	$
	n^{-1}\sum_{i=1}^{n} Z_i  \hat{e}_i^2 \bm{w}_i\bm{w}_i' = o_P(n),
	$
	$
	n^{-1}\sum_{i=1}^{n} (1 - Z_i ) \hat{e}_i^2 \bm{w}_i\bm{w}_i' = o_P(n),
	$
	and 
	$
	\bm{H}_{22} = o_P(n).
	$
\end{itemize}
\end{lemma}

\begin{proof}[\bf Proof of Lemma \ref{lemma:H}]
	
	First, we prove (i).  By definition and from Lemma \ref{lemma:var_adjust_outcome_est},  
	\begin{align*}
	\hat{\sigma}^2_{e,z}
	=n_z^{-1}  (  n_z-1 )  s^2_{Y(z;\hat{\bm{\beta}}_z)} 
	= 
	s^2_{Y(z;\hat{\bm{\beta}}_z)} + o_P(1) = S^2_{Y(z;\tilde{\bm{\beta}}_z)} + o_P(1) = O_P(1),\quad (z=0,1).
	\end{align*}
	This further implies $\hat{\sigma}^2_{e} = r_1 \hat{\sigma}^2_{e,1} + r_0 \hat{\sigma}^2_{e,0} = O_P(1)$. 
	Thus, $\bm{H}_{11} = O_P(1)$. 
	
	Second, we prove (ii). 
	For any $1\leq j\leq J$,  
	\begin{align*}
	\left| 
	n^{-1}\sum_{i=1}^{n} Z_i  \hat{e}_i^2 w_{ij} 
	\right|
	& \leq 
	\max_{1\leq i\leq n} |w_{ij} | \cdot
	n_1^{-1}\sum_{i=1}^{n} Z_i  \hat{e}_i^2 = 
	\max_{1\leq i\leq n} |w_{ij} | \cdot \hat{\sigma}^2_{e,1}.
	\end{align*}
	Condition \ref{con:fp} implies that $\max_{1\leq i\leq n} |w_{ij} |/n^{1/2} \rightarrow 0$ and thus $\max_{1\leq i\leq n} |w_{ij} | = o(n^{1/2})$. 
	Lemma \ref{lemma:H}(i) implies $\hat{\sigma}^2_{e,1}= O_P(1)$. 
	Thus, $n^{-1}\sum_{i=1}^{n} Z_i  \hat{e}_i^2 w_{ij}  = o(n^{1/2}) O_P(1) = o_P(n^{1/2}).$ 
	This further implies $n^{-1}\sum_{i=1}^{n} Z_i  \hat{e}_i^2 \bm{w}_i'  = o_P(n^{1/2})$. 
	Similarly, $n^{-1}\sum_{i=1}^{n} (1-Z_i)  \hat{e}_i^2 \bm{w}_i'  = o_P(n^{1/2})$. 
	By definition, 
	$$
	n^{-1}\sum_{i=1}^{n} \hat{e}_i^2 \bm{w}_i'  = n^{-1}\sum_{i=1}^{n} Z_i \hat{e}_i^2 \bm{w}_i' + n^{-1}\sum_{i=1}^{n} (1-Z_i) \hat{e}_i^2 \bm{w}_i' = o_P(n^{1/2}).
	$$
	Thus, $\bm{H}_{12} = o_P(n^{1/2})$.  
	
	Third, we prove (iii). 
	For any $1\le l,j\le J$,  
	\begin{align*}
	\left| n^{-1}\sum_{i=1}^{n} Z_i  \hat{e}_i^2 w_{il}  w_{ij}  \right| & \le 
	n^{-1}\sum_{i=1}^{n} Z_i  \hat{e}_i^2 |w_{il} | |w_{ij} |
	\le 
	\max_{1\le i\le n} |w_{il} | \cdot \max_{1\le i\le n} |w_{ij} | \cdot 
	n_1^{-1} \sum_{i=1}^{n} Z_i  \hat{e}_i^2\\
	& = \max_{1\le i\le n} |w_{il} | \cdot \max_{1\le i\le n} |w_{ij} | \cdot 
	\hat{\sigma}^2_{e,1}.
	\end{align*}
	Condition \ref{con:fp} implies $\max_{1\le i\le n} |w_{il} | = o(n^{1/2})$ and $\max_{1\le i\le n} |w_{ij} |  = o(n^{1/2})$. 
	Lemma \ref{lemma:H}(i) implies $\hat{\sigma}^2_{e,1}=O_P(1)$. 
	Thus, 
	$n^{-1}\sum_{i=1}^{n} Z_i  \hat{e}_i^2 w_{il}  w_{ij} = o(n^{1/2}) o(n^{1/2}) O_P(1) = o_P(n)$. 
	This then implies $n^{-1}\sum_{i=1}^{n} Z_i  \hat{e}_i^2 \bm{w}_i\bm{w}_i' = o_P(n)$. 
	Similarly, $n^{-1}\sum_{i=1}^{n} (1 - Z_i ) \hat{e}_i^2 \bm{w}_i\bm{w}_i' = o_P(n)$.
	Thus
	\begin{align*}
	n^{-1}\sum_{i=1}^{n} \hat{e}_i^2 \bm{w}_i \bm{w}_i' = 
	n^{-1}\sum_{i=1}^{n} Z_i \hat{e}_i^2 \bm{w}_i \bm{w}_i' + n^{-1}\sum_{i=1}^{n} (1-Z_i) \hat{e}_i^2 \bm{w}_i \bm{w}_i' = o_P(n).
	\end{align*}
	By definition, $\bm{H}_{22} = o_P(n)$. 
	
\end{proof}

\subsection{Proofs}

\begin{proof}[\bf Proof of Proposition \ref{prop::equivalent}]
We first consider the asymptotic distributions. Let $\hat{\bm{\rcoef}} = r_0 \hat{\bm{\beta}}_1 + r_1 \hat{\bm{\beta}}_0$, which satisfies  $\hat{\bm{\rcoef}} - \tilde{\bm{\rcoef}} =o_P(1)$ by Lemma \ref{lemma:var_adjust_outcome_est}. From Lemma  \ref{lemma:tau_W}, 
$\hat{\bm{\tau}}_{\bm{w}} = O_P(n^{-1/2})$.
Then
\begin{align*}
\hat{\tau}(\hat{\boldsymbol{\beta}}_1, \hat{\boldsymbol{\beta}}_0) - \hat{\tau}(\tilde{\boldsymbol{\beta}}_1, \tilde{\boldsymbol{\beta}}_0)
& = 
\left( \hat{\tau} - \hat{\bm{\rcoef}}' \hat{\bm{\tau}}_{\bm{w}} \right) - \left( \hat{\tau} - \tilde{\bm{\rcoef}}' \hat{\bm{\tau}}_{\bm{w}} \right) = \left(
\tilde{\bm{\rcoef}} - \hat{\bm{\rcoef}}
\right)' \hat{\bm{\tau}}_{\bm{w}} 
=o_P(1) O_P(n^{-1/2})
= o_P(n^{-1/2}), 
\end{align*}
which implies that $n^{1/2}\{ \hat{\tau}(\hat{\boldsymbol{\beta}}_1, \hat{\boldsymbol{\beta}}_0) - \tau\}$ has the same asymptotic distribution as $n^{1/2}\{ \hat{\tau}(\tilde{\boldsymbol{\beta}}_1, \tilde{\boldsymbol{\beta}}_0) -\tau\}$.

We then consider the probability limit of the estimated distribution.
From Lemma \ref{lemma:var_adjust_outcome_est}, 
we can show that 
$\hat{V}_{\tau\tau}(\hat{\bm{\beta}}_1, \hat{\bm{\beta}}_0) - \tilde{V}_{\tau\tau}(\tilde{\bm{\beta}}_1, \tilde{\bm{\beta}}_0) = o_P(1)$ and $
\hat{R}^2_{\tau, \bm{x}}(\hat{\bm{\beta}}_{1}, \hat{\bm{\beta}}_0) 
- 
\tilde{R}^2_{\tau, \bm{x}}(\tilde{\bm{\beta}}_{1}, \tilde{\bm{\beta}}_0) = o_P(1)
$
under ReM. 
Therefore, under ReM and Condition \ref{con:fp}
the estimated distributions of $\hat{\tau}(\hat{\boldsymbol{\beta}}_1, \hat{\boldsymbol{\beta}}_0)$ and $\hat{\tau}(\tilde{\boldsymbol{\beta}}_1, \tilde{\boldsymbol{\beta}}_0)$ have the same probability limit.  
\end{proof}

\begin{proof}[\bf Proof of Theorem \ref{thm::HWequivalent}] 
Let  
\begin{align*}
\bm{\Lambda} = 
\begin{pmatrix}
\bm{G}_{11} & \bm{0}_{2\times 2J}\\
\bm{0}_{2J\times 2} & \bm{G}_{22}
\end{pmatrix}, \quad 
\bm{\Delta} = 
\bm{G} - \bm{\Lambda} = 
\begin{pmatrix}
\bm{0}_{2 \times 2} & \bm{G}_{12}\\
\bm{G}_{21} & \bm{0}_{2J \times 2J}
\end{pmatrix}, \quad
\bm{\Psi} = \bm{G}^{-1} - \bm{\Lambda}^{-1}  .
\end{align*}

First, we first find the stochastic orders of $\bm{\Delta}$ and $\bm{\Psi}$. 
From Lemma \ref{lemma:tau_W}, $\bm{G}_{12} = \bm{G}_{21}' = O_P(n^{-1/2})$, and thus $\bm{\Delta} = O_P(n^{-1/2})$. 
From Lemma \ref{lemma:matrix},  
\begin{align*}
\bm{\Psi} & = \left(
\bm{\Lambda} + \bm{\Delta}
\right)^{-1} - \bm{\Lambda}^{-1}
= 
\bm{\Lambda}^{-1} \bm{\Delta} \left( 
\bm{\Lambda} + \bm{\Delta}
\right)^{-1} 
\bm{\Delta} \bm{\Lambda}^{-1}
- 
\bm{\Lambda}^{-1} \bm{\Delta} \bm{\Lambda}^{-1}. 
\end{align*}
From Lemma \ref{lemma:G22}, the probability limit of $\bm{\Lambda}$ exists and is nonsingular. 
Thus, $\bm{\Psi} = O_P(n^{-1/2})$. 

Second, we consider the difference between $\bm{G}^{-1} \bm{H} \bm{G}^{-1}$ and  $\bm{\Lambda}^{-1} \bm{H} \bm{\Lambda}^{-1}$: 
\begin{align}\label{eq:df_HW}
\bm{G}^{-1} \bm{H} \bm{G}^{-1} - \bm{\Lambda}^{-1} \bm{H} \bm{\Lambda}^{-1}
& = 
\left(
\bm{\Lambda}^{-1} + \bm{\Psi}
\right) \bm{H} \left(
\bm{\Lambda}^{-1} + \bm{\Psi}
\right) - \bm{\Lambda}^{-1} \bm{H} \bm{\Lambda}^{-1}
\nonumber
\\
& = 
\bm{\Psi} \bm{H} \bm{\Lambda}^{-1} + \bm{\Lambda}^{-1} \bm{H} \bm{\Psi} + 
\bm{\Psi} \bm{H} \bm{\Psi}. 
\end{align}
In \eqref{eq:df_HW}, we focus on the sub-matrix of the first two rows and the first two columns. We consider the corresponding submatrices of the
three terms in \eqref{eq:df_HW}. Below let $\left[ \cdot \right]_{(1\text{-}2, 1\text{-}2)}$ denote the sub-matrix of the first two rows and the first two columns, and $\left[ \cdot \right]_{(2,2)}$ denote the $(2,2)$th element of a matrix.
The first term in \eqref{eq:df_HW} is 
$$
\bm{\Psi} \bm{H} \bm{\Lambda}^{-1} 
 = 
\begin{pmatrix}
\bm{\Psi}_{11} & \bm{\Psi}_{12}\\
\bm{\Psi}_{21} & \bm{\Psi}_{22}
\end{pmatrix}
\begin{pmatrix}
\bm{H}_{11} & \bm{H}_{12}\\
\bm{H}_{21} & \bm{H}_{22}
\end{pmatrix}
\begin{pmatrix}
\bm{G}_{11}^{-1} & \bm{0}\\
\bm{0} & \bm{G}_{22}^{-1}
\end{pmatrix}
= 
\begin{pmatrix}
\bm{\Psi}_{11} \bm{H}_{11} \bm{G}_{11}^{-1} + \bm{\Psi}_{12}\bm{H}_{21} \bm{G}_{11}^{-1} & *\\
* & *
\end{pmatrix}.
$$
From Lemmas 
\ref{lemma:G22} and \ref{lemma:H},
$$
[\bm{\Psi} \bm{H} \bm{\Lambda}^{-1}]_{(1\text{-}2, 1\text{-}2)}  
= 
\bm{\Psi}_{11} \bm{H}_{11} \bm{G}_{11}^{-1} + \bm{\Psi}_{12}\bm{H}_{21} \bm{G}_{11}^{-1} \\
= 
O_P(n^{-1/2}) O_P(1) O_P(1) + O_P(n^{-1/2}) o_P(n^{1/2}) O_P(1) . 
$$
For the second term in \eqref{eq:df_HW}, 
$[\bm{\Lambda}^{-1} \bm{H} \bm{\Psi}]_{(1\text{-}2, 1\text{-}2)} = [\bm{\Psi} \bm{H} \bm{\Lambda}^{-1}]_{(1\text{-}2, 1\text{-}2)}' = o_P(1)$. 
For the third term in \eqref{eq:df_HW}, 
Lemma \ref{lemma:H} implies 
$\bm{H} = o_P(n)$ and thus $\bm{\Psi} \bm{H} \bm{\Psi} = O_P(n^{-1/2}) o_P(n) O_P(n^{-1/2})  = o_P(1)$. 
Therefore, from \eqref{eq:df_HW}, 
$
[\bm{G}^{-1} \bm{H} \bm{G}^{-1}]_{(1\text{-}2, 1\text{-}2)} - [\bm{\Lambda}^{-1} \bm{H} \bm{\Lambda}^{-1}]_{(1\text{-}2, 1\text{-}2)} = o_P(1),
$
which further implies 
\begin{align*}
\hat{V}_{\text{HW}} - [\bm{\Lambda}^{-1} \bm{H} \bm{\Lambda}^{-1}]_{(2,2)}
= 
[\bm{G}^{-1} \bm{H} \bm{G}^{-1}]_{(2,2)} - [\bm{\Lambda}^{-1} \bm{H} \bm{\Lambda}^{-1}]_{(2,2)} = o_P(1) .
\end{align*}

Third, we consider the difference between $[\bm{\Lambda}^{-1} \bm{H} \bm{\Lambda}^{-1}]_{(2,2)}$ and $\hat{V}_{\tau\tau}(\hat{\bm{\beta}}_1, \hat{\bm{\beta}}_0)$. 
Because
\begin{align*}
[\bm{\Lambda}^{-1} \bm{H} \bm{\Lambda}^{-1}]_{(1\text{-}2, 1\text{-}2)} & = 
\bm{G}_{11}^{-1} \bm{H}_{11} \bm{G}_{11}^{-1} 
= (r_1r_0)^{-2} 
\begin{pmatrix}
r_1 & -r_1\\
-r_1 & 1
\end{pmatrix}
\begin{pmatrix}
\hat{\sigma}^2_{e} & r_1 \hat{\sigma}^2_{e,1}\\
r_1 \hat{\sigma}^2_{e,1} & r_1 \hat{\sigma}^2_{e,1}
\end{pmatrix}
\begin{pmatrix}
r_1 & -r_1\\
-r_1 & 1
\end{pmatrix}\\
& = 
\begin{pmatrix}
r_0^{-1} \hat{\sigma}^2_{e,0} & - r_0^{-1} \hat{\sigma}^2_{e,0}\\
- r_0^{-1} \hat{\sigma}^2_{e,0} & r_1^{-1} \hat{\sigma}^2_{e,1} + r_0^{-1} \hat{\sigma}^2_{e,0}
\end{pmatrix}.
\end{align*}
using Lemma 
\ref{lemma:H}, we have 
\begin{align*}
[\bm{\Lambda}^{-1} \bm{H} \bm{\Lambda}^{-1}]_{(2,2)} & = 
r_1^{-1} \hat{\sigma}^2_{e,1} + r_0^{-1} \hat{\sigma}^2_{e,0}
= 
r_1^{-1}  \bm{s}^2_{Y(1;\hat{\bm{\beta}}_1)} + r_0^{-1}  \bm{s}^2_{Y(0;\hat{\bm{\beta}}_0)} + o_P(1). 
\end{align*}
The property of OLS implies $\bm{s}_{Y(z;\hat{\bm{\beta}}_z),\bm{w}} = \bm{0}$ for $z=0,1$. 
Using the definition in \eqref{eq:V_beta_hat}, we can then simplify $\hat{V}_{\tau\tau}(\hat{\bm{\beta}}_1, \hat{\bm{\beta}}_0)$ as 
$
\hat{V}_{\tau\tau}(\hat{\bm{\beta}}_1, \hat{\bm{\beta}}_0)  =  
r_1^{-1}s^2_{Y(1;\hat{\bm{\beta}}_1)} + r_0^{-1}s^2_{Y(0;\hat{\bm{\beta}}_0)}.
$
Therefore, 
$[\bm{\Lambda}^{-1} \bm{H} \bm{\Lambda}^{-1}]_{(2,2)} = \hat{V}_{\tau\tau}(\hat{\bm{\beta}}_1, \hat{\bm{\beta}}_0) + o_P(1)$.
From the above, we have 
$\hat{V}_{\text{HW}} = [\bm{\Lambda}^{-1} \bm{H} \bm{\Lambda}^{-1}]_{(2,2)} + o_P(1) = \hat{V}_{\tau\tau}(\hat{\bm{\beta}}_1, \hat{\bm{\beta}}_0) + o_P(1).$ 
\end{proof}

\end{document}